\documentclass[smallcondensed, numbook]{svjour3}

\usepackage{amssymb,amsmath,amsthm}
\usepackage{enumitem}
\usepackage{hyperref}
\usepackage{cleveref}
\usepackage{nameref} 
\usepackage{stmaryrd}
\usepackage{comment}
\usepackage[dvipsnames]{xcolor} 
\usepackage[all]{xy} 
\usepackage{tikz}
\usepackage{tikz-cd} 
\usetikzlibrary{calc}

\usepackage{xparse} 
\usepackage{cancel} 
\usepackage{scrextend} 
\usepackage{mathtools} 
\usepackage{diagbox} 
\usepackage{array} 
\usepackage{chngcntr}

\usepackage[ruled, vlined]{algorithm2e} 
\usepackage{graphicx, graphbox}
\graphicspath{{figures/}}

\usepackage{lineno}

\numberwithin{equation}{section}
\counterwithout{figure}{section}
\counterwithout{table}{section}
\smartqed  

\renewcommand{\paragraph}[1]{\noindent\textbf{#1}}

\crefname{equation}{equation}{equations}
\crefname{figure}{figure}{figures}
\crefname{lem}{lemma}{lemmas}

\newtheorem{thm}{Theorem}[section]

\newtheorem{lem}[thm]{Lemma}
\newtheorem{cor}[thm]{Corollary}
\newtheorem{prop}[thm]{Proposition}

\newtheorem*{mainque*}{Main question}
\newtheorem{conj}[thm]{Conjecture}

\theoremstyle{definition}
\newtheorem{defi}[thm]{Definition}

\newtheorem{ex}[thm]{Example}

\theoremstyle{remark}
\newtheorem{rk}[thm]{Remark}

\newtheorem*{pb*}{Problem}

\newtheorem*{que*}{Question}

\newtheorem{conv}[thm]{Convention}
\newtheorem{assu}[thm]{Assumption}

\newtheorem*{intui*}{Intuition}

\let\oldequation\equation
\let\oldendequation\endequation
\renewenvironment{equation}{\begin{linenomath*}\oldequation}{\oldendequation\end{linenomath*}}

\expandafter\let\expandafter\oldnequation\csname equation*\endcsname
\expandafter\let\expandafter\endoldnequation\csname endequation*\endcsname
\renewenvironment{equation*}{\begin{linenomath*}\oldnequation}{\endoldnequation\end{linenomath*}}

\def\R{{\mathbb R}}    

\renewcommand{\phi}{\varphi} 
\newcommand{\libr}{\llbracket} 
\newcommand{\ribr}{\rrbracket} 

\DeclareMathOperator{\Ima}{Im}
\DeclareMathOperator{\Ker}{Ker}

\newcommand{\pcleq}{\preccurlyeq}
\newcommand{\pcgeq}{\succcurlyeq}

\newcommand{\squarediag}[8]{
    \begin{tikzcd}[arrows=-stealth, ampersand replacement=\&, column sep=4em, row sep=3em]
        #3 \arrow[r,"#7"] \& #4 \\
        #1 \arrow[r,"#5"'] \arrow[u, "#6"] \& #2 \arrow[u,"#8"']
    \end{tikzcd}
    }
    
    \newcommand{\ses}[6][0]{
        \ifodd#1
        \begin{tikzcd}[ampersand replacement=\&]
            #2 \arrow[r,"#5"] \& #3 \arrow[r,"#6"] \& #4
        \end{tikzcd}
        \else
        \begin{tikzcd}[ampersand replacement=\&]
            0 \arrow[r] \& #2 \arrow[r,"#5"] \& #3 \arrow[r,"#6"] \& #4 \arrow[r] \& 0
        \end{tikzcd}
        \fi
        }

\newcommand{\cat}[1]{{\normalfont\mathsf{#1}}} 
\newcommand{\colim}[1][ ]{\varinjlim_{#1}}  
\renewcommand{\lim}[1][ ]{\varprojlim_{#1}}
\DeclareMathOperator{\Fun}{Fun}
\DeclareMathOperator{\Hom}{Hom}
\DeclareMathOperator{\Ran}{Ran}
\DeclareMathOperator{\Lan}{Lan}
\newcommand{\mymatrix}[1]{\begin{psmallmatrix} #1 \end{psmallmatrix}}

\def\field{\textnormal{\textbf{k}}}

\newcommand{\catper}[1]{\cat{Per}\left(#1\right)}

\def\mor{{\rho}}
\def\basemod{M}
\newcommand{\extension}[1]{\widetilde{#1}}
\DeclareMathOperator{\End}{End}

\DeclareMathOperator{\convexes}{Conv}
\DeclareMathOperator{\hullconvex}{conv} 

\newcommand{\supp}[1]{\mathrm{supp}\left(#1\right)}
\newcommand{\decclass}[1]{\langle #1 \rangle}
\def\testsubsets{\mathcal{Q}}
\def\indecsupports{\mathcal{S}}
\NewDocumentCommand{\fullsubpos}{o}{\IfNoValueTF{#1}{\textnormal{PSub}}{\textnormal{PSub}(#1)}}
\def\subgrids{\fullsubpos}
\NewDocumentCommand{\squares}{o}{\fullsubpos_2(#1)}

\NewDocumentCommand{\intervals}{o}{\IfNoValueTF{#1}{\textnormal{Int}}{\textnormal{Int}(#1)}}
\NewDocumentCommand{\rectangles}{o}{\IfNoValueTF{#1}{\textnormal{Rec}}{\textnormal{Rec}(#1)}}
\NewDocumentCommand{\blocks}{o}{\IfNoValueTF{#1}{\textnormal{Blc}}{\textnormal{Blc}(#1)}}

\NewDocumentCommand{\totallysubpos}{o}{\IfNoValueTF{#1}{\textnormal{TSub}}{\textnormal{TSub}(#1)}}

\def\sigmapos{P}

\def\hook{h}
\def\partconv{P}

\def\abscisse{X}
\def\ordinates{Y}
\def\poset{\abscisse\times\ordinates}
\def\leqabs{\leq_\abscisse}

\def\leqord{\leq_\ordinates}

\newcommand{\squarepos}[2]{\textnormal{Q}_{#1}^{#2}}

\def\omod{N}

\def\zeromod{\cat{0}}

\def\cut{c}
\def\ocut{d}
\newcommand{\tcut}[1][\cut]{\overline{#1}}
\newcommand{\bcut}[1][\cut]{\underline{#1}}
\newcommand{\lcut}[1][\cut]{\rule{0.4pt}{1ex}\hspace{1pt}#1}
\newcommand{\rcut}[1][\cut]{#1\hspace{1pt}\rule{0.4pt}{1ex}}

\def\rec{{R}}

\def\identity{\mathrm{Id}}
\NewDocumentCommand{\indimod}{d[] d<>}{%
  \IfNoValueTF{#1}{\def\rectemp{\rec}}{\def\rectemp{#1}}
  \IfNoValueTF{#2}{\field_{\rectemp}}{\field_{\rectemp,#2}}%
} 

\def\pyramid{\mathcal{P}}
\def\pyramidhigh{\includegraphics[align=c, scale=0.05]{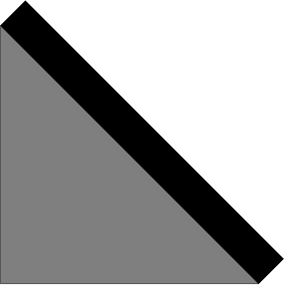}}
\def\pyramidlow{\includegraphics[align=c, scale=0.05]{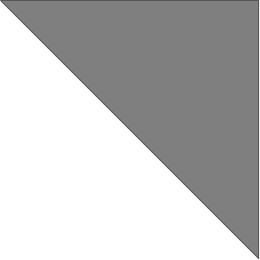}}

\NewDocumentCommand{\pyramidrecbox}{d[] m}{%
\IfNoValueTF{#1}{\includegraphics[align=c, scale=0.05]{PB#2.pdf}}%
{\includegraphics[align=c, scale=0.05]{PB#2.pdf}}}

\def\boxper{\widetilde{\basemod}}
\def\sbirth{\mathbf{bb}}
\def\sdeath{\mathbf{db}}
\def\blockcode{\mathbf{B}}

\def\maxsize{m}
\def\dual{D}

\def\letterfilt{V}
\def\letterdfilt{W}

\newcommand{\Kfilt}[3][t]{\Ker_{#2,#1}^{#3}}
\newcommand{\Ifilt}[3][t]{\Ima_{#2,#1}^{#3}}

\NewDocumentCommand{\filt}{d[] m d<>}{%
\IfNoValueTF{#1}{\def\rectemp{\rec}}{\def\rectemp{#1}}
\IfNoValueTF{#3}{\letterfilt_{\rectemp}^{#2}}{\letterfilt_{\rectemp,#3}^{#2}}%
}

\newcommand{\countingfunctor}[1][\rec]{C_{#1}}

\NewDocumentCommand{\dfilt}{d[] m d<>}{%
\IfNoValueTF{#1}{\def\rectemp{\rec}}{\def\rectemp{#1}}
\IfNoValueTF{#3}{\letterdfilt_{\rectemp}^{#2}}{\letterdfilt_{\rectemp,#3}^{#2}}%
}

\NewDocumentCommand{\filtrate}{d[] d<>}{%
  \IfNoValueTF{#1}{\def\rectemp{\rec}}{\def\rectemp{#1}}
  \IfNoValueTF{#2}{\basemod_{\rectemp}}{\basemod_{\rectemp,#2}}%
}
\def\letterdyingelts{K}
\NewDocumentCommand{\submodule}{d[] d<>}{%
  \IfNoValueTF{#1}{\def\rectemp{\rec}}{\def\rectemp{#1}}
  \IfNoValueTF{#2}{\letterdyingelts_{\rectemp}}{\letterdyingelts_{\rectemp,#2}}%
}
\newcommand{\letterlim}[1]{\mathcal{#1}}
\def\limmor{{\pi}}
\newcommand{\limfilt}[2][\rec]{\letterlim{\letterfilt}_{#1}^{#2}}

\newcommand\limrecsubmod[1][\rec]{\letterlim{\basemod}_{#1}}

\newcommand{\lifegrid}[1]{$#1$-skeleton}

\def\grid{G}
\newcommand{\gridres}[1][\grid]{\basemod^{#1}}
\newcommand{\res}[1]{\widetilde{#1}}
\def\recres{\res{\rec}}
\newcommand{\cutres}[0]{\tilde{\cut}}

\def\leftindex{-L_h}
\def\rightindex{K_h}
\def\lowerindex{-L_v}
\def\upperindex{K_v}

\def\recleftindex{-l_h}
\def\recrightindex{k_h}

\newcommand{\dart}[1]{\mathcal{D}_{#1}}

\NewDocumentCommand{\indec}{m d<>}{%
  \IfNoValueTF{#2}{\basemod^{#1}}{\basemod^{#1}_{#2}}%
}
\NewDocumentCommand{\indecgrid}{m d<>}{%
  \IfNoValueTF{#2}{\omod^{#1}}{\omod^{#1}_{#2}}%
}

\begin{document}

\title{Local characterizations for decomposability of 2-parameter persistence modules}


\author{Magnus B. Botnan \and Vadim Lebovici \and Steve Oudot}


\institute{Magnus B. Botnan \at 
            Vrije Universiteit Amsterdam \\
            Amsterdam, Netherlands\\
            \email{m.b.botnan@vu.nl}
           \and
            Vadim Lebovici \at
            Universit\'e Paris-Saclay\\
            Orsay, France\\
            \email{vadim.lebovici@ens.fr}
            \and 
            Steve Oudot \at
            Universit\'e Paris-Saclay\\
            Orsay, France\\
            \email{steve.oudot@inria.fr}
}


\maketitle

Corresponding author: Vadim Lebovici
\bigskip

\begin{abstract}
 We investigate the existence of sufficient local conditions under which poset representations decompose as direct sums of indecomposables from a given class. In our work, the indexing poset is the product of two totally ordered sets, corresponding to the setting of 2-parameter persistence in topological data analysis. Our indecomposables of interest belong to the so-called interval modules, which by definition are indicator representations of intervals in the poset. While the whole class of interval modules does not admit such a local characterization, we show that the subclass of rectangle modules does admit one and that it is, in some precise sense, the largest subclass to do so.

\end{abstract}

\tableofcontents

\section*{Conflict of interest}
The authors declare that they have no conflict of interest.

\section*{Data availability statement}
Data sharing not applicable to this article as no datasets were generated or analysed during the current study.

\section{Introduction}
\label{sec:introduction}
Recent work by Botnan and Crawley-Boevey~\cite{botnan2018decomposition} shows that pointwise finite-dimen\-sional (pfd) representations of posets over an arbitrary field~$\field$ decompose as direct sums of indecomposables with local endomorphism ring. Here we are interested in posets that are finite products of totally ordered sets~$X_1\times\cdots\times X_d$, equipped with the product order. This choice is motivated by applications in topological data analysis (TDA), where representations of such posets (typically $\R^d$) arise naturally. While a poset of this form is usually of wild representation type (at least when $d>1$), we are only interested in a subclass of its indecomposable representations, called {\em interval modules}, which by definition are indicator representations~$\field_I$ of {\em intervals} (i.e. connected convex subsets)~$I$ of $X_1\times\cdots\times X_d$. Here, connectivity and convexity are understood in the product order; see Figure~\ref{Oudotfig:interval-rep} for an example where $d=2$. 

\begin{figure}[htb]
  \centering
  \includegraphics[scale=0.5]{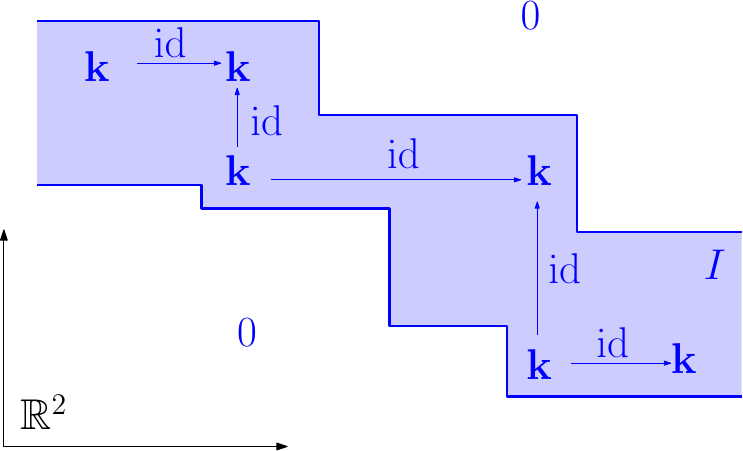}
  \caption{An interval~$I\subset\R^2$ and its associated interval module~$\field_I$.}
  \label{Oudotfig:interval-rep}
\end{figure}

These indecomposables play a key role in TDA. Indeed, given a pfd representation~$M$, the collection of the supports of the interval summands  appearing in its direct-sum decomposition can be used as a descriptor for~$M$---called its {\em barcode}---and thereby also as a descriptor for the data from which the representation originates. This descriptor is purely geometric by nature, therefore easy to interpret for practitioners, and efficient to encode and manipulate on a computer. Furthermore, its stability properties in the 1-parameter setting~\cite{edelsbrunner2008persistent,oudot2015persistence} make it a relevant choice for deriving consistent estimators in statistical analysis. 

In practice, one would like to be
able to determine whether a given representation~$M$
of~$X_1\times\cdots\times X_d$ admits interval summands in its
decomposition, and if so, whether it admits only such summands---in
which case it is called {\em interval-decomposable}.
The straightforward approach for this consists in decomposing~$M$ then checking its summands one by one.
In this paper, we advocate a different approach that consists of checking {\em local} conditions, i.e., conditions involving only restrictions of~$M$ to certain collections of subsets of~$X_1\times\cdots\times X_d$. Provided the considered subsets are small enough, the restrictions of~$M$ will have a simple structure, potentially leading to  algorithmic improvements and simplified  mathematical analyses.
%


An example of this, taken from the TDA literature, is  {\em level-set persistence}~\cite{bendich2013homology,carlsson2009zigzag}, which constructs invariants for $\R$-valued functions~$f$ on a topological space by looking at the pre-images through~$f$ of bounded open intervals of~$\R$. Once properly indexed, the homology groups of these pre-images arrange themselves into a representation of~$\R^2$ that turns out to be interval-decomposable, with summands supported on a special class of  intervals of~$\R^2$ called {\em blocks}---specifically: upper-right or lower-left quadrants, and horizontal or vertical infinite bands. Decomposability into block summands in this setting is a straightforward consequence of the Mayer-Vietoris theorem, once the following  local characterization has been established: a pfd representation~$M$ of~$\R^2$ decomposes exclusively into block summands if, and only if, all its restrictions to squares $\{x_1, x_1'\}\times \{x_2, x_2'\}\subset \R^2$ do. This fact was proven in~\cite{botnan2018decomposition,Cochoy2016}, and at the time, it provided a cleaner, and more general theory for level-set persistence than the one established previously by Bendich et al.~\cite{bendich2013homology}, which required an extra ``Morse-type'' condition on the $\R$-valued function~$f$ under consideration.

In this paper, we further generalize the theory by enlarging the class of intervals of interest to include all the axis-aligned rectangles in~$\R^2$, or more generally, all the products of 1-dimensional intervals in a product of two totally ordered sets $X_1\times X_2$ that admits a countable coinitial subset.
The indicator representations supported on rectangles are called {\em rectangle modules}, and an interval-decomposable representation whose summands are supported exclusively on rectangles is called {\em rectangle-decomposable}.
We prove (Theorem~\ref{thm:positive-rec-squares}) that a pfd representation of~$X_1\times X_2$ is rectangle-decomposable if, and only if, all its restrictions to squares $\{x_1, x_1'\}\times \{x_2, x_2'\}\subseteq X_1\times X_2$ are.
We also prove (Theorem~\ref{thm:negative-squares-hook}) that the rectangle modules are, in some precise sense, the largest subclass of interval modules that can be characterized locally, at least as far as restrictions to squares are concerned. These results generalize previous results  obtained under the constraint that the sets~$X_1, X_2$ be finite~\cite{botnan2020rectangledecomposable}---this constraint is lifted here and replaced by a much milder assumption.

The interest in rectangle modules is currently ramping up  among the TDA community, with the realization of their simplicity of use and potential for generalization. Indeed, while they constitute only a small fraction of the indecomposables over~$\R^d$, they can serve as a basis for encoding certain invariants of more general classes of representations. For instance, it was proven in~\cite{botnan2021signed} that the so-called rank invariant of any finitely presented representation~$M$ of $\R^d$, i.e., the collection of the ranks of all the internal morphisms of~$M$, decomposes essentially uniquely as the difference between the rank invariants of two rectangle-decomposable representations. Rectangle modules---or a slight generalization thereof---also appear in projective resolutions in certain exact structures~\cite{blanchette2021homological,botnan2021signed}. These facts give rise to a notion of signed barcode for general classes of representations of~$\R^d$.  While our work is not directly related to these recent developments, it contributes to the background knowledge on rectangle modules, and it allows us to answer practical questions such as determining whether a representation itself---not just its rank invariant---decomposes into rectangle summands.



The next section states our results formally, reviews the state of the art, and  provides  a brief outline of the paper.

\section{Main question and results}
\label{sec:main_results}
\subsection{Preliminaries}\label{sec:context}

Our exposition uses the language of topological data analysis to talk about representations. Here we spend a few paragraphs defining our terms. 

\subsubsection{Persistence modules}
A \emph{persistence module} over a poset $(P,\pcleq)$ is a functor $\basemod : (P,\pcleq) \rightarrow \cat{Vec}$ where $\cat{Vec}$ denotes the category of vector spaces over a fixed field~$\field$. Morphisms between persistence modules are natural transformations between functors. 

As $\cat{Vec}$ is an abelian category, so is the category $\cat{Per}(P,\pcleq) := \Fun((P,\pcleq),\cat{Vec})$ of persistence modules over $(P,\pcleq)$. More precisely, kernels, cokernels and images, as well as products, direct sums and quotients of persistence modules are defined pointwise at each index $p\in P$, and the category of persistence modules admits a zero object $\zeromod$, the persistence module whose spaces and internal morphisms are all equal to $0$.

For every $p\in P$, the vector space $\basemod(p)$ is called \emph{internal space of $M$ at $p$} and denoted for short by $\basemod_p$. For every $p\pcleq q$, the linear map of $\basemod(p\pcleq q) : \basemod_p \to \basemod_q$ is called \emph{internal morphism of $\basemod$ between $p$ and $q$} and denoted for short by $\mor_p^q$. We say that $\basemod$ is \emph{pointwise finite-dimensional} (pfd) if for every $p\in P$, $\basemod_p$ is finite-dimensional. The {\em support} of~$\basemod$ is the set of indices~$p\in P$ for which $\basemod_p\neq 0$. We say that a family $\{\basemod_\alpha\}_{\alpha\in A}$ of persistence modules over $P$ is \emph{locally finite} if the set $\{\alpha \in A \,|\, p\in \supp{\basemod_\alpha}\}$ is finite for each $p\in P$. A \emph{locally finite direct sum} is the direct sum of a locally finite family of persistence modules. 

A morphism $f : \basemod\to\omod$ between two persistence modules over $P$ is a \emph{monomorphism} (resp. \emph{epimorphism}) if for every $p\in P$, $f_p : \basemod_p \to \omod_p$ is injective (resp. surjective). A morphism between two persistence modules is an \emph{isomorphism} if is it both a monomorphism and an epimorphism. A \emph{submodule} of $\basemod$ is a persistence module $\omod$ together with a monomorphism $\omod \to \basemod$ of persistence modules, often denoted by $\omod \hookrightarrow \basemod$.

\subsubsection{Decomposability}\label{sec:decomposability}
A persistence module over $(P,\pcleq)$ is said to be \emph{decomposable} if it decomposes as direct sum of at least two nontrivial persistence modules. Otherwise, it is said to be \emph{indecomposable}. The endomorphism ring $\End(\basemod) := \Hom(\basemod, \basemod)$ is \emph{local} if $\theta$ or $\identity_\basemod-\theta$ is invertible for all $\theta \in \End(\basemod)$. It is easy to see that if $\basemod$ has a non-trivial decomposition then $\End(\basemod)$ is not local. The pfd persistence modules over $(P,\pcleq)$ form a Krull-Schmidt subcategory of~$\cat{Per}(P,\pcleq)$:
\begin{thm}[\cite{botnan2018decomposition}]
    \label{thm:decomposition-general}
    Every pfd persistence module~$\basemod$ over~$(P,\pcleq)$ decomposes as a direct sum of indecomposable modules with local endomorphism rings. By Azumaya's theorem~\cite{Azumaya1950} this decomposition is unique up to isomorphism.
\end{thm}

\subsubsection{Product posets}
\label{sec:product_posets}

In the paper we focus on persistence modules over product posets. Given two totally ordered sets $(\abscisse,\leqabs)$ and $(\ordinates,\leqord)$, their {\em product} $(\poset,\leq)$ is the Cartesian product $\poset$ equipped with the \emph{product order} $\leq$ defined by
\begin{equation}
    \forall s,t\in\poset,\, s \leq t \Longleftrightarrow s_x \leqabs t_x \text{ and } s_y \leqord t_y,
\end{equation}
where the coordinates of a point $u\in\poset$ are denoted by $(u_x,u_y)$. Henceforth we use the notation $\poset$ instead of $(\poset,\leq)$ as the only order considered on the Cartesian product $\poset$ will be the product order. 

\begin{conv}
    From now on and until the end of the paper, we fix two totally ordered sets $(\abscisse,\leqabs)$ and $(\ordinates,\leqord)$ and consider their product $(\poset,\leq)$.
\end{conv}

A persistence module~$M$ over $\poset$ is called a \emph{2-parameter persistence module}, or a \emph{persistence bimodule} for short. Any two comparable points $s\leq t$ in $\poset$ yield a square $\squarepos{s}{t}:=\{s,(s_x,t_y),(t_x,s_y),t\}$ and an associated commutative diagram
%
\begin{linenomath*}
    \begin{equation}\label{eq:commuting_square}
    \squarediag{\basemod_s}{\basemod_{(t_x,s_y)}}{\basemod_{(s_x,t_y)}}{\basemod_t}{\mor_s^{(t_x,s_y)}}{\mor_s^{(s_x,t_y)}}{\mor_{(s_x,t_y)}^t}{\mor_{(t_x,s_y)}^t}
    \end{equation}
\end{linenomath*}
%
%
\subsubsection{Intervals, rectangles, blocks}\label{sec:int-rec-blc}
Let $P$ be a poset. We say that $S\subseteq P$ is {\em convex} if, for every $p\pcleq q\in S$, we have $s\in S$ for all $s\in P$ such that $p\pcleq s \pcleq q$. A convex set~$S$ is an {\em interval} if it is also {\em connected}, i.e. for every $p,q\in S$ there is a finite sequence $(p_i)_{i\in\libr 0,n\ribr}$ of points of $S$ such that $p = p_0 \perp \dots \perp p_n = q$, where $\perp$ is the binary relation defined by $p \perp q$ if and only if $p$ and $q$ are comparable ($p \pcleq q$ or $q \pcleq p$). We write $\convexes(P)$ (resp. $\intervals[P]$) for the set of convex (resp. interval) subsets of~$P$.  

To any convex set~$S\subseteq P$  we associate a persistence module $\indimod[S]$,
called the \emph{indicator module} of~$S$, defined as follows:
\begin{equation}
    \indimod[S](p) =
    \begin{cases}
        \field &\mbox{ if } p \in S, \\
        0 &\mbox{ else},
    \end{cases}
\end{equation}
and
\begin{equation}
    \indimod[S](p\leq q) =
    \begin{cases}
        \identity_\field &\mbox{ if } p \mbox{ and } q \in S, \\
        0 &\mbox{ else},
    \end{cases}
\end{equation}
and by convention we set $\indimod[\emptyset]= \zeromod$. When~$S$ is an interval, $\indimod[S]$ is called an {\em interval module}, and it is indecomposable because its endomorphism ring~$\End(\indimod[S])$ is isomorphic to $\field$ (by connectivity of~$S$) and thus local. Otherwise, $\indimod[S]$ decomposes as the direct sum of the indicator modules of its connected components. If $\basemod$ is isomorphic to a direct sum of interval modules, then we say that $\basemod$ is \emph{interval-decomposable}. 

We call any product $I\times J$ of two intervals $I\subseteq X$ and $J\subseteq Y$ a \emph{rectangle}, and we write $\rectangles[\poset]$ for the set of rectangles of $\poset$.
Note that rectangles are intervals by definition, and their associated interval modules are called {\em rectangle modules}.

A \emph{block} is any rectangle $B = I\times J$ that satisfies either of the following (non-exclusive) conditions:
\begin{itemize}
    \item $I$ is cofinal\footnote{Given a poset $(P,\pcleq)$, a subset $Q\subseteq P$ is \emph{coinitial} if every $p\in P$ admits a $q\in Q$ such that $q \pcleq p$, and \emph{cofinal} if every $p\in P$ admits a $q\in Q$ such that $q\pcgeq p$.} in $X$ and $J$ is cofinal in $Y$---$B$ is then called a \emph{birth quadrant};
    \item $I$ is coinitial in $X$ and $J$ is coinitial in $Y$---$B$ is then called a \emph{death quadrant};
    \item $I$ is both coinitial and cofinal in $X$---$B$ is then called a \emph{horizontal band};
    \item $J$ is both coinitial and cofinal in $Y$---$B$ is then called a \emph{vertical band}.
\end{itemize}
We write $\blocks[\poset]$ for the set of blocks of $\poset$.
Blocks are rectangles by definition, and their associated rectangle modules are called \emph{block modules}. 

\begin{ex}
    Let $Q$ be a square of $\poset$ represented by the following diagram.
    \begin{equation*}
        \begin{tikzcd}
            \bullet \arrow[r]           & \bullet           \\
            \bullet \arrow[u] \arrow[r] & \bullet \arrow[u]
        \end{tikzcd}
    \end{equation*}
    $\blocks[Q]$ is the collection of the following subposets of~$Q$:
    \begin{equation*}
        \begin{tikzcd}
            \circ \arrow[r]           & \circ           \\
            \circ \arrow[u] \arrow[r] & \circ \arrow[u]
        \end{tikzcd},\qquad  
        \begin{tikzcd}
            \bullet \arrow[r]           & \bullet           \\
            \bullet \arrow[u] \arrow[r] & \bullet \arrow[u]
        \end{tikzcd}, \qquad
        \begin{tikzcd}
            \bullet \arrow[r]         & \bullet         \\
            \circ \arrow[u] \arrow[r] & \circ \arrow[u]
        \end{tikzcd}, \qquad
        \begin{tikzcd}
            \circ \arrow[r]           & \bullet           \\
            \circ \arrow[u] \arrow[r] & \bullet \arrow[u]
        \end{tikzcd},
    \end{equation*}
    \begin{equation*}
        \begin{tikzcd}
            \circ \arrow[r]           & \bullet           \\
            \circ \arrow[u] \arrow[r] & \circ \arrow[u]
        \end{tikzcd}, \qquad
        \begin{tikzcd}
            \circ \arrow[r]           & \circ           \\
            \bullet \arrow[u] \arrow[r] & \circ \arrow[u]
        \end{tikzcd}, \qquad
        \begin{tikzcd}
            \bullet \arrow[r]           & \circ           \\
            \bullet \arrow[u] \arrow[r] & \circ \arrow[u]
        \end{tikzcd}, \qquad
        \begin{tikzcd}
            \circ \arrow[r]           & \circ           \\
            \bullet \arrow[u] \arrow[r] & \bullet \arrow[u]
        \end{tikzcd},
    \end{equation*}
    where $\bullet$ represents a point that belongs to the corresponding block and $\circ$ represents a point that does not. The collection $\rectangles[Q]$ consists of the above subposets and the following two ''corners'': 
    \begin{equation*}
        \begin{tikzcd}
            \bullet \arrow[r]           & \circ           \\
            \circ \arrow[u] \arrow[r] & \circ \arrow[u]
        \end{tikzcd}, \qquad
        \begin{tikzcd}
            \circ \arrow[r]           & \circ           \\
            \circ \arrow[u] \arrow[r] & \bullet \arrow[u]
        \end{tikzcd}.
    \end{equation*}
\end{ex}

\begin{defi}A persistence bimodule~$\basemod$ is said to be
\emph{rectangle-decomposable} (\emph{block-decomposable}) if it decomposes into a direct sum of interval modules supported on rectangles (blocks). 
%
\end{defi}

\subsection{Local characterization}\label{sec:main_que}
Our aim is to work out a local condition that characterizes the decomposability of pfd persistence bimodules over a  given class of interval modules. We specify this class of interval modules via the set $\indecsupports\subseteq \intervals[\poset]$ of their supports, and we write $\decclass{\indecsupports}$ for  the set of all pfd persistence bimodules that are obtained (up to isomorphism) as direct sums of such interval modules:\begin{equation*}
  \decclass{\indecsupports} := \left\{\basemod\in \cat{Per}(\poset) \mid \basemod\simeq\bigoplus_{S\in\indecsupports}\indimod[S]^{m_{S}}\ \mbox{where}\ 
\sum_{S\ni t}m_S <\infty  \ \mbox{for all}\ t\in \poset\right\}.
\end{equation*}
Note that the class $\decclass{\indecsupports}$ is still well-defined for $\indecsupports\subseteq\convexes(\poset)$. In that case, however, $\field_S$ need not be indecomposable.

\begin{figure}[t]
    \centering
    \includegraphics[scale=0.5]{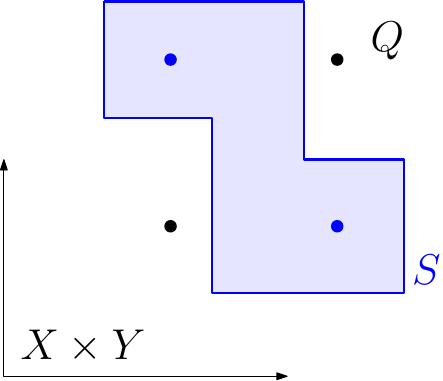}
    \caption{An example of $S\in\indecsupports$ (the solid blue polygon) and $Q\in\testsubsets$ (the four dots arranged in a square) such that $S\cap Q$ is convex but not connected in $Q$.}
    \label{fig:remark-conv-not-connected}
  \end{figure}

Locality is understood as taking restrictions to a collection~$\testsubsets$ of strict subsets of~$\poset$, called the \emph{test subsets}. Given $Q\in\testsubsets$, let $\indecsupports_{|Q}:=\{S\cap Q\,|\,S\in\indecsupports\}$ be the set of intervals in~$\indecsupports$ restricted to~$Q$. Note that $\indecsupports_{|Q}\subseteq \convexes(Q)$, since convexity is preserved under taking restrictions. However, it may not be the case that $\indecsupports_{|Q}\subseteq \intervals[Q]$, since connectivity is not always preserved under taking restrictions; see \Cref{fig:remark-conv-not-connected} for an example. Nevertheless, we still have\footnote{This is easily proven. Given an interval $S\subseteq \poset$, the restriction $S_{|Q}$ is convex in~$Q$ therefore $\field_S$ decomposes as the direct sum of the indicator modules of its connected components in~$Q$. Conversely, given an interval~$S'$ of~$Q$, the set $S=\{s\in\poset \mid p \pcleq s \pcleq q\ \mbox{for some}\ p,q\in S'\}$ is an interval of~$\poset$ that restricts to~$S'$ on $Q$.}:
\begin{equation*}
  \decclass{\intervals[\poset]_{|Q}} = \decclass{\intervals[Q]}.
\end{equation*}
While intervals of~$\poset$ may not restrict to intervals on product subsets (see again \Cref{fig:remark-conv-not-connected}), rectangles and blocks do restrict to rectangles and blocks respectively. This means that we for any $Q\in\fullsubpos[\poset]$ have:
\begin{equation*}
    \begin{split}
    \rectangles[\poset]_{|Q} &= \rectangles[Q], \\
    \blocks[\poset]_{|Q} &= \blocks[Q].
    \end{split}
\end{equation*}
%
%
\subsubsection{Problem formulation}
	Let $\indecsupports$ be a collection of intervals in $X\times Y$, and let $\testsubsets$ be a collection of test subsets. Since restriction preserves interval-decomposability, we see that if $M\in \decclass{\indecsupports}$, then $M\in \decclass{\indecsupports_{|Q}}$ for every $Q\in \testsubsets$. Symbolically, 
	\[
	        \basemod \in \decclass{\indecsupports} \Longrightarrow \forall Q\in\testsubsets,\,\basemod_{|Q} \in \decclass{\indecsupports_{|Q}}.
	        \]
In this paper we are concerned with the reverse implication. 
\begin{mainque*}
	Characterize the collections $\indecsupports$ and $\testsubsets$ such that
		\[
	        \basemod \in \decclass{\indecsupports} \Longleftarrow \forall Q\in\testsubsets,\,\basemod_{|Q} \in \decclass{\indecsupports_{|Q}}.
	        \]
\end{mainque*}
Throughout the paper we focus on the setting where the test subsets $\testsubsets$ are of the form $X'\times Y'\subseteq X\times Y$ for $X'\subseteq X$ and $Y'\subseteq Y$, and we denote by~$\fullsubpos[\poset]$ the set of such subsets.  Observe that the main question is trivial if $X\times Y$ is a member of $\testsubsets$. 

Focusing on $\fullsubpos[\poset]$ does not provide a complete picture, and there are several natural next steps; see \Cref{sec:conclusion} for a discussion on this.
%
%
%
%

\subsection{State of the art}
\label{sec:state-of-the-art}

We now review positive and negative results in the setting of our main question. To this end we need the following notation. Let $\subgrids_m(\poset)$ denotes the collection of all finite product subsets of $\poset$ of size $m\times m$. For example, the elements of~$\squares[\poset]$ are the squares, i.e., the $2\times 2$ grids embedded in $\poset$. 
\subsubsection{Positive results: $\indecsupports = \blocks[\poset]$ and $\indecsupports = \rectangles[\poset]$}
A local characterization of block-decomposable pfd bimodules was given in~\cite{Cochoy2016} for $\poset=\R^2$, and later extended to products of two totally ordered sets in~\cite{botnan2018decomposition}. 
\begin{thm}[\cite{Cochoy2016,botnan2018decomposition}]
    \label{thm:blc-decomposition}
    Let $X$ and $Y$ be totally ordered sets, and $M$ a pfd module over $\poset$. Then, $M$ is block-decomposable if and only if the restriction of $M$ to an arbitrary $2\times 2$ grid is block-decomposable. Symbolically, 
    \begin{equation*}
        \basemod\in\decclass{\blocks[\poset]} \Longleftrightarrow \forall Q\in \squares[\poset], \basemod_{|Q}\in \decclass{\blocks[Q]}.
    \end{equation*}
\end{thm}
%

For $X$ and $Y$ finite sets, the previous theorem generalizes to rectangles. 
%
%
\begin{thm}[\cite{botnan2020rectangledecomposable}]
    \label{thm:finite-case}
 Let $X$ and $Y$ be finite totally ordered sets, and $M$ a pfd module over $\poset$. Then, $M$ is rectangle-decomposable if and only if the restriction of $M$ to an arbitrary $2\times 2$ grid is rectangle-decomposable. Symbolically,  
    \begin{equation*}
        \basemod\in\decclass{\rectangles[\poset]} \Longleftrightarrow \forall Q\in \squares[\poset], \basemod_{|Q}\in \decclass{\rectangles[Q]}.
    \end{equation*}
\end{thm}
%

\subsubsection{Negative result: $\indecsupports = \intervals[\poset]$}
It was shown in~\cite{botnan2020rectangledecomposable} that when $\poset$ is finite, there is no local characterization of interval-decomposable modules. 
%
%
\begin{thm}[\cite{botnan2020rectangledecomposable}]
    \label{thm:negative-int-subgrids_finite}
    Suppose $X$ and $Y$ are finite with  $|X| \geq 3$ and $|Y| \geq 3$. Then, there exists a pfd persistence module $\basemod$ over $\poset$ that is not interval-decomposable, but for which $M|_Q$ is interval-decomposable for all $Q=X'\times Y'\subsetneq X\times Y$. 
    
\end{thm}
%

\subsubsection{Negative result: $\rectangles[\poset] \subsetneq \indecsupports \subsetneq \intervals[\poset]$}
In light of \Cref{thm:negative-int-subgrids_finite} it is natural to wonder if there is a class of intervals more general than rectangles for which a local characterization over $2\times 2$ grids is possible. That turns out not to be the case. 

%
\begin{thm}[\cite{botnan2020rectangledecomposable}]
    \label{thm:negative-squares-hook_finite}
    Suppose $X$ and $Y$ are finite with $|X| \geq 2$, $|Y| \geq 2$, where $(|X|, |Y|) \neq (2, 2)$, and let $\indecsupports \subseteq \intervals[\poset]$ be such that $\indecsupports_{|Q} \supsetneq \rectangles[Q]$ for all $2\times 2$ grids $Q$. Then, there exists a pfd persistence module $M$ over $\poset$ such that $M$ is not in $\decclass{\indecsupports}$, but the restriction of $M$ to $Q$ is in $\decclass{\indecsupports_{|Q}}$ for all $2\times 2$ grids $Q$. 
    
\end{thm}
\subsection{Our results}
The main result of this paper is a generalization of \Cref{thm:finite-case} to the product of two totally satisfying a weak condition; we conjecture that the result holds for arbitrary totally ordered sets but proving that seems to require a novel set of ideas. 

%
\begin{thm}
    \label{thm:positive-rec-squares}
    Suppose that any interval of $X$ or $Y$ admits a countable coinitial subset. Then, any pfd persistence module~$\basemod$ over $\poset$ is rectangle-decomposable if and only if the restriction of $M$ to any $2\times 2$ grid is rectangle-decomposable. Symbolically, 
    \begin{equation*}
        \basemod\in\decclass{\rectangles[\poset]} \Longleftrightarrow \forall Q\in \squares[\poset], \basemod_{|Q}\in \decclass{\rectangles[Q]}.
    \end{equation*}
\end{thm}
Note that since rectangles contain blocks, this theorem also generalizes \Cref{thm:blc-decomposition}.

The assumption on the posets is fairly mild. For instance, it is satisfied by arbitrary subsets  $X,Y$ of~$\R$ endowed with the canonical order. Furthermore, an equivalent formulation of the assumption is that~$X$ and~$Y$ both admit a countable subset which is dense in their order topology, as it is done in \cite{Crawley-Boevey2012}. It can also easily be seen to be equivalent to the hypothesis that any rectangle in $\poset$ admits a countable coinitial subset, which is instrumental when considering inverse limits of exact sequences; see \Cref{lem:iso-limfilt-filt}. 

As in the finite setting~\cite{botnan2020rectangledecomposable}, our analysis uses the following local characterization for rectangle-decomposability.
\begin{defi}[Weak exactness]\label{def:weak-exactness}
  A persistence module $\basemod$ over $\poset$ is \emph{weakly exact} if, for every $s\leq t\in \poset$, the following conditions hold in the commutative diagram~\eqref{eq:commuting_square}:
  \begin{equation*}
    \begin{split}
        \Ima\mor_s^t &= \Ima\mor_{(t_x,s_y)}^t \cap \Ima\mor_{(s_x,t_y)}^t, \\
        \Ker\mor_s^t &= \Ker\mor_s^{(t_x,s_y)} + \Ker\mor_s^{(s_x,t_y)}.
    \end{split}
    \end{equation*}
\end{defi}
This condition is a weakened version of the so-called {\em strong exactness} condition  that was proven to be equivalent to block-decomposability in~\cite{Cochoy2016}.  Furthermore, it is not hard to check that a pfd persistence module over a $2\times 2$ grid is weakly exact if and only if it is rectangle-decomposable. With this observation in mind, we see how the following theorem is equivalent to \Cref{thm:positive-rec-squares}. 

%
\begin{thm}\label{thm:rec-dec-weak-exact}
    Suppose that any interval of $X$ or $Y$ admits a countable coinitial subset. Then, a pfd persistence module~$\basemod$ over~$\poset$ is weakly exact if, and only if, $\basemod$ is rectangle-decomposable.  
\end{thm}
%

\medskip
We also consider extensions of \Cref{thm:negative-int-subgrids_finite} and \Cref{thm:negative-squares-hook_finite}. 
%
\begin{thm}[extends \Cref{thm:negative-int-subgrids_finite}]
    \label{thm:negative-int-subgrids}
    Suppose $X$ and $Y$ are totally ordered sets with $|X| \geq 3$ and $|Y| \geq 3$, and let $2 \leq m < \min(|X|,|Y|)$ be an integer. Then, there exists a pfd persistence module $\basemod$ over $\poset$ that is not interval-decomposable, but for which $M|_Q$ is interval-decomposable for all grids $Q$ of side-lengths at most $m$. 
    
       
\end{thm}
\begin{thm}[extends \Cref{thm:negative-squares-hook_finite}]
    \label{thm:negative-squares-hook}
    Suppose $X$ and $Y$ are totally ordered sets with $|X| \geq 2$, $|Y| \geq 2$, and $(|X|, |Y|) \neq (2, 2)$, and that $\indecsupports \subseteq \intervals[\poset]$ is such that $\indecsupports_{|Q} \supsetneq \rectangles[Q]$ for all $2\times 2$ grids $Q$. Then, there exists a pfd persistence module $M$ over $\poset$ such that $M$ is not in $\decclass{\indecsupports}$, but the restriction of $M$ to $Q$ is in $\decclass{\indecsupports_{|Q}}$ for all $2\times 2$ grids $Q$.    
    
\end{thm}
%

The proofs of these results adapt and extend the scheme used in the finite setting~\cite{botnan2020rectangledecomposable}, making use of embeddings of a certain indecomposable representation of the quiver $\dart{n}$. While the proof of \Cref{thm:negative-int-subgrids} is a straightforward adaptation of \Cref{thm:negative-int-subgrids_finite}, the proof of \Cref{thm:negative-squares-hook} is more involved. 

\medskip

As an application of \Cref{thm:rec-dec-weak-exact} we prove a ``continuous'' version of the so-called \emph{pyramid basis theorem} from level-sets persistence~\cite{bendich2013homology}. Here continuous refers to the fact that the persistence modules need not be completely determined by the restriction to a finite set of indices.

\subsection{Paper outline}\label{sec:paper-outline}

The proof of~\Cref{thm:rec-dec-weak-exact} is developed in Sections~\ref{sec:functorial-filtration} through~\ref{sec:covering}. Since it is easy to verify that rectangle modules are weakly exact and that being weakly exact is invariant under taking locally finite direct sums, one can easily see that any pfd persistence module that is rectangle-decomposable is also weakly exact. We are left with proving the converse statement.

For that, we follow the same approach as in~\cite{Cochoy2016} using the so-called \emph{functorial filtrations}, with some major adjustments at key steps due to the weaker notion of exactness used in our local condition (\Cref{def:weak-exactness}). Consequently, our approach in this paper is entirely different from the simple rank-based approach we followed in the finite setting~\cite{botnan2020rectangledecomposable}. 

Summarized, we start by defining for each rectangle $\rec\in\poset$, a submodule $\filtrate$ of $\basemod$ called the \emph{rectangle filtrate of $\basemod$ associated to $\rec$}. This submodule is constructed such that ~$\filtrate<t>$ contains precisely the elements of~$\basemod_t$ whose ``lifespan'' is exactly~$\rec$. In particular, $\filtrate$ is isomorphic to a finite direct sum of copies of $\indimod[\rec]$. We then prove that the sum of these filtrates is an internal direct sum of $\basemod$. The proof concludes by showing that the resulting internal direct sum generates $\basemod$. 
\medskip

Our proofs of Theorems~\ref{thm:negative-int-subgrids} and~\ref{thm:negative-squares-hook} are given in \Cref{sec:negative-answers}.
\medskip

In section \Cref{sec:example}, we discuss our results in the context of TDA. 

\subsection{Comparison to the work on block-decomposable modules}
We now compare our approach to the one in the block-decomposable case \cite{Cochoy2016}. The functorial filtration technique for weakly exact persistence bimodules already appeared in \cite[Secs.~3,4]{Cochoy2016}. However, their definition of $\filtrate$ as a submodule of $\basemod$ \cite[Prop.~5.3]{Cochoy2016} does not work for weakly exact persistence bimodules. That is, the resulting family of vector spaces does not assemble into a submodule. This poses a serious technical challenge that we overcome by defining the rectangle filtrate within a carefully constructed weakly exact submodule of $\basemod$ (\Cref{lem:submodule-is-submodule}). Proving that the resulting sum of rectangle filtrates is an internal direct sum (\Cref{prop:direct-sum}) can be adapted directly from the proof of the analogous result in the block-decomposable case \cite[Prop.~6.6]{Cochoy2016}. Contrary to this, proving that the rectangle filtrations generate $\basemod$ is considerably more involved, as the work from \cite[Sec.~7]{Cochoy2016} does not carry over to our setting. Our approach includes a series of technical lemmas that prove that it suffices to consider the restriction to a certain finite grid (\Cref{def:t-skeleton}). Once this is established, the result is a consequence of the structure theorem for finite grids (\Cref{thm:finite-case}).






\section{Functorial filtration and the counting functor}\label{sec:functorial-filtration}
In this section, we recall the definition of the \emph{functorial filtration} that we use to construct our rectangle filtrates in \Cref{sec:filtrates}. The functorial filtration is inspired by \cite{Ringel1975}, introduced in the $1$-d case by \cite{Crawley-Boevey2012} and generalized for the $2$-d case in \cite{Cochoy2016}. We write out the definitions and results already written in \cite{Cochoy2016} for completeness, and one may also read \cite[Example~3.3]{Cochoy2016} for an enlightening explicit computation of the functorial filtration. 

\begin{assu}
    \label{assu:countable-coinitial-intervals}
    In Sections~\ref{sec:functorial-filtration} through~\ref{sec:covering}, we assume that any interval of $X$ or $Y$ admits a countable coinitial subset.
\end{assu}

Consider a persistence module $\omod$ over $\poset$, and denote by~$\mor_s^t$ the internal morphism $\omod(s\leq t): \omod_s \to \omod_t$ for any $s\leq t\in \poset$.

\paragraph{Cuts.}
To set up the functorial filtration technique, we need a characterization of rectangles in $\poset$ using the notion of cuts. A cut $\cut$ of a totally ordered set $\abscisse$ is a partition of $\abscisse$ into two (possibly empty) sets $(\cut^-,\cut^+)$ such that $x < x'$ for all $x\in\cut^-$ and $x'\in\cut^+$. We call $\cut^-$ the \emph{lower part} and $\cut^+$ the \emph{upper part} of $\cut$. For instance, $\cut^- = (-\infty,1]$ and $\cut^+ = (1,+\infty)$ define a cut of $\R$. The following lemma is a direct consequence of the fact that $X$ is totally ordered.
\begin{lem}
    \label{prop:cuts-totally-ordered}
    In a totally ordered set, the set of all cuts~$\cut$ can be totally ordered in two canonical ways: inclusion on the lower part~$\cut^-$, or inclusion on the upper part~$\cut^+$. The two orders are opposite from each other.
\end{lem}

One can easily see that any interval $I$ in a totally ordered set $(T,\leq)$ can be written as $I = \lcut^+\cap\rcut^-$ for two cuts $\lcut$ and $\rcut$ of $T$ (see \cite[Sec.~3]{Crawley-Boevey2012}). Then, it follows from its definition that any rectangle $\rec$ of $\poset$ can be written as~$\rec = (\lcut^+\cap\rcut^-)\times(\bcut^+\cap\tcut^-)$, with two cuts of $\abscisse$ called the \emph{left cut} $\lcut$ and the \emph{right cut} $\rcut$, and with two cuts of $\ordinates$ called the \emph{top cut} $\tcut$ and the \emph{bottom cut} $\bcut$. Moreover, writing a block~$B = (\lcut^+\cap\rcut^-)\times(\bcut^+\cap\tcut^-)$, one can directly check from the definition of a block that:
\begin{itemize}
    \item $\rcut^+=\tcut^+=\emptyset$, if $B$ is a \emph{birth quadrant}, or 
    \item $\lcut^-=\bcut^-=\emptyset$, if $B$ is a \emph{death quadrant}, or 
    \item $\lcut^-=\rcut^+=\emptyset$, if $B$ is a \emph{horizontal band}, or 
    \item $\bcut^-=\tcut^+=\emptyset$, if $B$ is a \emph{vertical band}.
\end{itemize}
%

\paragraph{Pointwise filtration.} 
Consider a rectangle $\rec = (\lcut^+ \cap \rcut^-)\times (\bcut^+ \cap \tcut^-)$ of $\poset$ and let $t\in\rec$. We start by defining the following subspaces of $\omod_t$, called \emph{horizontal contributions} (associated to the left and the right cuts of $\rec$) and \emph{vertical contributions} (associated to the top and the bottom cuts of $\rec$) \emph{of $\rec$ in $\omod$}:
\begin{equation}\label{eqn:def-ImaKer-horizvertic}
    \begin{gathered}
      \begin{array}{lcl}
      \Ifilt{\lcut}{+}(\omod) = \bigcap\limits_{\begin{smallmatrix}x \in \lcut^+\\x \leq t_x\end{smallmatrix}} \Ima \mor_{(x, t_y)}^t
    &\quad&
    \Ifilt{\lcut}{-}(\omod) = \sum\limits_{x \in \lcut^-} \Ima \mor_{(x, t_y)}^t\\[2em]
    \Kfilt{\rcut}{+}(\omod) = \bigcap\limits_{x \in \rcut^+} \Ker \mor_t^{(x, t_y)}
  &&
    \Kfilt{\rcut}{-}(\omod) = \sum\limits_{\begin{smallmatrix}x \in \rcut^-\\x\geq t_x\end{smallmatrix}} \Ker \mor_t^{(x, t_y)}\\[2em]
    \Ifilt{\bcut}{+}(\omod) = \bigcap\limits_{\begin{smallmatrix}y \in \bcut^+\\y \leq t_y\end{smallmatrix}} \Ima \mor_{(t_x, y)}^t
    &\quad&
    \Ifilt{\bcut}{-}(\omod) = \sum\limits_{y \in \bcut^-} \Ima \mor_{(t_x, y)}^t\\[2em]
    \Kfilt{\tcut}{+}(\omod) = \bigcap\limits_{y \in \tcut^+} \Ker \mor_t^{(t_x, y)}
  &&
    \Kfilt{\tcut}{-}(\omod) = \sum\limits_{\begin{smallmatrix}y \in \tcut^-\\y\geq t_y\end{smallmatrix}} \Ker \mor_t^{(t_x, y)}
   \end{array} \end{gathered}
\end{equation}
with the convention that $\Ifilt{\cut}{-}(\omod) = 0$ when $c^- = \emptyset$ and $\Kfilt{\cut}{+}(\omod) = \omod_t$ when~$c^+ = \emptyset$. See \Cref{fig:ImaKer} for an illustration.
\begin{figure}[htb]
    \centering
    \includegraphics[scale=0.8]{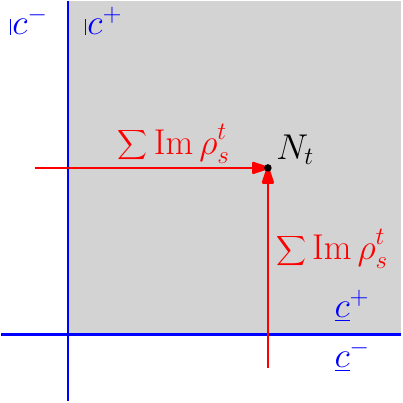}\hfill
    \includegraphics[scale=0.8]{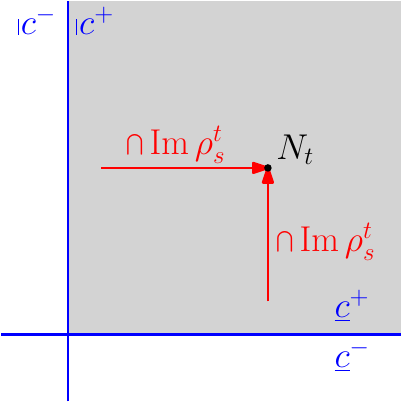}\\[5ex]
    \includegraphics[scale=0.8]{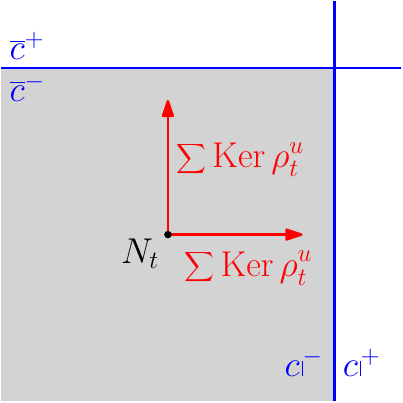}\hfill
    \includegraphics[scale=0.8]{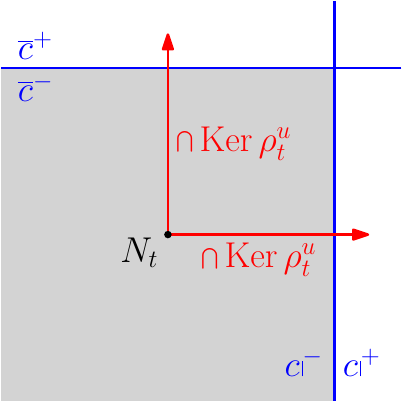}
    \caption{From top to bottom and from left to right: the spaces $\Ima^-_{\cut,t}$, $\Ima^+_{\cut,t}$, $\Ker^-_{\cut,t}$ and $\Ker^+_{\cut,t}$.}
    \label{fig:ImaKer}
\end{figure}
%
\begin{rk}
    In \cite{Cochoy2016}, the definitions and the results are stated for the poset $\poset = \R^2$, but all results hold verbatim (with identical proofs) for a general product $\poset$ under \Cref{assu:countable-coinitial-intervals}. From now on, we will therefore cite and use the results of \cite{Cochoy2016} in our setting.
\end{rk}

The following lemma states that when $\omod$ is pfd, horizontal and vertical contributions can be realized as kernels and images of its internal morphisms.
%
\begin{lem}[Realization]\cite[Lemma~3.1]{Cochoy2016}
    \label{lem:realization}
    Assume that $\omod$ is pfd. For any $t\in\rec$, one has:
    \begin{itemize}
        \item[] $\Ifilt{\lcut}{+}(\omod) = \Ima\mor_{(x,t_y)}^t$ for some $x\in\lcut^+\cap(-\infty,t_x]$ and any lower $x\in\lcut^+$,
        \item[] $\Ifilt{\lcut}{-}(\omod) = \Ima\mor_{(x,t_y)}^t$ for some $x\in\lcut^-\cup\{-\infty\}$ and any greater $x\in\lcut^-$,
        \item[] $\Kfilt{\rcut}{+}(\omod) = \Ker\mor_t^{(x,t_y)}$ for some $x\in\rcut^+\cup\{+\infty\}$ and any lower $x\in\rcut^+$,
        \item[] $\Kfilt{\rcut}{-}(\omod) = \Ker\mor_t^{(x,t_y)}$ for some $x\in\rcut^-\cap[t_x,+\infty)$ and any greater $x\in\rcut^-$,
    \end{itemize}
    with the conventions that $\mor_{(x,t_y)}^t : 0 \to \basemod_t$ when $x=-\infty$ and $\mor_t^{(x,t_y)}:\basemod_t \to 0$ when~$x=+\infty$. Similar statements hold for vertical cuts.
\end{lem}
\begin{conv}\label{conv:extension}
    Throughout the rest of the paper, we keep the conventions on internal morphisms introduced in \Cref{lem:realization} without further reference to it.
\end{conv}
\begin{rk}
    \label{rk:extension-weak-exactness}
    The conventions defined in \Cref{lem:realization} are equivalent to considering 
    the extension $\extension{\omod}$ of~$\omod$ to the poset $\extension{X}\times \extension{Y}$ with $\extension{X} = X\cup\{\pm\infty\}$ and $\extension{Y} = Y\cup\{\pm\infty\}$ (and the obvious ordering) such that $\extension{\omod}_{(\pm\infty,\cdot)} = \extension{\omod}_{(\cdot,\pm\infty)} = 0$. This extension is called \emph{extension of $\omod$ at infinity} in \cite{Cochoy2016}. In this article, we only use the following fact: if $\omod$ is pfd and weakly exact, then so is its extension at infinity. This last fact can be checked by a direct computation. 
\end{rk}
\begin{rk}\label{rk:realization}
    Let $t\in \rec$ and consider a finite number $\omod_1, \dots, \omod_k$ of pfd persistence modules over~$\poset$. Then, there exists $x\in\lcut^+\cap(-\infty,t_x]$ such that for all $1\leq i \leq k$, one has
    \begin{equation*}
        \Ifilt{\lcut}{+}(\omod_i) = \Ima\omod_i\big((x,t_y)\leq t\big).
    \end{equation*}
    Similar remarks hold for other contributions.
\end{rk}

\bigskip

We combine horizontal and vertical contributions in the following way:
\begin{equation}\label{eqn:def-ImKer-rec}
    \begin{gathered}
    \begin{array}{rcl}
    \Ifilt{\rec}{+}(\omod) &=&  \Ifilt{\lcut}{+}(\omod) \cap \Ifilt{\bcut}{+}(\omod),\\[0.5em]
    \Ifilt{\rec}{-}(\omod) &=&  \left(\Ifilt{\lcut}{-}(\omod) + \Ifilt{\bcut}{-}(\omod)\right) \cap \Ifilt{\rec}{+}(\omod),\\[0.5em]
    &=& \Ifilt{\lcut}{-}(\omod) \cap \Ifilt{\bcut}{+}(\omod) + \Ifilt{\bcut}{-}(\omod)\cap \Ifilt{\lcut}{+}(\omod),\\[0.5em]
    \Kfilt{\rec}{+}(\omod) &=&  \left(\Kfilt{\rcut}{+}(\omod) + \Kfilt{\tcut}{-}(\omod)\right) \cap \left(\Kfilt{\rcut}{-}(\omod) + \Kfilt{\tcut}{+}(\omod)\right),\\[0.5em]
    &=&  \Kfilt{\rcut}{+}(\omod)\cap \Kfilt{\tcut}{+}(\omod)   + \Kfilt{\rcut}{-}(\omod) + \Kfilt{\tcut}{-}(\omod),\\[0.5em]
    \Kfilt{\rec}{-}(\omod) &=& \Kfilt{\rcut}{-}(\omod) + \Kfilt{\tcut}{-}(\omod),
\end{array}
\end{gathered}
\end{equation}
where equalities between formulas come from the inclusions $\Ifilt{\cut}{-}(\omod) \subseteq \Ifilt{\cut}{+}(\omod)$, $\Kfilt{\cut}{-}(\omod) \subseteq \Kfilt{\cut}{+}(\omod)$ and the following elementary lemma.
\begin{lem}\label{lem:inclusion-vector-spaces}
    Let $E$ be a $\field$-vector space. Let $A$, $B$ and $C$ three vector subspaces of $E$ such that $A\subseteq C$. Then, $(A+B)\cap C = A\cap C + B \cap C$.
\end{lem}
It is immediate from the definitions that $\Ifilt{\rec}{-}(\omod) \subseteq \Ifilt{\rec}{+}(\omod)$ and $\Kfilt{\rec}{-}(\omod) \subseteq \Kfilt{\rec}{+}(\omod)$. This leads us to the following important family of vector spaces.

\begin{defi}\label{def:functorial-filtration}
The \emph{functorial filtration of $\omod$ associated to $R$} is the following pair of families of vector spaces indexed by $t\in \rec$:
\begin{equation}
    \label{eq:filt-on-rec}
    \begin{split}
        \filt{+}<t>(\omod) &= \Ifilt{\rec}{+}(\omod) \cap \Kfilt{\rec}{+}(\omod), \\
        \filt{-}<t>(\omod) &= \Ifilt{\rec}{+}(\omod) \cap \Kfilt{\rec}{-}(\omod) + \Ifilt{\rec}{-}(\omod) \cap \Kfilt{\rec}{+}(\omod).
    \end{split}
\end{equation}
\end{defi}
\noindent Since $\Ifilt{\rec}{-}(\omod) \subseteq \Ifilt{\rec}{+}(\omod)$ and $\Kfilt{\rec}{-}(\omod) \subseteq \Kfilt{\rec}{+}(\omod)$, we also have $\filt{-}<t>(\omod) \subseteq \filt{+}<t>(\omod)$. If $\omod$ is interval-decomposable, then it is a small exercise to check that the dimension of the quotient vector space $\filt{+}<t>(\omod)/\filt{-}<t>(\omod)$ equals the multiplicity of the summand $\indimod$ in the decomposition.

The following lemma shows that vectors spaces in~\eqref{eqn:def-ImKer-rec} and~\eqref{eq:filt-on-rec} are preserved by the internal morphisms of pfd and weakly exact persistence bimodules.
\begin{lem}[Transportation]\cite[Corollary~3.5, Lemma~4.1]{Cochoy2016}
    \label{lem:transportation}
    Assume that $\omod$ is pfd and weakly exact. For any $s\leq t$ in $\rec$, we have:
    \begin{equation*}
    \begin{split}
        \mor_s^t\left(\Ifilt[s]{\rec}{\pm}(\omod)\right) &= \Ifilt[t]{\rec}{\pm}(\omod), \\
        \left(\mor_s^t\right)^{-1}\left(\Kfilt[t]{\rec}{\pm}(\omod)\right) &= \Kfilt[s]{\rec}{\pm}(\omod), \\
        \mor_s^t\left(\filt{\pm}<s>(\omod)\right) &= \filt{\pm}<t>(\omod).
    \end{split}
    \end{equation*}
\end{lem}
\begin{rk}
    \label{rk:func-filt-is-system}
    \Cref{lem:transportation} ensures that if $\omod$ is pfd and weakly exact, the families~$\left(\filt{\pm}<t>(\omod)\right)_{t\in\rec}$ form systems of vector spaces.
\end{rk}
\paragraph{Counting functor.} A key object in the ``filtration technique'' is the \emph{counting functor} associated to a rectangle~$\rec$. 
\begin{defi}\cite[Section~4]{Cochoy2016}
    \emph{The counting functor} $\countingfunctor$ \emph{associated to a rectangle}~$\rec$ is defined for a pfd and weakly exact persistence module $\omod$ over $\poset$ as the inverse limit:
    \begin{equation*}
        \countingfunctor(\omod) := \lim[t\in\rec] \filt{+}<t>(\omod)/\filt{-}<t>(\omod),
    \end{equation*}
    where the transition maps are given by the naturally defined quotient maps $\bar\mor_s^t : \filt{+}<s>(\omod)/\filt{-}<s>(\omod) \rightarrow \filt{+}<t>(\omod)/\filt{-}<t>(\omod)$.
\end{defi}
The \emph{counting functor} owes its name to the following crucial fact. 
\begin{lem}\cite[Lemma~4.2]{Cochoy2016}
    \label{lem:multiplicity}
    Assume that $\omod$ is pfd and rectangle-decomposable. For any rectangle~$\rec$ of $\poset$, the multiplicity of the summand~$\indimod$ in the rectangle-decomposition of~$\omod$ is given by~$\dim\countingfunctor(\omod)$.
\end{lem}

\section{Definition of the rectangle filtrates}\label{sec:filtrates}
The goal of this section is to define rectangle filtrates (\Cref{def:rectangle-filtrate}). In this section and in \Cref{sec:directsum,sec:covering}, we consider a pfd and weakly exact persistence module $\basemod$ over $\poset$, and denote by~$\mor_s^t$ the internal morphism $\basemod(s\leq t): \basemod_s \to \basemod_t$ for any $s\leq t\in \poset$. Let~$\rec = (\lcut^+ \cap \rcut^-)\times (\bcut^+ \cap \tcut^-)$ be a rectangle of $\poset$.
%
%
%
%
%
\subsection{Elements dead above the rectangle}\label{sec:elts-dead-above-rec}
%
The rectangle filtrate associated to $\rec$ will be constructed within the submodule of~$\basemod$ defined below (\Cref{lem:submodule-is-submodule}). Consider $\rec^- = \{t\in\poset|\, \exists s \in \rec, t\leq s \}$. Note that $\rec^- = \rcut^-\times\tcut^-$, so that the contributions $\Kfilt[t]{\rcut}{+}(\basemod)$ and $\Kfilt[t]{\tcut}{+}(\basemod)$ are well-defined for any $t\in\rec^-$. 
\begin{defi}\label{lem:submodule-is-submodule}
    We call \emph{submodule of $\basemod$ of elements dead above $\rec$}, and denote by $\submodule(\basemod)$, the submodule of $\basemod$ whose spaces at each $t\in \poset$ are given by:
    \begin{equation*}
        \submodule<t>\left(\basemod\right) = 
        \begin{cases}
            \Kfilt[t]{\rcut}{+}(\basemod) \cap \Kfilt[t]{\tcut}{+}(\basemod) &\mbox{if } t\in \rec^-, \\
            0 &\mbox{else.}
        \end{cases}
    \end{equation*}
\end{defi}
The fact that $\submodule(\basemod)$ yields a well-defined persistence submodule of~$\basemod$ is an easy consequence of the definition of horizontal and vertical contributions~\eqref{eqn:def-ImaKer-horizvertic}.
%
%
When there is no ambiguity, the submodule~$\submodule\left(\basemod\right)$ is referred to as~$\submodule$ for readability. 
\begin{prop}
    \label{prop:submodule-weakly-exact}
    The persistence module~$\submodule$ is weakly exact.
\end{prop}
\begin{proof}
    In this proof we write $\tilde{\mor}_u^v = \mor_{u|\submodule<u>}^v$ for any $u\leq v \in \poset$. Let $s\leq t\in\poset$ and denote $a = (s_x,t_y)$ and $b = (t_x,s_y)$. Let us first prove the equality:
    \begin{equation}
        \label{eq:weak-exact-ker}
        \Ker\tilde{\mor}_s^t = \Ker\tilde{\mor}_s^a + \Ker\tilde{\mor}_s^b.
    \end{equation}
    Suppose that $t\notin\rec^-$. Then $\Ker\tilde{\mor}_s^t = \submodule<s>$. Moreover, $a\notin\rec^-$ or $b\notin\rec^-$, so $\Ker\tilde{\mor}_s^a = \submodule<s>$ or $\Ker\tilde{\mor}_s^b = \submodule<s>$. Hence~\eqref{eq:weak-exact-ker} in that case. Now suppose that $t\in\rec^-$. Then $s\in\rec^-$, $a\in\rec^-$ and $b\in\rec^-$. Therefore, one has $\Ker\mor_s^a \subseteq \Kfilt[s]{\tcut}{+}(\basemod)$ and $\Ker\mor_s^b \subseteq \Kfilt[s]{\rcut}{+}(\basemod)$. Using the weak exactness of $\basemod$ and \Cref{lem:inclusion-vector-spaces} twice, we get:
    \begin{equation*}
        \begin{split}
        \Ker\tilde{\mor}_s^t
        &= \left(\Ker\mor_s^t\right)\cap\submodule<s> \\
        &= \left(\Ker\mor_s^a + \Ker\mor_s^b\right)\cap\Kfilt[s]{\rcut}{+}(\basemod) \cap \Kfilt[s]{\tcut}{+}(\basemod) \\
        &= \left(\Ker\mor_s^a\cap\Kfilt[s]{\rcut}{+}(\basemod) + \Ker\mor_s^b\cap\Kfilt[s]{\rcut}{+}(\basemod)\right) \cap \Kfilt[s]{\tcut}{+}(\basemod) \\
        &= \Big(\Ker\mor_s^a\Big)\cap\submodule<s> + \left(\Ker\mor_s^b\right)\cap\submodule<s>\\
        &= \Ker\tilde{\mor}_s^a + \Ker\tilde{\mor}_s^b.
        \end{split}
    \end{equation*}

    Let us now prove the equality~$\Ima\tilde{\mor}_s^t = \Ima\tilde{\mor}_a^t \cap \Ima\tilde{\mor}_b^t$, i.e.
    \begin{equation}
        \label{eq:weak-exact-ima}
        \mor_s^t\left(\submodule<s>\right) = \mor_a^t\left(\submodule<a>\right) \cap \mor_b^t\left(\submodule<b>\right).
    \end{equation}
    The inclusion $\mor_s^t\left(\submodule<s>\right) \subseteq \mor_a^t\left(\submodule<a>\right) \cap \mor_b^t\left(\submodule<b>\right)$ is clear. Let us show the converse. Applying \Cref{lem:realization} at the point $s$ and at the point $a$, one can choose low enough~$y\in\tcut^+\cup\{+\infty\}$ such that:
    \begin{equation*}
        \begin{split}
            \Kfilt[s]{\tcut}{+}(\basemod) &= \Ker\mor_s^{(s_x, y)},\\
            \Kfilt[a]{\tcut}{+}(\basemod) &= \Ker\mor_a^{(s_x, y)}.
        \end{split}
    \end{equation*}
    Similarly, there exists $x\in\rcut^+\cup\{+\infty\}$ such that:
    \begin{equation*}
        \Kfilt[b]{\rcut}{+}(\basemod) = \Ker\mor_b^{(x, s_y)}, \quad \text{and} \quad \Kfilt[s]{\rcut}{+}(\basemod) = \Ker\mor_s^{(x, s_y)}.
    \end{equation*}
    The result will follow from repeated use of the weak exactness property of $\basemod$ (and \Cref{rk:extension-weak-exactness}). The following diagram will help picturing the various spaces involved in this proof. Denote~$c=(s_x, y)$,~$d = (x, s_y)$, $e=(t_x, y)$ and $f = (x, t_y)$.
    \begin{equation*}
        \begin{tikzcd}
            {\basemod_{c}} \arrow[r]           & {\basemod_{e}}         &                     \\
            {\basemod_{a}} \arrow[u] \arrow[r] & \basemod_{t} \arrow[u] \arrow[r]         & {\basemod_{f}} \\
            {\basemod_{s}} \arrow[u] \arrow[r]   & {\basemod_{b}} \arrow[r] \arrow[u] & {\basemod_{d}} \arrow[u]
            \end{tikzcd}
    \end{equation*}
    Let $z\in\submodule<t>$ be such that there are $z_a\in\submodule<a>$ and $z_b\in\submodule<b>$ such that $z = \mor_a^t(z_a) = \mor_b^t(z_b)$. Since $\basemod$ is weakly exact, there exists $z_s\in\basemod_s$ such that $z = \mor_s^t(z_s)$. 
    
    We claim that $z_s \in \Ker\mor_s^e$. Indeed, using that $z_a \in \Kfilt[a]{\tcut}{+}(\basemod) = \Ker\mor_a^c$, we get~$\mor_s^e(z_s) = \mor_t^e(z) = \mor_a^e(z_a) = \mor_c^e\circ \mor_a^c(z_a) = 0$. 

    Thus, by weak exactness of $\basemod$ and \Cref{rk:extension-weak-exactness}, there exist $z'\in\Ker\mor_s^b$ and $z''\in\Ker\mor_s^c$ such that $z_s = z' + z''$. This yields that $z = \mor_s^t(z_s) = \mor_s^t(z'')$. 

    We claim now that $z''\in\Ker\mor_s^f$. Indeed, using that $z_b \in \Kfilt[b]{\rcut}{+}(\basemod) = \Ker\mor_b^d$, we get~$\mor_s^f(z'') = \mor_t^f(z) = \mor_b^f(z_b) = \mor_d^f\circ \mor_b^d(z_b) = 0$.

    Thus, by weak exactness of $\basemod$ and \Cref{rk:extension-weak-exactness}, there exist $\tilde{z}'\in\Ker\mor_s^a$ and $\tilde{z}''\in\Ker\mor_s^d$ such that $z'' = \tilde{z}' + \tilde{z}''$. This yields that $z = \mor_s^t(z'') = \mor_s^t(\tilde{z}'')$. 
    
    We now claim that~$\tilde{z}'' \in \submodule<s>$, which completes the proof. Indeed, on the one hand $\tilde{z}''\in\Ker\mor_s^d = \Kfilt[s]{\rcut}{+}(\basemod)$. On the other, one has~$\tilde{z}'' = z''- \tilde{z}' \in\Ker\mor_s^c = \Kfilt[s]{\tcut}{+}(\basemod)$. Hence the result.

\end{proof}

\subsection{Rectangle filtrate}
Let $\omod$ be a pfd and weakly exact persistence module over $\poset$. Recall the definition of the functorial filtration (\Cref{def:functorial-filtration}) and the systems of vector spaces they form (\Cref{rk:func-filt-is-system}). As in \cite{Crawley-Boevey2012} and \cite{Cochoy2016}, we consider the following inverse limits:
\begin{equation}
    \limfilt{\pm}(\omod) = \lim[t\in\rec]\filt{\pm}<t>(\omod). 
\end{equation}
Note that, denoting $\limmor_t:\limfilt{+}(\omod) \to \filt{+}<t>(\omod)$ the natural map given by the limit, one has the following identification:
\begin{equation}\label{eq:limfilt-inclusion}
    \limfilt{-}(\omod) = \bigcap_{t\in\rec}\limmor_t^{-1}\left(\filt{-}<t>(\omod)\right)\subseteq \limfilt{+}(\omod).
\end{equation}
This implies that for any $t\in\rec$, the morphism $\limmor_t$ induces a morphism:
\begin{equation*}
    \bar{\limmor}_t : \limfilt{+}(\omod)/\limfilt{-}(\omod) \longrightarrow \filt{+}<t>(\omod) /  \filt{-}<t>(\omod).
\end{equation*}
\begin{lem}[\textnormal{\cite[Lem.~5.2]{Cochoy2016}}]
    \label{lem:iso-limfilt-filt}
    Recall that $\omod$ is pfd and weakly exact. For $t\in\rec$, the map $\bar{\limmor}_t : \limfilt{+}(\omod)/\limfilt{-}(\omod) \longrightarrow \filt{+}<t>(\omod) /  \filt{-}<t>(\omod)$ is an isomorphism.
\end{lem}
The rectangle filtrate of $\basemod$ associated to $\rec$ will be defined as the submodule of $\basemod$ given by the following proposition (\Cref{def:rectangle-filtrate}).
\begin{prop}
    \label{prop:recsubmod}
        Let $\limrecsubmod$ be a vector space complement of $\limfilt{-}(\submodule)$ in $\limfilt{+}(\submodule)$. For~$t\in\poset$, consider the vector subspace of $\submodule<t>$ given by:
    \begin{equation*}
        \filtrate<t> :=
        \begin{cases}
            \limmor_t(\limrecsubmod) &\mbox{ if } t\in R, \\
            0 &\mbox{ else,}
        \end{cases}
    \end{equation*}
    where $\limmor_t: \limfilt{+}(\submodule) \longrightarrow \filt{+}<t>(\submodule)$ is the natural maps given by the limit for $t\in\rec$. Then, the family $(\filtrate<t>)_{t\in\poset}$ forms a submodule of $\submodule$ (hence of $\basemod$).
\end{prop}
\begin{proof}
    Let $s\leq t$ in $\poset$. Suppose that $s\not\in \rec$. Then, one has $\mor_s^t(\filtrate<s>) = 0 \subseteq \filtrate<t>$.
    Now, suppose that $s$ and $t$ both lie in $\rec$. By definition of $\limmor$, one has $\mor_s^t \circ \limmor_s = \limmor_t$. Thus, one has $\mor_s^t(\filtrate<s>) = \mor_s^t(\limmor_s(\limrecsubmod))= \limmor_t(\limrecsubmod) = \filtrate<t>$. 
    
    Finally, suppose that $s\in \rec$ and $t\not\in \rec$. We show that $\mor_s^t(\filtrate<s>) = 0$. One has $\limmor_s(\limrecsubmod)\subseteq \Kfilt[s]{\rcut}{+} \cap \Kfilt[s]{\tcut}{+}$, for every $s\in \rec$. Moreover, for every $t \geq s$ with $s\in\rec$ and $t\not\in \rec$, we have $t_x\in \rcut^+$ or $t_y\in\tcut^+$, so that $\mor_s^t(\Kfilt[s]{\rcut}{+} \cap \Kfilt[s]{\tcut}{+}) = 0$. Hence:
    \begin{equation*}
        \mor_s^t\left(\filtrate<s>\right) = \mor_s^t(\limmor_s(\limrecsubmod)) \subseteq \mor_s^t(\Kfilt[s]{\rcut}{+} \cap \Kfilt[s]{\tcut}{+}) = 0.
    \end{equation*}
\end{proof}
%
%
\begin{rk}\label{rk:filtrate-complement-filt-minus}
    Since we have chosen $\limrecsubmod$ such that $\limfilt{+}(\submodule) = \limrecsubmod\oplus\limfilt{-}(\submodule)$, for every $t\in\rec$ we have 
    $\filt{+}<t>(\submodule) = \filtrate<t> \oplus \filt{-}<t>(\submodule)$ by \Cref{lem:iso-limfilt-filt}.
\end{rk}
\begin{defi}[Rectangle filtrate]
    \label{def:rectangle-filtrate}
    Let $\limrecsubmod$ be a vector space complement of $\limfilt{-}(\submodule)$ in $\limfilt{+}(\submodule)$. The submodule of $\basemod$ defined in \Cref{prop:recsubmod} is called a \emph{rectangle filtrate of $\basemod$ associated to $\rec$} and denoted by $\filtrate$.
\end{defi}
%
Our work on rectangle filtrates will rely on \Cref{prop:recsubmod} and \Cref{rk:filtrate-complement-filt-minus} and thus will not depend on the choice of vector space complement $\limrecsubmod$ of $\limfilt{-}(\submodule)$ in $\limfilt{+}(\submodule)$. The following convention will therefore be used.
\begin{conv}\label{conv:choice-of-complement}
    For each rectangle~$\rec'$ of $\poset$, choose a vector space complement $\limrecsubmod[\rec']$ of $\limfilt[\rec']{-}(\submodule)$ in $\limfilt[\rec']{+}(\submodule)$. From now on, the rectangle filtrate $\filtrate[\rec']$ will refer to the one associated to this choice of $\limrecsubmod[\rec']$. 
\end{conv}
Note that the axiom of choice is used in the above convention. This is inevitable in order to consider infinite families of rectangle filtrates in \Cref{sec:directsum,sec:covering}. The following lemma shows that~$\filtrate$ is rectangle-decomposable.
\begin{lem}
    \label{lem:recsubmodisisotorecmod}
    The persistence module $\filtrate$ is isomorphic to a direct sum of $\dim \countingfunctor(\submodule)$ copies of the rectangle module $\indimod[R]$.
\end{lem}
\begin{proof}
    The proof is a carbon copy of the one of \cite[Lemma~5.5]{Cochoy2016} and is included here for completeness. Since $\submodule$ is pfd and weakly exact (\Cref{prop:submodule-weakly-exact}), we know from \Cref{lem:iso-limfilt-filt} that the morphism $\bar{\limmor}_t$ is an isomorphism for any $t\in\rec$.

    Let~$\Gamma$ be a (finite) basis of~$\limrecsubmod$. For any~$\gamma\in\Gamma$, the elements of~$\limmor_t(\gamma)$ for~$t\in\rec$ are non-zero and they satisfy~$\mor_s^t(\limmor_s(\gamma)) = \limmor_t(\gamma)$ for all~$s\leq t$ in~$\rec$, so they span a submodule~$N(\gamma)$ of~$\filtrate$ that is isomorphic to~$\indimod[\rec]$. Now, for all~$t\in\rec$ the family~$\{\mor_t(\gamma)\}_{\gamma\in\Gamma}$ is a basis of~$\filtrate<t>$, so~$\filtrate \simeq \bigoplus_{\gamma\in\Gamma} N(\gamma)$. Finally, the size of the basis~$\Gamma$ is~$\dim\limrecsubmod = \dim \countingfunctor(\submodule)$.

\end{proof}


\subsection{Description of rectangle filtrates in the finitely rectangle-decomposable case}
We prove the following lemma showing that rectangle filtrates capture rectangle summands in the particular case of finite rectangle-decompositions. It will be crucial in the proof of \Cref{prop:cover} in \Cref{sec:covering}. 
\begin{lem}
    \label{lem:equality-dim-CF}
    If $\basemod\simeq \bigoplus_{i\in I} \indimod[\rec_i]^{m_i}$ where $I$ is a finite set, the $\rec_i$'s are pairwise distinct rectangles of $\poset$ and the $m_i$'s are positive integers, then one has~$\filtrate[\rec_i] \simeq \indimod[\rec_i]^{m_i}$, for any $i\in I$.
\end{lem}
In fact, this lemma also holds when the rectangle decomposition is only locally finite but we do not use such a general statement in the paper. The proof of \Cref{lem:equality-dim-CF} is postponed to the end of this section. It uses the following lemma that gives an expression of $\submodule\left(\basemod\right)$ when $\basemod$ is decomposable as a finite direct sum of rectangle modules. Recall the definition of $\rec^-$ from \Cref{sec:elts-dead-above-rec}.
\begin{lem}
    \label{lem:submodule-decomposable-case}
    Suppose that $\basemod\simeq \bigoplus_{i\in I} \indimod[\rec_i]$ where $I$ is a finite set and the $\rec_i$'s are rectangles of $\poset$. Then, one has:
    \begin{equation*}
        \submodule[\rec]\left(\basemod\right) \simeq \bigoplus_{\begin{smallmatrix}i\in I\\\rec_i \subseteq \rec^-\end{smallmatrix}} \indimod[\rec_i].
    \end{equation*}
\end{lem}
\Cref{lem:submodule-decomposable-case} is a direct consequence of \Cref{lem:capdirectsum,lem:submodule-additive,lem:submodule-recmod} below. The first lemma is elementary. It is used in the proof of \Cref{lem:submodule-additive}.
\begin{lem}
  \label{lem:capdirectsum}
  Let~$E$ be a~$\field$-vector space and~$E_1,\,E_2$ be two subspaces of~$E$ such that~$E = E_1 \oplus E_2$. Let~$A_1,B_1$ two subspaces of~$E_1$, and~$A_2,\, B_2$ two subspaces of~$E_2$. Then,
  \begin{equation*}
      (A_1 \oplus A_2)\cap(B_1 \oplus B_2) = (A_1\cap B_1) \oplus (A_2 \cap B_2).
  \end{equation*}
\end{lem}
\begin{lem}
    \label{lem:submodule-additive}
    Let $\omod_1$ and $\omod_2$ be two pfd and weakly exact persistence modules over~$\poset$. One has $\submodule(\omod_1\oplus\omod_2) = \submodule(\omod_1)\oplus\submodule(\omod_2)$.
\end{lem}
\begin{proof}
    The module $\omod_1$ can naturally be seen as a submodule of $\omod_1\oplus\omod_2$ by identification with $\omod_1\oplus\zeromod$. 
    Similarly, any submodule of $\omod_2$ can naturally be seen as a submodule of $\omod_1\oplus\omod_2$. Therefore, $\submodule(\omod_1)$ and $\submodule(\omod_2)$ can as well be naturally seen as submodules of $\omod_1\oplus\omod_2$. We implicitly make these identifications in the rest of the proof. To prove the result, it is thus sufficient to prove the equality $\submodule(\omod_1\oplus\omod_2) = \submodule(\omod_1) \oplus \submodule(\omod_2)$ of persistence submodules of $\omod_1\oplus\omod_2$. In other words, it is sufficient to prove that for any $t\in\poset$, one has:
    \begin{equation*}
        \submodule<t>(\omod_1\oplus \omod_2) = \submodule<t>(\omod_1) \oplus \submodule<t>(\omod_2).
    \end{equation*}
    Let $t\in\poset$. For $t\notin\rec^-$, both sides of the equality vanish, so let us assume that~$t\in\rec^-$. Denote by $\rho$ the internal morphisms of $\omod_1$ and by $\eta$ the internal morphisms of $\omod_2$. The internal morphisms of $\omod_1\oplus\omod_2$ are then given by $\rho\oplus\eta$. By \Cref{rk:realization}, there exists $x\in\rcut^+\cup\{+\infty\}$ such that:
    \begin{equation*}
        \begin{split}
            \Kfilt{\rcut}{+}(\omod_1) &= \Ker\rho_t^{(x,t_y)},\\
            \Kfilt{\rcut}{+}(\omod_2) &= \Ker\eta_t^{(x,t_y)},\\
            \Kfilt{\rcut}{+}(\omod_1\oplus \omod_2) &= \Ker\left(\rho_t^{(x,t_y)}\oplus\eta_t^{(x,t_y)}\right) = \left(\Ker\rho_t^{(x,t_y)}\right) \oplus \left(\Ker\eta_t^{(x,t_y)}\right).
        \end{split}
    \end{equation*}
    Hence, one has $\Kfilt{\rcut}{+}(\omod_1\oplus \omod_2) = \Kfilt{\rcut}{+}(\omod_1)\oplus\Kfilt{\rcut}{+}(\omod_2)$. Similarly, one can prove that $\Kfilt{\tcut}{+}(\omod_1\oplus \omod_2) = \Kfilt{\tcut}{+}(\omod_1)\oplus\Kfilt{\tcut}{+}(\omod_2)$. This implies:
    \begin{equation*}
        \begin{split}    
            \submodule<t>(\omod_1\oplus\omod_2) 
            &\hspace{0.7em} = \hspace{0.7em} \Kfilt{\rcut}{+}(\omod_1\oplus \omod_2) \cap \Kfilt{\tcut}{+}(\omod_1\oplus \omod_2) \\
            &\hspace{0.7em} = \hspace{0.7em} \left(\Kfilt{\rcut}{+}(\omod_1)\oplus\Kfilt{\rcut}{+}(\omod_2)\right) \cap \left(\Kfilt{\tcut}{+}(\omod_1)\oplus\Kfilt{\tcut}{+}(\omod_2)\right) \\
            &\hspace{-1em}\overset{\textnormal{(Lem.~\ref{lem:capdirectsum})}}{=} \left(\Kfilt{\rcut}{+}(\omod_1)\cap\Kfilt{\tcut}{+}(\omod_1)\right)\oplus\left(\Kfilt{\rcut}{+}(\omod_2)\cap \Kfilt{\tcut}{+}(\omod_2)\right) \\
            &\hspace{0.7em} = \hspace{0.7em} \submodule<t>(\omod_1) \oplus \submodule<t>(\omod_2).
        \end{split}
    \end{equation*}

\end{proof}
\begin{lem}
    \label{lem:submodule-recmod}
    Let $\rec$ and $\rec'$ be two rectangles of $\poset$. Then,
    \begin{equation*}
        \submodule\left(\indimod[\rec']\right) =
                    \begin{cases}
                        \indimod[\rec'] & \mbox{ if } \rec'\subseteq \rec^-, \\
                        0 & \mbox{ else.}
                    \end{cases}
    \end{equation*}
\end{lem}
\begin{proof}
    Write $\rec = (\lcut^+\cap \rcut^-)\times(\bcut^+\cap\tcut^-)$ and $\rec' = (\lcut[d]^+\cap \rcut[d]^-)\times(\bcut[d]^+\cap\tcut[d]^-)$. The persistence module $\submodule\left(\indimod[\rec']\right)$ is a submodule of $\indimod[\rec']$, so we only have to check that for any $t\in\poset$, one has $\submodule<t>\left(\indimod[\rec']\right) = \field$ if $t\in\rec'$ and $\rec'\subseteq \rec^-$, and $\submodule<t>\left(\indimod[\rec']\right) = 0$ otherwise. Let $t\in\poset$. One has $\submodule<t>\left(\indimod[\rec']\right) \subseteq \indimod[\rec']<t> = 0$ for any $t\not\in\rec'$, so let us suppose that $t\in\rec'$. 
    
    Suppose that $\rec'\subseteq \rec^-$. For any $x\in\rcut^+$, one has $x\in\rcut[d]^+$, so $(x,t_y)\not\in\rec'$ and $\indimod[\rec'](t\leq (x,t_y)) = 0$. Thus, one has $\Kfilt{\rcut}{+}(\indimod[\rec']) = \indimod[\rec']<t>$. Similarly, one has $\Kfilt{\tcut}{+}(\indimod[\rec']) = \indimod[\rec']<t>$, and hence $\submodule<t>\left(\indimod[\rec']\right) = \indimod[\rec']<t> = \field$. 

    Now, suppose that~$\rec' \not\subseteq \rec^-$. There exists~$x\in \rcut[d]^-\setminus \rcut^-$ or~$y\in\tcut[d]^-\setminus \tcut^-$. Say there exists~$x\in \rcut[d]^-\setminus \rcut^-$, the other case being similar. Suppose first that~$t\not\in\rec^-$. Then, one has~$\submodule<t>\left(\indimod[\rec']\right) = 0$ by definition of~$\submodule$. Now, suppose that~$t\in\rec^-$. Since~$x\not\in\rcut^-$, one has~$x\geq t_x$. Moreover, since~$t\in\rec'$, one has~$t_x \in\lcut[d]^+$, thus~$x\in\lcut[d]^+$. Therefore, one has~$x\in\lcut[d]^+\cap\rcut[d]^-$ and~$(x,t_y)\in\rec'$. Hence~$\indimod[\rec'](t\leq (x,t_y))$ is an isomorphism and~$\Kfilt{\rcut}{+}(\indimod[\rec'])\subseteq \Ker \indimod[\rec'](t\leq (x,t_y)) = 0$. Therefore, one has~$\submodule<t>\left(\indimod[\rec']\right) = 0$.

\end{proof}

\begin{rk}\label{rk:multiplicities-submodule-basemod}
    Consider the setting of \Cref{lem:submodule-decomposable-case}. Since $\rec\subseteq \rec^-$, \Cref{lem:submodule-decomposable-case} implies that the multiplicities of the summand $\indimod[\rec]$ in the rectangle-decompositions of $\submodule(\basemod)$ and of $\basemod$ are the same.
\end{rk}
\begin{proof}[Proof of \Cref{lem:equality-dim-CF}]
    Let $i\in I$. \Cref{lem:recsubmodisisotorecmod} yields:
    \begin{equation*}
        \filtrate[\rec_i] \simeq \indimod[\rec_i]^{\dim\countingfunctor[\rec_i]\left(\submodule[\rec_i]\left(\basemod\right)\right)}.
    \end{equation*}
    Moreover, \Cref{lem:submodule-decomposable-case} ensures that $\submodule[\rec_i]\left(\basemod\right)$ is rectangle-decomposable. Then, \Cref{lem:multiplicity} ensures that $\dim\countingfunctor[\rec_i]\left(\submodule[\rec_i]\left(\basemod\right)\right)$ is equal to the multiplicity of the summand $\indimod[\rec_i]$ in the rectangle-decomposition of $\submodule[\rec_i]\left(\basemod\right)$. In fact, \Cref{rk:multiplicities-submodule-basemod} implies that this multiplicity is the same as the multiplicity of $\indimod[\rec_i]$ in the rectangle-decomposition of $\basemod$, which is $m_i$. Hence the result.

\end{proof}

\section{The sum of rectangle filtrates is a direct sum}\label{sec:directsum}
Recall that $\basemod$ is a pfd and weakly exact persistence bimodule and that its internal morphisms are denoted by~$\mor_s^t: \basemod_s \to \basemod_t$ for any $s\leq t\in \poset$. Recall also our choice of rectangle filtrates; \Cref{conv:choice-of-complement}. In this section, we prove that the sum of $\filtrate$ for~$\rec$ ranging over all rectangles in $\poset$, is an internal direct sum of $\basemod$.
\begin{prop}
    \label{prop:direct-sum}
    The sum of $\left(\filtrate\right)_{\rec:\,\textnormal{rectangle}}$ is a direct sum.
\end{prop}
The proof of the above result is a straightforward adaptation of the proof of \cite[Proposition~6.6]{Cochoy2016}. As in \cite{Cochoy2016}, we first prove the result when all the rectangles share the same upper right corner, then that we can always reduce to this specific case. The following lemma will be instrumental in the proof.
\begin{lem}
    \label{lem:inclusion-Ima-recs-included}
    Let $\omod$ be a pfd and weakly exact persistence module over $\poset$. Let $\rec_1 = (\lcut_1^+ \cap \rcut_1^-) \times (\bcut_1^+ \cap \tcut_1^-)$ and $\rec_2 = (\lcut_2^+ \cap \rcut_2^-) \times (\bcut_2^+ \cap \tcut_2^-)$ be two rectangles of $\poset$ such that $(\rcut_1,\tcut_1) = (\rcut_2,\tcut_2)$ and $\rec_1$ is a strict subset of~$\rec_2$. For any~$t\in \rec$, one has $\Ima_{\rec_2,t}^+\left(\omod\right) \subseteq \Ima_{\rec_1,t}^-\left(\omod\right)$.
\end{lem}
\begin{proof}
    Elementary geometric considerations show that $\lcut_2^+\supseteq \lcut_1^+$ and $\bcut_2^+\supseteq \bcut_1^+$, and that one of the two inclusions must be strict, i.e. $\lcut_2^+\supsetneq \lcut_1^+$ or $\bcut_2^+\supsetneq \bcut_1^+$. From the first two inclusions and the definition of horizontal and vertical contributions~\eqref{eqn:def-ImaKer-horizvertic}, it follows that:
    \begin{equation}\label{eq:inclu-Im-inclu-cut}
            \Ima_{\lcut_2,t}^+\left(\omod\right) \subseteq \Ima_{\lcut_1,t}^+\left(\omod\right) \textnormal{ and } \Ima_{\bcut_2,t}^+\left(\omod\right) \subseteq \Ima_{\bcut_1,t}^+\left(\omod\right).
    \end{equation}
    From the second two inclusions and the definition of horizontal and vertical contributions, it follows that:
    \begin{equation}\label{eq:inclu-Im-strict-inclu-cut}
        \Ima_{\lcut_2,t}^+\left(\omod\right) \subseteq \Ima_{\lcut_1,t}^-\left(\omod\right) \textnormal{ or } \Ima_{\bcut_2,t}^+\left(\omod\right) \subseteq \Ima_{\bcut_1,t}^-\left(\omod\right).
        \end{equation}
    From~\eqref{eq:inclu-Im-inclu-cut} and~\eqref{eq:inclu-Im-strict-inclu-cut}, we finally deduce:
    \begin{equation*}
    \begin{split}
        \Ima_{\rec_2,t}^+\left(\omod\right) 
        &= \Ima_{\lcut_2,t}^+\left(\omod\right) \cap \Ima_{\bcut_2,t}^+\left(\omod\right) \\
        &\subseteq \Ima_{\lcut_1,t}^+\left(\omod\right) \cap \Ima_{\bcut_1,t}^-\left(\omod\right) + \Ima_{\lcut_1,t}^-\left(\omod\right) \cap \Ima_{\bcut_1,t}^+\left(\omod\right) \\
        &= \Ima_{\rec_1,t}^-\left(\omod\right).
    \end{split}
    \end{equation*}
\end{proof}
\begin{proof}[Proof of \Cref{prop:direct-sum}]
    Let $(\rec_i)_{i\in \libr 1,n \ribr}$ be a finite family of pairwise distinct rectangles,
    and write~$\rec_i = (\lcut[\cut_i]^+ \cap \rcut[\cut_i]^-) \times (\bcut[\cut_i]^+ \cap \tcut[\cut_i]^-)$. We show that the sum of the submodules $(\filtrate[\rec_i])_{i\in \libr 1,n\ribr}$ is an internal direct sum of $\basemod$, i.e that for any $t\in \poset$, the sum of the subspaces $(\filtrate[\rec_i]<t>)_{i\in \libr 1,n\ribr}$ is an internal direct sum of $\basemod_t$. 
    

    \medskip

    \paragraph{Case where all rectangles have the same upper right corner.}
    Suppose that the set $\{(\rcut[\cut_1],\tcut[\cut_1]),\cdots, (\rcut[\cut_n],\tcut[\cut_n])\}$ of upper right corners is a singleton. The proof of \cite[Proposition~6.6]{Cochoy2016} can be adapted in a straightforward way, replacing the words ``birth quadrants" by ``rectangles with the same upper right corner". We write the proof here for the sake of completeness.

    First, note the equality $\submodule[\rec_i](\basemod) = \submodule[\rec_j](\basemod)$ of submodules of $\basemod$ for all $1\leq i,j \leq n$. Therefore, we denote this submodule of $\basemod$ simply by $\submodule[]$.

    It suffices to prove is that there is at least one of the $\rec_i$'s (say $\rec_1$) such that the (binary) sum of $\filtrate[\rec_1]$ and $\sum_{i\ne 1}\filtrate[\rec_i]$ is a direct sum.
    Then the result follows from an induction on the size~$n$ of the family. Hence, we prove that for any $t\in \poset$, we have:
    \begin{equation*}
        \filtrate[\rec_1]<t>\cap \left(\sum_{i\ne 1}\filtrate[\rec_i]<t>\right) = 0.
    \end{equation*}
    
    Let $t\in \poset$. Since $\filtrate[\rec_i]<t> = 0$ for every $i$ such that $t\notin\rec_i$, we can assume without loss of generality that $t\in\rec_i$ for every~$i\in\libr 1,n \ribr$.

    Up to reordering, we can assume that $\rec_1$ has the rightmost left cut and, in case of ties, that
    it also has the topmost bottom cut among the rectangles with the same left cut. Formally, $\rec_1$ is the rectangle whose bottom left corner is maximal in the lexicographical order on the multiset of bottom left corners $\{(\lcut[\cut_1],\bcut[\cut_1]),\cdots, (\lcut[\cut_n],\bcut[\cut_n])\}$ induced by the total order on cuts given by inclusion on their lower parts (\Cref{prop:cuts-totally-ordered}). It follows that $\rec_1$ contains none of the other rectangles. Those can be partitioned into two subfamilies: the ones (say $\rec_2, \cdots, \rec_k$) contain $\rec_1$ strictly, while the others ($\rec_{k+1}, \cdots, \rec_n$) neither contain~$\rec_1$ nor are contained in $\rec_1$. See \Cref{fig:birth_quadrants} for an illustration. 

    \begin{figure}[htb]
    \centering
    \includegraphics[scale=0.7]{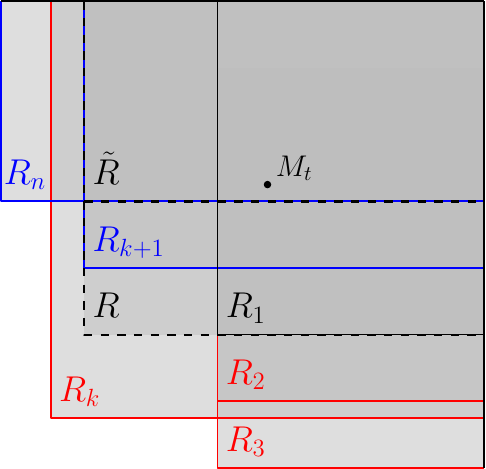}
    \caption{Rectangles partitioned into two subfamilies.}
    \label{fig:birth_quadrants}
    \end{figure}

    For any rectangle $\rec'$ of $\poset$, we denote $\Ima_{\rec',t}^+\left(\submodule[]\right)$ by $\Ima_{\rec',t}^+$ until the end of the proof. For every $i\in \libr 2, k\ribr$, \Cref{lem:inclusion-Ima-recs-included} implies $\Ima_{\rec_i,t}^+ \subseteq \Ima_{\rec_1,t}^-$. Therefore, we obtain:
    \begin{equation}\label{eqn:same-ur-corner-sum-im-plus}
    \sum_{i=2}^k \Ima_{\rec_i,t}^+ \subseteq \Ima_{\rec_1,t}^-.
    \end{equation}
    For every $i\in \libr k+1, n\ribr$, we have $\lcut_i^+\supsetneq\lcut_1^+$ and $\bcut_i^+\subsetneq\bcut_1^+$. Let $\tilde\rec = \bigcap_{i=k+1}^n \rec_{i}$ --- this rectangle neither contains $\rec_1$ nor is contained in it. Let now $\rec$ be the smallest rectangle containing both $\rec_1$ and $\tilde\rec$. We have:
    \begin{equation}\label{eqn:same-ur-corner-cap-im-plus}
        \Ima_{\rec_1,t}^+ \cap \left(\sum_{i=k+1}^n \Ima_{\rec_i,t}^+\right) \subseteq \Ima_{\rec_1,t}^+\cap \Ima_{\tilde\rec,t}^+ = \Ima_{\rec,t}^+ \subseteq \Ima_{\rec_1,t}^-,
    \end{equation}
    where the last inclusion follows from \Cref{lem:inclusion-Ima-recs-included} and the fact that $\rec$ strictly contains $\rec_1$. Combining \eqref{eqn:same-ur-corner-sum-im-plus} and \eqref{eqn:same-ur-corner-cap-im-plus}, we obtain:
    \begin{equation*}
        \begin{split}
            \filtrate[\rec_1]<t>\cap \left(\sum_{i=2}^k \filtrate[\rec_i]<t> + \sum_{i=k+1}^n \filtrate[\rec_i]<t>\right)
            & \subseteq
            \Ima_{\rec_1,t}^+ \cap \left(\sum_{i=2}^k \Ima_{\rec_i,t}^+ + \sum_{i=k+1}^n \Ima_{\rec_i,t}^+\right) \\
            & = \sum_{i=2}^k \Ima_{\rec_i,t}^+ + \Ima_{\rec_1,t}^+ \cap \left(\sum_{i=k+1}^n \Ima_{\rec_i,t}^+\right) \\
            & \subseteq \Ima_{\rec_1,t}^-.
        \end{split}
    \end{equation*}
    Meanwhile, \Cref{rk:filtrate-complement-filt-minus} implies that $\filtrate[\rec_1]<t> \cap \filt[\rec_1]{-}<t>\left(\submodule[]\right) = 0$. Thus,
    \begin{equation*}
        \Ifilt{\rec_1}{-} \cap \filtrate[\rec_1]<t> = \Ifilt{\rec_1}{-} \cap \filt[\rec_1]{+}<t>\left(\submodule[]\right) \cap \filtrate[\rec_1]<t> \subseteq \filt[\rec_1]{-}<t>\left(\submodule[]\right) \cap \filtrate[\rec_1]<t> = 0.
    \end{equation*}
    Since $\filtrate[\rec_1]<t> \cap (\sum_{i=2}^n \filtrate[\rec_i]<t>)$ is contained in both $\filtrate[\rec_1]<t>$ and $\Ifilt{\rec_1}{-}$, one has:
    \begin{equation*}
        \filtrate[\rec_1]<t>\cap \left(\sum_{i=2}^n \filtrate[\rec_i]<t>\right) \subseteq \filtrate[\rec_1]<t> \cap \Ifilt{\rec_1}{-} = 0.
    \end{equation*}

    \paragraph{General case.} Let $t\in \poset$. For each $i\in \libr 1,n\ribr$, let $z_i \in \filtrate[\rec_i]<t>$. Denote $z = \sum_{i=1}^n z_i$ and suppose that $z=0$. Let us show that $z_i = 0$ for all $i\in \libr 1,n\ribr$. Again, since $z_i = 0$ for every $i$ such that $t\notin\rec_i$, we can assume without loss of generality that $t\in\rec_i$ for every~$i\in\libr 1,n \ribr$.

    Order the collection $(\rcut[\cut_1],\tcut[\cut_1]) \preceq \dots \preceq (\rcut[\cut_n],\tcut[\cut_n])$ of upper right corners by the lexicographical order $\preceq$ induced by the total order on cuts given by inclusion on their lower parts (\Cref{prop:cuts-totally-ordered}). Let $(\rcut[\ocut_1],\tcut[\ocut_1]) \prec \dots \prec (\rcut[\ocut_k],\tcut[\ocut_k])$ be the distinct elements in the ordered sequence. In particular, $(\rcut[\ocut_{k}],\tcut[\ocut_{k}])$ is the upper right corner of rectangles with the rightmost right cut and, in case of ties, the topmost top cut. Let $J = \{i\in \libr 1,n \ribr, \, (\rcut[\cut_i],\tcut[\cut_i]) = (\rcut[\ocut_{k}],\tcut[\ocut_{k}]) \}$ and let us show that for all $j\in J$, $z_j = 0$. A direct recursion will then yield $z_i = 0$, for all $i\in \libr 1,n \ribr$.

    By maximality of $(\rcut[\ocut_k],\tcut[\ocut_k])$ in the lexicographical order on upper right corners, there exists $u\in \rcut[\ocut_{k}]^- \times \tcut[\ocut_{k}]^- \setminus (\bigcup_{l\ne k}(\rcut[\ocut_l]^- \times \tcut[\ocut_l]^-)$. Therefore, for $j\notin J$, we have $u\notin \rec_j$, so $\filtrate[\rec_j]<u> = 0$, and thus $\mor_t^u(z_j) = 0$. Hence,
    \begin{equation}
        \label{eqn:sumI0vanishes}
        0 = \mor_t^u(z) = \sum_{i=1}^n \mor_t^u(z_i) = \sum_{j\in J}\mor_t^u(z_j).
    \end{equation}
    Moreover, for all $j\in J$, we have $t\in \rec_j$ and $u\in \rec_j$, thus $\mor_{t}^u$ restricted to $\filtrate[R_{j}]<t>$ is injective by \Cref{lem:recsubmodisisotorecmod}. Therefore, it only remains to show that, for all $j\in J$, one has~$\mor_t^u(z_j) = 0$.

    Since the rectangles of the family~$(\rec_j)_{j\in J}$ all have same upper right corner, the first case ensures that the sum of~$(\filtrate[\rec_j])_{j\in J}$ is a direct sum. Yet, the element~$\mor_t^u(z_j)$ belong to~$\filtrate[\rec_j]<u>$ for any~$j\in J$, so \cref{eqn:sumI0vanishes} implies~$\mor_t^u(z_j) = 0$ for all~$j\in J$, which concludes the proof.

\end{proof}

\section{Rectangle filtrates cover $\basemod$}\label{sec:covering}
Recall that $\basemod$ is a pfd and weakly exact persistence bimodule and that its internal morphisms are denoted by~$\mor_s^t: \basemod_s \to \basemod_t$ for any $s\leq t\in \poset$. Again, recall our choice of rectangle filtrates; \Cref{conv:choice-of-complement}. The goal of this section is to prove the following. 
\begin{prop}
    \label{prop:cover}
    The (direct) sum of submodules $(\filtrate)_{\rec :\,\textnormal{rectangle}}$ generates $\basemod$, i.e.
    \begin{equation*}
        \basemod  = \bigoplus_{R:\emph{ rectangle}} \filtrate.
    \end{equation*}
\end{prop}


To prove the above proposition, we first consider the restriction of $\basemod$ to a specifically constructed finite grid; see \Cref{def:t-skeleton}. This restriction captures all the information on kernels and images of internal morphisms of $\basemod$ accessible from a fixed index $t\in\poset$; see \Cref{lem:grid}. Then, \Cref{lem:unique-rec} explains how the rectangle filtrates of the restriction relates to the rectangle filtration of $\basemod$. These two lemmas, in conjunction with \Cref{thm:finite-case} and \Cref{lem:equality-dim-CF}, will prove \Cref{prop:cover}. \Cref{thm:rec-dec-weak-exact} then follows as a corollary of \Cref{prop:cover} and \Cref{lem:recsubmodisisotorecmod}. The proofs of \Cref{lem:grid,lem:unique-rec} have been postponed to \Cref{sec:grid,sec:linkresbase}, respectively.
\begin{lem}
    \label{lem:grid}
    Let $t\in\poset$. There exist natural numbers $L_h$, $\rightindex$, $L_v$, $\upperindex{}$ and a finite grid 
    \begin{equation*}
        \grid = (x_i,y_j)_{(i,j)\in \libr \leftindex,\rightindex\ribr\times\libr \lowerindex{},\upperindex{}\ribr} \subseteq \poset
    \end{equation*}
    such that:
    \begin{enumerate}[label=(\roman*)]
        \item\label{itm:t} $t = (x_0,y_0)$,
        \item\label{itm:orderKer} for all $i \in \libr 0,\rightindex\ribr$, we have $\Ker\mor_t^{(x_i,t_y)} \subsetneq \Ker\mor_t^{(x_{i+1},t_y)}$ and same for vertical kernels where $x_{i} = +\infty$ for $i = \rightindex+1$ and $y_{j} = +\infty$ for $j = \upperindex+1$ by convention (recall also \Cref{conv:extension}),
        \item \label{itm:orderIm} for all $i \in \libr \leftindex,0 \ribr$, we have $\Ima\mor^t_{(x_{i-1},t_y)} \subsetneq \Ima\mor^t_{(x_i,t_y)}$ and same for vertical images where $x_{i} = -\infty$ for $i = \leftindex-1$ and $y_{j} = -\infty$ for $j = \lowerindex-1$ by convention,
        \item\label{itm:coverKer} for all $x\in [t_x,+\infty]$, there exists $i \in \libr 0,\rightindex+1\ribr$ such that $\Ker\mor_t^{(x,t_y)} = \Ker\mor_t^{(x_{i},t_y)}$ and same for vertical kernels.
        \item\label{itm:coverIm} for all $x\in [-\infty,t_x]$, there exists $i \in \libr \leftindex-1,0 \ribr$ such that $\Ima\mor^t_{(x,t_y)} = \Ima\mor^t_{(x_i,t_y)}$ and same for vertical images.
    \end{enumerate}
\end{lem}
\begin{defi}\label{def:t-skeleton}
    Let $t\in\poset$. Any finite grid $G\subseteq \poset$ given by \Cref{lem:grid} is called \emph{\lifegrid{t} of $\basemod$}.
\end{defi}
\begin{rk}
    Note that the statements \ref{itm:orderKer} and \ref{itm:orderIm} ensure that the indices $i$ given in \ref{itm:coverKer} and \ref{itm:coverIm} -- realizing kernels and images of the base module inside the grid -- are unique.
\end{rk}

\begin{ex}
    \label{ex:grid}
    \Cref{fig:grid-example} illustrates a $t$-skeleton when $\basemod$ is the direct sum of rectangle modules associated to rectangles $\rec_1,\rec_2$ and $\rec_3$. 
\end{ex}

\begin{figure}[!ht]
    \centering
    \includegraphics[width=\textwidth]{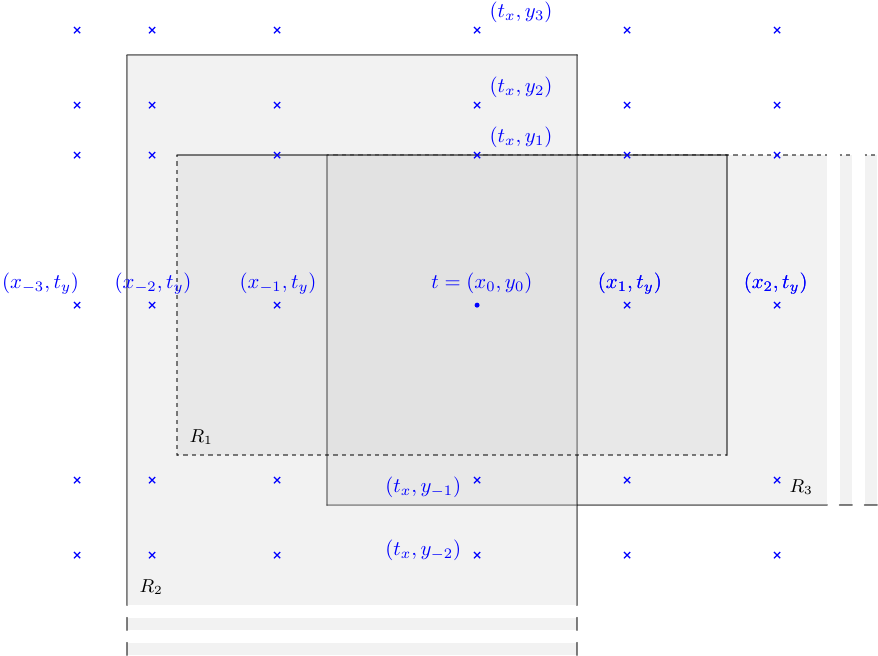}
    \caption{Example of a choice of a grid construction as in \Cref{lem:grid}. Dashed lines denote open boundaries and dashed rectangles denote infinite sides. The initial point $t$ is denoted by a dot, while crosses denote points on the constructed grid.}
    \label{fig:grid-example}
\end{figure}



\begin{lem}
    \label{lem:unique-rec}
    Let $t\in \poset$, let $\grid$ be a \lifegrid{t} of $\basemod$ and denote $\gridres := \basemod_{|\grid}$. To any rectangle $\recres{} = (\lcut[\cutres]^+ \cap \rcut[\cutres]^-) \times (\bcut[\cutres]^+ \cap \tcut[\cutres]^-)$ of $\grid$ such that $t\in \recres$, one can associate a rectangle $\rec = (\lcut^+ \cap \rcut^-) \times (\bcut^+ \cap \tcut^-)$ of $\poset$, such that:
    \begin{enumerate}[label=(\roman*)]
        \item\label{item:t-in-rec} $t\in\rec$,
        \item\label{itm:equality-ker-im} one has:
        \begin{equation*}
            \begin{split}
                \Kfilt{\rcut[\cutres]}{\pm}(\gridres) &= \Kfilt{\rcut}{\pm}(\basemod), \\
                \Ifilt{\lcut[\cutres]}{\pm}(\gridres) &= \Ifilt{\lcut}{\pm}(\basemod),
            \end{split}
        \end{equation*}
        and similarly for vertical cuts,
        \item\label{itm:injectivity-ker-im} the map $\recres{}\mapsto\rec$ is injective.
    \end{enumerate} 
    In particular,
    \begin{equation*}
        \dim \filtrate[\rec]<t> = \dim \filtrate[\res{\rec}]<t>^\grid.
    \end{equation*}
\end{lem}
We can now prove \Cref{prop:cover}.
\begin{proof}[Proof of \Cref{prop:cover}]
    Let $t\in \poset$ and let us show that
    \begin{equation}
        \label{eqn:cover-pointwise}
        \basemod_t  = \bigoplus_{R:\textnormal{ rectangle}} \filtrate<t>.
    \end{equation}

    Take a $t$-skeleton $\grid$ of $\basemod$ given by \Cref{lem:grid}, and denote $\gridres := \basemod_{|\grid}$. Notice that, since $t\in\grid$ by \ref{itm:t} of \Cref{lem:grid}, we have
    \begin{equation}
        \label{eqn:basemod-equal-gridres-t}
        \basemod_t = \gridres_t.
    \end{equation}
    Moreover, since $\gridres$ is still pfd and weakly exact, it decomposes as a direct sum of rectangle modules by \Cref{thm:finite-case}:
    \begin{equation}
        \label{eq:decomposition-gridres}
        \gridres \simeq \bigoplus_{j\in\tilde{J}} \left(\indimod[\res{\rec}_j]^\grid\right)^{m_{\res{\rec}_j}},
    \end{equation}
    where the rectangles $\res{\rec}_j$ are pairwise distinct rectangles of the grid $\grid$, and where the integers $m_{\res{\rec}_j}>0$ are the multiplicities of the rectangle modules $\indimod[\res{\rec}_j]^\grid$ in the decomposition.
    Since $\grid$ is finite and $\gridres$ is pfd, the set $\tilde{J}$ appearing in the decomposition~\eqref{eq:decomposition-gridres} is finite. Therefore, \Cref{lem:equality-dim-CF} implies that for each $j\in \tilde{J}$,
    \begin{equation}
        \label{eq:dim-filtrate-multiplicity}
        \dim\left(\filtrate[\res{\rec}_j]<t>^\grid\right) = m_{\res{\rec}_j}.
    \end{equation}
    Besides, since $t\in\grid$, we can consider the subset $J := \{j\in\tilde{J}\,|\, t\in \res{\rec}_j\}$, so that:
    \begin{equation}
        \label{eq:dim-sum}
        \dim\left(\gridres_t\right) = \sum_{j\in J} m_{\tilde{\rec}_j}.
    \end{equation}
    Meanwhile, writing $\{\rec_i\}_{i\in I}$ the set of rectangles of $\poset$ containing $t$, \Cref{lem:unique-rec} yields an injection $\iota : J \hookrightarrow I$, such that
    \begin{equation}
        \label{eqn:dim-filtrate-injection}
        \dim \filtrate[\rec_{\iota(j)}]<t> = \dim \filtrate[\recres_j]<t>^\grid.
    \end{equation}
    We can compute the finite dimensions:
    \begin{equation*}
        \begin{split}
            \dim(\basemod_t)
            &\overset{\textnormal{\eqref{eqn:basemod-equal-gridres-t}}}{=} \dim\left(\gridres_t\right) \\
            &\overset{\textnormal{\eqref{eq:dim-sum}}}{=} \sum_{j\in J} m_{\tilde{\rec}_j} \\
            &\overset{\textnormal{\eqref{eq:dim-filtrate-multiplicity}}}{=} \sum_{j\in J} \dim\left(\filtrate[\res{\rec}_j]<t>^\grid\right) \\
            &\overset{\textnormal{\eqref{eqn:dim-filtrate-injection}}}{=} \sum_{j\in J} \dim\left(\filtrate[\rec_{\iota(j)}]<t>\right)\\
            &\hspace{0.7em} = \hspace{0.7em} \dim\left(\bigoplus_{j\in J}\filtrate[\rec_{\iota(j)}]<t>\right),
        \end{split}
    \end{equation*}
    and conclude by the inclusion $\bigoplus_{j\in J}\filtrate[\rec_{\iota(j)}]<t> \subseteq \basemod_t$ that $\bigoplus_{j\in J}\filtrate[\rec_{\iota(j)}]<t> = \basemod_t$. Finally, we have
    \begin{equation*}
        \basemod_t = \bigoplus_{j\in J}\filtrate[\rec_{\iota(j)}]<t> \subseteq \bigoplus_{\rec:\textnormal{ rectangle}}\filtrate[\rec]<t> \subseteq \basemod_t,
    \end{equation*}
    which concludes the proof of \cref{eqn:cover-pointwise}. Hence the result.
\end{proof}

\subsection{Proof of \Cref{lem:grid}}
\label{sec:grid}
\Cref{lem:grid} is a direct consequence of finite dimensionality, as was \Cref{lem:realization}. We write its proof here for completeness. Since $\basemod$ is pfd, the function $x\in[t_x,+\infty] \mapsto \dim \Ker\rho_t^{(x,t_y)}$ takes a finite number of values $0 = n_0 < n_1 < \dots < n_{\rightindex} < n_{\rightindex+1} = \dim\basemod_t$. This function is also increasing, so fixing $x_0 = t_x$ and $x_{\rightindex+1} = +\infty$, we can find real numbers $x_0 < x_1 < \dots < x_{\rightindex} < x_{\rightindex+1}$ such that $\dim\Ker\rho_t^{(x_i,t_y)} = n_i$ for all $0 \leq i \leq \rightindex+1$. We can define similar real numbers for vertical kernels: $t_y=y_0 < y_1 < \dots < y_{\upperindex} < y_{\upperindex+1} = +\infty$ such that $\dim\Ker\rho_t^{(t_x,y_j)} = \tilde{n}_j$ where $(\tilde{n}_j)_{j\in\libr 0,\upperindex+1\ribr}$ are the distinct dimensions of vertical kernels.

Similarly, since $\basemod$ is pfd, $x\in[-\infty,t_x] \mapsto \dim\Ima\rho^t_{(x,t_y)}$ takes a finite number of values $0 = m_{\leftindex-1} < m_{\leftindex} < \dots < m_{-1} < m_0 = \dim \basemod_t$. This function is also increasing, so fixing $x_{\leftindex-1} = -\infty$ we can find real numbers $x_{\leftindex-1} < x_{\leftindex} < \dots < x_{-1} < x_0 = t_x$ such that $\dim\Ima\rho^t_{(x_i,t_y)} = m_i$ for all $ i\in\libr\leftindex-1,0\ribr$. We can define similar real numbers for vertical images: $-\infty = y_{\lowerindex-1} < y_{\lowerindex} < \dots < y_{-1} < t_y$ such that $\dim\Ima\rho^t_{(t_x,y_j)} = \tilde{m}_j$ where $(\tilde{m}_j)_{j\in\libr \lowerindex-1,0 \ribr}$ are the distinct dimensions of vertical images.

Define finally the finite grid $\grid := \{(x_i,y_j),\,(i,j)\in \libr \leftindex,\rightindex\ribr \times \libr \lowerindex,\upperindex \ribr\}$. It remains to show that this grid satisfies the required properties.

First, \ref{itm:t} comes from $x_0=t_x$ and $y_0 = t_y$.

Second, \ref{itm:orderKer} and \ref{itm:orderIm} are clear from the construction of the grid: spaces associated to indices are ordered by inclusion and they are distinct if indices are distinct because then their dimensions are distinct.

Third, \ref{itm:coverKer} and \ref{itm:coverIm} are also clear from the construction of the grid: every possible horizontal or vertical kernel and image has been represented by an index in the grid.

\subsection{Proof of \Cref{lem:unique-rec}}
\label{sec:linkresbase}
Let $t\in \poset$, let $\grid = (x_i,y_j)_{(i,j)\in \libr \leftindex,\rightindex\ribr\times\libr \lowerindex{},\upperindex{}\ribr}$ be a \lifegrid{t} of $\basemod$ and denote $\gridres := \basemod_{|\grid}$. The proof of \Cref{lem:unique-rec} is postponed to the end of this section. It uses the following three lemmas.

\begin{lem}
    \label{lem:unique-cut-ker}
    To any cut $\rcut[\cutres]$ of $(x_i)_{i\in \libr \leftindex,\rightindex\ribr}$ such that $t_x\in\rcut[\cutres]^-$, one can associate a cut $\rcut$ of $\abscisse$ such that:
    \begin{enumerate}[label=(\roman*)]
        \item\label{itm:t-kernels} $t_x \in \rcut^-$,
        \item\label{itm:equality-kernels} $\Kfilt{\rcut[\cutres]}{\pm}(\gridres) = \Kfilt{\rcut}{\pm}(\basemod)$,
        \item\label{itm:injectivity-kernels} the map $\rcut[\cutres]\mapsto\rcut$ is injective. 
    \end{enumerate}
    A similar result holds for vertical cuts and vertical kernels.
\end{lem}

\begin{proof}
    Any cut $\rcut[\cutres]$ of $(x_i)_{i\in \libr \leftindex,\rightindex\ribr}$ such that $t_x\in\rcut[\cutres]^-$ can be denoted by $\rcut[\cutres]^- = (x_i)_{i\in \libr \leftindex,\recrightindex\ribr}$ with $\recrightindex\in\libr 0,\rightindex\ribr$. This implies

    \begin{align}
        \Kfilt{\rcut[\cutres]}{-}(\gridres) &= \Ker\mor_t^{(x_{\recrightindex},t_y)},\\
        \label{eqn:unique-cut-ker+-grid}\Kfilt{\rcut[\cutres]}{+}(\gridres) &= \Ker\mor_t^{(x_{\recrightindex+1},t_y)}
    \end{align}
    with possibly $x_{\recrightindex+1} = +\infty$. Now, define the cut $\rcut$ of $\abscisse$ by
    \begin{equation}
        \label{eqn:unique-cut-def}
        \begin{split}
            \rcut^- &:= (-\infty,x_{\recrightindex}]\cup \Big\{\normalfont x \in [x_{\recrightindex},x_{\recrightindex+1}),\, \Ker\mor_t^{(x,t_y)} = \Ker\mor_t^{(x_{\recrightindex},t_y)}\Big\},\\
            \rcut^+ &:= \abscisse\setminus\rcut^-,
        \end{split}
    \end{equation}
    and notice that $t_x\in\rcut^-$, hence~\ref{itm:t-kernels}.

    Let us show~\ref{itm:equality-kernels}. By \Cref{lem:realization} applied to $\basemod$, we can find $x\in \rcut^-$ such that $\Kfilt{\rcut}{-}(\basemod) = \Ker\mor_t^{(x,t_y)}$. Since $x_{\recrightindex} \in \rcut^-$, we can even choose $x$ in $\rcut^- \cap [x_{\recrightindex},+\infty)$, which implies by definition of $\rcut$ that $\Ker\mor_t^{(x,t_y)} = \Ker\mor_t^{(x_{\recrightindex},t_y)}$. Hence,
    \begin{equation*}
        \Kfilt{\rcut}{-}(\basemod) = \Ker\mor_t^{(x,t_y)} = \Ker\mor_t^{(x_{\recrightindex},t_y)} = \Kfilt{\rcut[\cutres]}{-}(\gridres).
    \end{equation*}
    Similarly, by \Cref{lem:realization} we can find $x\in \rcut^+\cup\{+\infty\}$ such that $\Kfilt{\rcut}{+}(\basemod) = \Ker\mor_t^{(x,t_y)}$. Since $x \in \rcut^+\cup\{+\infty\}$ and $t_x\leq x_{k_h}\in \rcut^-$, we have:
    \begin{equation}
        \label{eqn:unique-cut-inclusion-small-ker}
        \Ker\mor_t^{(x_{\recrightindex},t_y)} \subsetneq \Ker\mor_t^{(x,t_y)}.
    \end{equation}
    Moreover, since $x_{\recrightindex+1} \in \rcut^+\cup\{+\infty\}$, we can lower $x$ if necessary to choose $x\in \rcut^+\cap (-\infty,x_{\recrightindex+1}]$, and then
    \begin{equation}
        \label{eqn:unique-cut-big-ker}
        \Ker\mor_t^{(x,t_y)} \subseteq \Ker\mor_t^{(x_{\recrightindex+1},t_y)}.
    \end{equation}
    By definition of a \lifegrid{t} (\Cref{lem:grid}~\ref{itm:coverKer}), there exists $i\in\libr 0, \rightindex+1\ribr$ such that $\Ker\mor_t^{(x,t_y)} = \Ker\mor_t^{(x_i,t_y)}$. Therefore, \Cref{lem:grid}~\ref{itm:orderKer} combined with equations \eqref{eqn:unique-cut-inclusion-small-ker} and \eqref{eqn:unique-cut-big-ker} implies:
    \begin{equation}
        \label{eqn:unique-cut-associated-grid-ker}
        \Ker\mor_t^{(x,t_y)} = \Ker\mor_t^{(x_{\recrightindex+1},t_y)},
    \end{equation}
    and it finally follows by \eqref{eqn:unique-cut-ker+-grid} and \eqref{eqn:unique-cut-associated-grid-ker} that
    \begin{equation*}
        \Kfilt{\rcut}{+}(\basemod) = \Ker\mor_t^{(x,t_y)} = \Ker\mor_t^{(x_{\recrightindex+1},t_y)} = \Kfilt{\rcut[\cutres]}{+}(\gridres).
    \end{equation*}

    Let us now show~\ref{itm:injectivity-kernels}. Let $\rcut[\cutres] \ne \rcut[\cutres']$ be two cuts of $\grid$ with $t_x\in\rcut[\cutres]$ and $t_x\in\rcut[\cutres']$. Write
    \begin{equation*}
        \begin{split}
            \rcut[\cutres]^- &= (x_i)_{i\in \libr \leftindex,\recrightindex\ribr},\\
            \rcut[\cutres']^- &= (x_i)_{i\in \libr \leftindex,\recrightindex'\ribr}.
        \end{split}
    \end{equation*}
    Write also $\rcut$ and $\rcut[\cut']$ the respective cuts associated to $\rcut[\cutres]$ and $\rcut[\cutres']$ by the previous construction. Since $\rcut[\cutres] \ne \rcut[\cutres']$, the indices delimiting the cuts must differ: say for instance~$\recrightindex < \recrightindex'$, the other case being similar. Then, it is clear from the definition~\eqref{eqn:unique-cut-def} that $x_{\recrightindex'} \in \rcut[\cut']^- \setminus \rcut^-$, and therefore $\rcut \ne \rcut[\cut']$.

\end{proof}

A similar result holds for images, as shown by the following lemma:
\begin{lem}
    \label{lem:unique-cut-im}
    To any cut $\lcut[\cutres]$ of $(x_i)_{i\in \libr \leftindex,\rightindex\ribr}$ such that $t_x\in\lcut[\cutres]^+$ one can associate a cut $\lcut$ of $\abscisse$ such that:
    \begin{enumerate}[label=(\roman*)]
        \item\label{itm:t-images} $t_x\in\lcut^+$,
        \item\label{itm:equality-images} $\Ifilt{\lcut[\cutres]}{\pm}(\gridres) = \Ifilt{\lcut}{\pm}(\basemod)$,
        \item\label{itm:injectivity-images} the map $\lcut[\cutres]\mapsto\lcut$ is injective. 
    \end{enumerate} 
    A similar result holds for vertical cuts and vertical images.
\end{lem}

\begin{proof}
    Let $\lcut[\cutres]$ be a cut of $(x_i)_{i\in \libr \leftindex,\rightindex\ribr}$ such that $t_x\in\lcut[\cutres]^+$. Write $\lcut[\cutres]^+ = (x_i)_{i\in \libr \recleftindex,\rightindex\ribr}$ with $\recleftindex\in\libr \leftindex,0 \ribr$. This implies
    \begin{equation*}
        \begin{split}
            \Ifilt{\lcut[\cutres]}{+}(\gridres) &= \Ima\mor_{(x_{\recleftindex},t_y)}^t, \\
            \Ifilt{\lcut[\cutres]}{-}(\gridres) &= \Ima\mor_{(x_{\recleftindex-1},t_y)}^t.
        \end{split}
    \end{equation*}
    with possibly $x_{\recleftindex-1} = -\infty$. We can now define the cut $\lcut$ of $\abscisse$ by
    \begin{equation*}
        \begin{split}
            \lcut^+ &:= \Big\{\normalfont x \in (x_{\recleftindex-1},x_{\recleftindex}],\, \Ima\mor_{(x,t_y)}^t = \Ima\mor_{(x_{\recleftindex},t_y)}^t\Big\} \cup [x_{\recleftindex},+\infty),\\
            \lcut^- &:= \abscisse\setminus\lcut^+,
        \end{split}
    \end{equation*}
    and notice that $t_x\in\lcut^+$. The rest of the proof is symmetric to the one for kernels (\Cref{lem:unique-cut-ker}).

\end{proof}

Finally, we describe the horizontal and vertical contributions associated to the rectangle $\rec$ in $\submodule(\basemod)$. Recall that we refer to $\submodule(\basemod)$ simply as $\submodule$. 
\begin{lem}
    \label{lem:filt-submodule}
    For $t\in \rec$, we have: 
    \begin{enumerate}[label=(\roman*)]
        \item\label{itm:Ifilt-submodule} $\Ifilt{\lcut}{\pm}(\submodule) = \Ifilt{\lcut}{\pm}(\basemod) \cap \submodule<t>$,
        \item\label{itm:Kfilt-submodule} $\Kfilt{\rcut}{\pm}(\submodule) = \Kfilt{\rcut}{\pm}(\basemod) \cap \submodule<t>$,
    \end{enumerate}
    and similar statements for vertical contributions.
\end{lem}
\begin{proof}
    In this proof we write~$\tilde{\mor}_u^v = \mor_{u|\submodule<u>}^v$ for any~$u\leq v \in \poset$. Let us first show~\ref{itm:Ifilt-submodule}. We prove the result on $\Ifilt{\lcut}{-}(\submodule)$, the other one is similar. By \Cref{rk:realization}, one can find~$x\in \lcut^-\cup\{-\infty\}$ such that:
    \begin{equation*}
        \begin{split}
        \Ifilt{\lcut}{-}\left(\basemod\right) &= \Ima\mor_{(x, t_y)}^t,\\
        \Ifilt{\lcut}{-}\left(\submodule\right) &= \Ima\tilde{\mor}_{(x, t_y)}^t = \mor_{(x, t_y)}^t\left(\submodule<(x, t_y)>\right).
        \end{split}
    \end{equation*}
    Denote~$a = (x, t_y)$. Then, one has:
    \begin{equation*}
        \Ifilt{\lcut}{-}\left(\submodule\right) = \mor_{a}^t\left(\submodule<a>\right) \subseteq \Ima\mor_{a}^t = \Ifilt{\lcut}{-}\left(\basemod\right).
    \end{equation*}
    Since also~$\mor_{a}^t\left(\submodule<a>\right) \subseteq \submodule<t>$, one has~$\Ifilt{\lcut}{-}(\submodule) \subseteq \Ifilt{\lcut}{-}(\basemod) \cap \submodule<t>$.

    Now, let~$z\in \Ifilt{\lcut}{-}\left(\basemod\right)\cap\submodule<t>  = \Ima\mor_{a}^t \cap \submodule<t>$. There is~$z_{a}\in \basemod_{a}$ such that~$z = \mor_{a}^t(z_{a})$. 
    %
    Applying \Cref{lem:realization} at the point $t$ and at the point $a$, one can choose low enough~$x'\in\rcut^+\cup\{+\infty\}$ and~$y\in\tcut^+\cup\{+\infty\}$ such that:
    \begin{equation*}
        \begin{split}
            \submodule<t> &= \Ker\mor_t^{(t_x,y)} \cap \Ker\mor_t^{(x', t_y)},\\
            \submodule<a> &= \Ker\mor_{a}^{(x, y)} \cap \Ker\mor_{a}^{(x', t_y)}.
        \end{split}
    \end{equation*}
    Denote~$c = (t_x, y)$,~$d = (x',t_y)$, and $b = (x, y)$. The following diagram will help picturing the various spaces involved in this proof.
    \begin{equation*}
        \begin{tikzcd}
            {\basemod_{b}} \arrow[r]           & {\basemod_{c}}         &                     \\
            {\basemod_{a}} \arrow[u] \arrow[r] & \basemod_{t} \arrow[u] \arrow[r]         & {\basemod_{d}}
            \end{tikzcd}
    \end{equation*}
    Since~$z \in\submodule<t>\subseteq \Ker\mor_t^c$, one has~$\mor_{a}^c(z_{a}) = \mor_t^c(z) = 0$, i.e.~$z_{a}\in \Ker\mor_{a}^c$. Thus, by weak exactness of~$\basemod$ and \Cref{rk:extension-weak-exactness}, there exist~$z'\in \Ker\mor_{a}^{b}$ and~$z'' \in \Ker\mor_{a}^t$ such that~$z_{a} = z' + z''$. Moreover, since~$z \in\submodule<t>\subseteq \Ker\mor_t^d$, one has~$\mor_{a}^d(z_{a}) = \mor_t^d(z) = 0$, i.e.~$z_{a}\in \Ker\mor_{a}^d$. Then also~$z' = z_{a} - z'' \in \Ker\mor_{a}^d$. Hence~$z'\in \submodule<a>$ and
    \begin{equation*}
        z = \mor_{a}^t(z_{a}) =  \mor_{a}^t(z') \in \mor_{a}^t\left(\submodule<a>\right).
    \end{equation*}

    Let us now show~\ref{itm:Kfilt-submodule}. 
    By \Cref{rk:realization}, there are~$x_\pm\in\rcut^+\cup\{+\infty\}$ such that:
    \begin{equation*}
        \begin{split}
            \Kfilt{\rcut}{\pm}\left(\basemod\right) &= \Ker\mor_t^{(x_\pm,t_y)},\\
            \Kfilt{\rcut}{\pm}\left(\submodule\right) &= \Ker\tilde{\mor}_t^{(x_\pm,t_y)}. 
        \end{split}
    \end{equation*}
    Hence, we get:
    \begin{equation*}
        \Kfilt{\rcut}{\pm}\left(\submodule\right)
        = \Ker\mor_t^{(x_\pm,t_y)}\cap\submodule<t> = \Kfilt{\rcut}{\pm}\left(\basemod\right) \cap \submodule<t>.
    \end{equation*}
\end{proof}

\begin{proof}[Proof of \Cref{lem:unique-rec}.]
    For each cut composing $\recres$ (namely  $\lcut[\cutres]$, $\rcut[\cutres]$, $\bcut[\cutres]$ and $\tcut[\cutres]$), we can apply \Cref{lem:unique-cut-ker} and \Cref{lem:unique-cut-im} to find a cut (respectively $\lcut$, $\rcut$, $\bcut$ and $\tcut$) satisfying the corresponding wanted equality on kernels or images. Considering the rectangle $\rec = (\lcut^+ \cap \rcut^-) \times (\bcut^+ \cap \tcut^-)$ of $\poset$ yields~\ref{item:t-in-rec},~\ref{itm:equality-ker-im} and~\ref{itm:injectivity-ker-im}. For the equality of dimensions of the rectangle filtrates, note that:
    \begin{equation*}
        \submodule[\recres]<t>\left(\gridres\right) = \submodule[\rec]<t>(\basemod).
    \end{equation*} 
    Moreover, the computations of the filtrations of $\submodule[\recres](\gridres)$ and $\submodule(\basemod)$ (\Cref{lem:filt-submodule}) combined with~\ref{itm:equality-ker-im} implies:
    \begin{equation*}
        \filt[\recres]{\pm}<t>\left(\submodule[\recres]\left(\gridres\right)\right) = \filt[\rec]{\pm}<t>\left(\submodule[\rec](\basemod)\right).
    \end{equation*}
    Hence $\dim \filtrate[\rec] = \dim \filtrate[\res{\rec}]^\grid$.
    
\end{proof}


\section{Negative results}
\label{sec:negative-answers}
\subsection{Proof of \Cref{thm:negative-int-subgrids}.}

Given an integer~$m\geq 2$, consider the following persistence module over the poset $\libr 1, m+1\ribr^2$, where~$\iota_i: \field\hookrightarrow\field^m$ denotes the injection into the~$i$-th axis of~$\field^m$, and  $\delta_m: t\in\field\mapsto (t,\dots,t)\in\field^m$ denotes the injection into the diagonal:
\begin{equation}
    \label{eqn:def-indec-grid}
    \indecgrid{m} :=
    \begin{tikzcd}
        \field \arrow[r, "\iota_1"]            & \field^m \arrow[r]                               & \field^m \arrow[r, dashed]           & \field^m \arrow[r]                               & \field^m                     \\
        0 \arrow[u, dotted] \arrow[r, dotted]  & \field \arrow[u, "\iota_2"] \arrow[r, "\iota_2"] & \field^m \arrow[u] \arrow[r, dashed] & \field^m \arrow[r] \arrow[u]                     & \field^m \arrow[u]           \\
                                               &                                                  & \ddots                               & \field^m \arrow[u, dashed] \arrow[r]             & \field^m \arrow[u, dashed]   \\
        0 \arrow[r, dotted] \arrow[uu, dotted] & 0 \arrow[uu, dotted] \arrow[rr, dotted]          &                                      & \field \arrow[r, "\iota_m"] \arrow[u, "\iota_m"] & \field^m \arrow[u]           \\
        0 \arrow[u, dotted] \arrow[r, dotted]  & 0 \arrow[u, dotted] \arrow[rr, dotted]           &                                      & 0 \arrow[u, dotted] \arrow[r, dotted]            & \field \arrow[u, "\delta_m"]
        \end{tikzcd}
\end{equation}
This persistence module was proven in~\cite{botnan2020rectangledecomposable} to have the following properties: 
\begin{prop}[\cite{botnan2020rectangledecomposable}]
    \label{prop:counter-example-grid}
    For $m\geq 2$, the persistence module~$\indecgrid{m}$ satisfies:
    \begin{enumerate}[label=(\roman*)]
        \item\label{itm:indecomposable}~$\indecgrid{m}$ is indecomposable with local endomorphism ring, in particular it is not interval-decomposable;
        \item\label{itm:minimal} for any strict subgrid~$X'\times Y'\subsetneq \libr 1,m+1\ribr^2$, the restriction~$\indecgrid{m}_{|X'\times Y'}$ belongs to~$\decclass{\intervals[X'\times Y']}$.
    \end{enumerate}
\end{prop}
These two properties follow intuitively from the fact that $\indecgrid{m}$ is the embedded image of the following indecomposable representation of the quiver $\dart{m}$ into the grid $\libr 1,m+1\ribr^2$:
\begin{equation*}
    \begin{tikzcd}[sep=tiny]
    \field \arrow[rrrr, "\iota_1"] &                                 &        &                                & \field^m                         \\
                        & \field \arrow[rrru, "\iota_2"'] &        &                                &                                  \\
                        &                                 & \ddots &                                &                                  \\
                        &                                 &        & \field \arrow[ruuu, "\iota_m"] &                                  \\
                        &                                 &        &                                & \field \arrow[uuuu, "\delta_m"']
    \end{tikzcd}
\end{equation*}

Suppose now that~$\poset$ is a product of two totally ordered sets such that~$|X| \geq 3$ and~$|Y| \geq 3$, and let~$\maxsize$ be an integer such that~$2 \leq \maxsize < \min(|X|,|Y|)$. Note that~$\subgrids_\maxsize(\poset)$, the set of grids of size at most~$\maxsize \times \maxsize$ included in~$\poset$, is a subset of~$\fullsubpos[\poset]\setminus \{\poset\}$.
In this setting, there are poset inclusions~$\libr 1,\maxsize+1\ribr \hookrightarrow X$ and~$\libr 1,\maxsize+1\ribr\hookrightarrow Y$, and we can consider their product~$\psi:\libr 1,\maxsize+1\ribr^2 \hookrightarrow \poset$. We extend the indecomposable module~$\indecgrid{\maxsize}$ from~\eqref{eqn:def-indec-grid} to a persistence module over~$\poset$ by taking its left Kan extension~$\basemod$ along~$\psi$. The resulting persistence module is simply a ``ceiling" modules~\cite[Sec.~2.5]{Botnan2016}. Specifically, for all~$t\in\poset$ we have:
\begin{equation}\label{eqn:characterization-lan}
    \basemod_t = \colim \indecgrid{\maxsize}_{|\psi_{\leq t}} \simeq
        \begin{cases}
            \indecgrid{\maxsize}_{\max(\psi_{\leq t})} & \mbox{ if } \psi_{\leq t} \ne\emptyset, \\[0.5em]
            0 &\mbox{ otherwise},
        \end{cases}
\end{equation}
where~$\psi_{\leq t}$ denotes the downset~$\{u\in\libr 1,\maxsize+1\ribr^2 \mid \psi(u) \leq t\}$. Similarly, the internal morphisms of $m$ are either trivial, or they correspond to internal morphisms $N^m$. 

%
From \Cref{prop:counter-example-grid} (i) it follows that $M$ is not interval-decomposable. Indeed, if it were interval-decomposable, then its restriction to $\libr 1,\maxsize+1\ribr^2$ would be as well. However, this restriction is precisely $N^m$, contradicting that $N^m$ is not interval-decomposable. It is not hard to check that $\basemod_{|X'\times Y'} \in \decclass{\intervals[X'\times Y']}$ for any grid~$X'\times Y'\in \subgrids_m(\poset)$. 

%

\subsection{Proof of \Cref{thm:negative-squares-hook}}
\label{sec:adding-hook} 

We will identify an interval in~$\indecsupports$ that is not a rectangle (\Cref{sec:setting-hook}), then construct a persistence module~$M$ over~$\poset$ from this interval (\Cref{sec:def-counter-example-hook}), and finally prove in \Cref{sec:proof-hook} that $M$ is not interval-decomposable despite satisfying~$M_{|Q} \in\decclass{\indecsupports_{|Q}}$ for every square~$Q\in\squares[\poset]$. 

\subsubsection{Intervals of a square}
\label{sec:def-square-notations}
For~$s\leq t$ in~$\poset$, let~$a := s$, $b := (s_x,t_y)$, $c := (t_x,s_y)$ and~$d := t$. In other words, $\squarepos{s}{t}$ is precisely the square~$\{a,b,c,d\}$ of~$\poset$. The set of intervals of~$\squarepos{s}{t}$ is then:
\begin{equation*}
    \intervals[\squarepos{s}{t}]=\{\{a\},\{b\}, \{c\}, \{d\},\{a,b\},\{a,c\},\{b,d\},\{c,d\},\{a,b,c\},\{b,c,d\},\{a,b,c,d\}\}.
\end{equation*}
Of these intervals, two are not rectangles: the \emph{bottom hook} and \emph{top hook}:
\begin{equation*}
        \hook_1(\squarepos{s}{t}) = \{a,b,c\}, \hspace{0.2\linewidth} \hook_2(\squarepos{s}{t}) = \{b,c,d\}.
\end{equation*}
Hence, for a square~$Q$ of~$\poset$, the condition~$\indecsupports_{|Q} \supsetneq \rectangles[Q]$ says precisely that~$\hook_1(Q)\in\indecsupports_{|Q}$ or~$\hook_2(Q)\in\indecsupports_{|Q}$.

\subsubsection{An interval in $\indecsupports$ that is not a rectangle}
\label{sec:setting-hook}
Assume that~$|X| \geq 2$ and~$|Y| \geq 2$, and that~$(|X|, |Y|) \neq (2, 2)$. Without loss of generality, one can assume that~$|X| \geq 3$ and~$|Y| \geq 2$. In that case, there exists~$x_1 < x_2 < x_3$ in~$X$ and~$y_1 < y_2$ in~$Y$ such that:
\begin{equation*}
    G := \{(x_i,y_j)\}_{(i,j)\in\{1,2,3\}\times\{1,2\}}\subseteq \poset.    
\end{equation*}

\medskip

Let~$\indecsupports \subseteq \intervals[\poset]$ be such that~$\indecsupports_{|Q} \supsetneq \rectangles[Q]$ for all~$Q\in\squares[\poset]$. Denoting~$Q_0 := \squarepos{(x_1,y_1)}{(x_3,y_2)}$ the outermost square of~$G$, we have~$\indecsupports_{|Q_0} \supsetneq \rectangles[Q_0]$. Therefore, either~$\hook_1(Q_0)\in \indecsupports_{|Q_0}$ or~$\hook_2(Q_0) \in \indecsupports_{|Q_0}$. One can assume without loss of generality that~$\hook_2(Q_0)\in \indecsupports_{|Q_0}$, the other case being dual. By definition of~$\indecsupports_{|Q_0}$, there is some interval~$S \in \indecsupports$ such that~$\hook_2(Q_0) = S\cap Q_0$. In particular, we have~$(x_1,y_1)\notin S$ while~$(x_1,y_2)$, $(x_3,y_1)$ and~$(x_3,y_2)$ are in~$S$, thus $S$ is not a rectangle.

\subsubsection{Building the counter-example}
\label{sec:def-counter-example-hook}
Consider the following partition of~$\hullconvex(G)$, which is the convex hull of~$G$ in~$\poset$ (i.e. the set of points~$z\in \poset$ such that~$(x_1,y_1) \leq z \leq (x_3, y_2)$):
\begin{equation}
    \label{eq:def-partition}
    \begin{split}
        \partconv_1 :=&\; (\{x_1\}\times [y_1,y_2]) \cap S ,\\
        \partconv_0 :=&\; (\{x_1\}\times [y_1,y_2]) \setminus \partconv_1 ,\\
        \partconv_3 :=&\; ((x_1,x_3)\times (y_1,y_2]) \cap S ,\\
        \partconv_2 :=&\; ((x_1,x_3)\times [y_1,y_2]) \setminus \partconv_3,\\
        \partconv_4 :=&\; \{x_3\} \times [y_1,y_2].\\
    \end{split}
\end{equation}
\begin{figure}[t]
    \centering
    \includegraphics[scale=0.7]{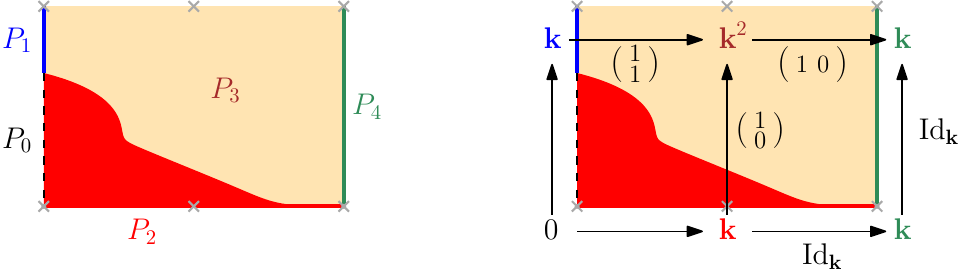}
    \caption{A graphical representation of the partition of the convex hull~$\hullconvex(G)$~(left), superimposed with its associated module~$M$~(right). The regions~$\partconv_0$, $\partconv_1$, $\partconv_2$, $\partconv_3$, $\partconv_4$ of~$\hullconvex(G)$ are represented respectively by the black dashed line segment, the blue segment, the red region (including the bottom red segment), the orange region and the green segment. The nodes of the grid~$G$ are represented as gray crosses.}
    \label{fig:counter-example_squares}
\end{figure}
\begin{table}[t]
  \centering
  \begin{tabular}{|c||*{5}{c|}}
      \hline
      \diagbox{$s\in$}{$t\in$} &~$\partconv_0$ &~$\partconv_1$ &~$\partconv_2$ &~$\partconv_3$ &~$\partconv_4$ \\ 
      \hline \hline 
     ~$\partconv_0$ &~$\leq$ &~$\leq$ &~$\square$ &~$\square$ &~$\square$ \\
      \hline
     ~$\partconv_1$ & &~$\leq$ & &~$\square$ &~$\square$ \\ 
      \hline
     ~$\partconv_2$ & & &~$\square$ &~$\square$ &~$\square$ \\ 
      \hline
     ~$\partconv_3$ & & & &~$\square$ &~$\square$\\ 
      \hline
     ~$\partconv_4$ & & & & &~$\leq$ \\
      \hline
  \end{tabular}
  \vspace{2ex}
  \caption{Summary of the comparability of the sets partitioning~$\hullconvex(G)$ defined in~\eqref{eq:def-partition}. For~$0 \leq i,j \leq 4$, an empty cell indicates that there is no~$s\in\partconv_i$ and~$t\in\partconv_j$ such that~$s\leq t$. On the contrary, a symbol~$\leq$ indicates that there is such~$s$ and~$t$, and a symbol~$\square$ refines this last case by indicating that such~$s$ and~$t$ can in addition (though it is not necessary) satisfy~$s_x < t_x$ and~$s_y < t_y$, or in other words that there exists a non-degenerate square of~$\poset$ with bottom-left corner in~$\partconv_i$ and top-right corner in~$\partconv_j$. The correctness of this table is clear.}
  \label{tab:partition}
\end{table}
See \Cref{fig:counter-example_squares}~(left) for a graphical representation of this partition, and \Cref{tab:partition} for a  summary of the comparability of the various sets in the  partition. Consider the subposet~$\sigmapos$ of~$\poset$ defined as follows:
\begin{equation}
    \label{eq:def-sigmapos}
    \sigmapos := \{(x_1,y_1), (x_1,y_2), (x_2,y_1), (x_2,y_2), (x_3,y_2)\} = G\setminus\{(x_3,y_1)\},
\end{equation}
whose Hasse diagram is:
\begin{equation*}
    \begin{tikzcd}[arrows=-stealth]
        (x_1,y_2) \arrow[r]           & (x_2,y_2) \arrow[r]            & (x_3,y_2) \\
        (x_1,y_1) \arrow[u] \arrow[r] & (x_2,y_1) \arrow[u] &  
    \end{tikzcd},
\end{equation*}
and define a persistence module~$\widetilde{M}$ over~$P$ by the following diagram:
\begin{equation*}
    \begin{tikzcd}[arrows=-stealth, ampersand replacement=\&]
        \field \arrow[r, "\mymatrix{1\\1}"]      \& \field^2 \arrow[r, "\mymatrix{1 & 0}"]          \& \field \\
        0 \arrow[u] \arrow[r] \& \field \arrow[u, "\mymatrix{1\\0}"] \&       
    \end{tikzcd}.
\end{equation*}
For any~$t\in\hullconvex(G)$, call~$\pi(t)$ the unique~$i\in\libr 0,4\ribr$ such that~$t\in\partconv_i$. A direct inspection --- eased by \Cref{tab:partition} --- yields that~$\pi : \hullconvex(G) \to \sigmapos$ is a poset morphism. Therefore, one can define the persistence module~$M$ over~$\hullconvex(G)$ as the pullback of~$\widetilde{M}$ along~$\pi$. In other words, for any~$s\leq t$ in~$\hullconvex(G)$, one has:
\begin{equation}
    \begin{split}
        M_t &:= \widetilde{M}_{\pi(t)},\\
        M(s\leq t) &:= \widetilde{M}(\pi(s)\leq\pi(t)).       
    \end{split}
\end{equation}
We consider in fact the extension of~$M$ to~$\poset$, still denoted by~$M$, with internal spaces set to be zero outside~$\hullconvex(G)$ and its internal morphisms to be the obvious ones. See \Cref{fig:counter-example_squares}~(right) for a graphical representation of~$M$.
\Cref{thm:negative-squares-hook} follows from the next proposition:
\begin{prop}
    \label{prop:counter-example-hook}
    The persistence module~$\basemod$ satisfies:
    \begin{enumerate}[label=(\roman*)]
        \item~$\basemod$ is not interval-decomposable;
        \item~$\basemod_{|Q} \in \decclass{\indecsupports_{|Q}}$ for any square~$Q\in \squares[\poset]$.
    \end{enumerate}
\end{prop}

\subsection{Proof of \Cref{prop:counter-example-hook}}
\label{sec:proof-hook}
    We first show that~$M$ is not interval-decomposable. Let~$\theta\in\End(M)$.  For~$s\leq t$ in~$\hullconvex(G)$ such that~$\pi(s) = \pi(t)$, i.e located in the same set~$\partconv_i$, we have that $M(s\leq t)=\identity_{M_t}$ by definition, so the naturality of~$\theta$ yields a commutative square:
    \begin{equation*}
        \begin{tikzcd}[arrows=-stealth]
            M_s \arrow[r, "\identity"] \arrow[d, "\theta_s"'] & M_t \arrow[d, "\theta_t"] \\
            M_s \arrow[r, "\identity"]                                  & M_t
            \end{tikzcd},
    \end{equation*}
    and~$\theta_s = \theta_t$ in that case. 
    Moreover, since~$M$ vanishes outside~$\hullconvex(G)$, so does~$\theta$. Thus, any~$\theta\in\End(M)$ is entirely determined by its values on the subposet~$P$ of~$\poset$ defined by~\eqref{eq:def-sigmapos}. Since~$M_{|\sigmapos}$ is isomorphic to~$\widetilde{M}$, which has an endomorphism ring isomorphic to~$\field$ by a direct verification, the persistence bimodule~$M$ itself has endomorphism ring isomorphic to~$\field$, which is local, hence~$M$ is indecomposable. Since it is not of pointwise dimension $0$ or~$1$ either, it is not interval-decomposable.
    
    \bigskip

    We now prove that the restriction~$M_{|Q}$ to any square~$Q$ of~$\poset$ belongs to~$\decclass{\indecsupports_{|Q}}$. By hypothesis, we have~$\indecsupports_{|Q}\supseteq \rectangles(Q)$, so for~$M_{|Q}$ to belong to~$\decclass{\indecsupports_{|Q}}$ it is sufficient (though not necessary) that~$M_{|Q}$ be rectangle-decomposable.  Note also that~$Q$ can be written as~$Q = \squarepos{s}{t}$, for two points~$s$ and~$t$ in~$\poset$. Since degenerate squares yield 1-parameter persistence modules, which are known to be interval-decomposable, we are left with the case where~$s_x < t_x$ and~$s_y < t_y$.

    Assume first that~$s\not\in\hullconvex(G)$ or~$t\not\in\hullconvex(G)$. We claim that~$M_{|\squarepos{s}{t}}$ is rectangle-decomposable in this case. Indeed, as any other pfd representation of the square, $M_{|\squarepos{s}{t}}$ is interval-decomposable, and it is then sufficient to prove that the interval summands of~$M_{|\squarepos{s}{t}}$ cannot be hooks. Assuming without loss of generality that~$t\not\in\hullconvex(G)$ (the other case being similar), we have that at least one point among~$(s_x,t_y)$ and~$(t_x,s_y)$ does not belong to~$\hullconvex(G)$, for otherwise we would have~$x_1 \leq t_x \leq x_3$ and~$y_1 \leq t_y \leq y_2$ hence~$t\in\hullconvex(G)$. Thus, $M_{|\squarepos{s}{t}}$ has at least two zero internal spaces, which implies that its interval summands cannot be hooks. This proves our claim, and so~$M_{|\squarepos{s}{t}}\in\decclass{\indecsupports_{|\squarepos{s}{t}}}$.

    Assume now that both~$s$ and~$t$ are in~$\hullconvex(G)$. Several cases are to be considered, corresponding to the cells containing the symbol~$\square$ in \Cref{tab:partition}:

    \medskip
    
    \paragraph{Case~$s\in\partconv_0$.}
    \begin{itemize}
        \item If~$(s_x,t_y)\in \partconv_0$, then~$M_s = M_{(s_x,t_y)}=0$ and no hooks can appear in the interval-decomposition of~$M_{|\squarepos{s}{t}}$, which is therefore rectangle-decomposable.
        \item If~$(s_x,t_y)\in \partconv_1$, then~$M_{|\squarepos{s}{t}}$ is of one of the three forms:
        \begin{equation*}
            \begin{gathered}
                \begin{array}{lll}
                    \squarediag{0}{\field}{\field}{\field^2}{}{}{\mymatrix{1\\1}}{\mymatrix{1\\0}}
                    \quad&
                    \squarediag{0}{\field^2}{\field}{\field^2}{}{}{\mymatrix{1\\1}}{\identity_{\field^2}}
                    \quad&
                    \squarediag{0}{\field}{\field}{\field}{}{}{\identity_\field}{\identity_\field}
                \end{array}
            \end{gathered}
        \end{equation*}
        which happen when~$t\in\partconv_3$ for the first two with~$(t_x,s_y)\in \partconv_2$ for the first and~$(t_x,s_y)\in\partconv_3$ for the second, and when~$t\in\partconv_4$ for the last one. The first one is rectangle-decomposable. For the last two, we have~$s\not\in S$ (since $s\in\partconv_0$) while the points~$(s_x,t_y)$, $(t_x,s_y)$ and~$t$ are in~$S$, hence~$\hook_2(\squarepos{s}{t}) \in \indecsupports_{|\squarepos{s}{t}}$. Since the modules are clearly interval-decomposable with interval summands being rectangles or top hooks, we do have that~$M_{|\squarepos{s}{t}}$ belongs to~$\decclass{\indecsupports_{|\squarepos{s}{t}}}$.
    \end{itemize}

    \medskip

    \paragraph{Case~$s\in\partconv_1$.}
    The restriction~$M_{|\squarepos{s}{t}}$ is then of one of the following forms:
    \begin{equation*}
        \begin{gathered}
            \begin{array}{ccc}
                \squarediag{\field}{\field^2}{\field}{\field^2}{\mymatrix{1\\1}}{\identity_\field}{\mymatrix{1\\1}}{\identity_{\field^2}}
                &\quad&
                \squarediag{\field}{\field}{\field}{\field}{\identity_\field}{\identity_\field}{\identity_\field}{\identity_\field}
            \end{array}
        \end{gathered}
    \end{equation*}
    which happen respectively when~$t\in\partconv_3$ for the first and~$t\in\partconv_4$ for the second. They are both clearly rectangle-decomposable.

    \medskip 

    \paragraph{Case~$s\in\partconv_2$.}
    \begin{itemize}
        \item If~$t\in\partconv_2$, then~$M_{|\squarepos{s}{t}}$ is of the form:
        \begin{equation*}
            \squarediag{\field}{\field}{\field}{\field}{\identity_\field}{\identity_\field}{\identity_\field}{\identity_\field},
        \end{equation*}
        which is clearly rectangle-decomposable.
        \item If~$t\in\partconv_3$, then~$M_{|\squarepos{s}{t}}$ is of one of the following forms:
        \medskip
        \begin{equation*}
            \begin{gathered}
                \begin{array}{c|c|c}
                    & (t_x,s_y)\in\partconv_2 & (t_x,s_y)\in\partconv_3 \\[2ex]
                    \hline &&\\
                    (s_x,t_y)\in\partconv_2 & \squarediag{\field}{\field}{\field}{\field^2}{\identity_\field}{\identity_\field}{\mymatrix{1\\0}}{\mymatrix{1\\0}} & \squarediag{\field}{\field^2}{\field}{\field^2}{\mymatrix{1\\0}}{\identity_\field}{\mymatrix{1\\0}}{\identity_{\field^2}}
                    \\[4.5em]
                    \hline
                    &&\\
                    (s_x,t_y)\in \partconv_3 & \squarediag{\field}{\field}{\field^2}{\field^2}{\identity_\field}{\mymatrix{1\\0}}{\identity_{\field^2}}{\mymatrix{1\\0}} & \squarediag{\field}{\field^2}{\field^2}{\field^2}{\mymatrix{1\\0}}{\mymatrix{1\\0}}{\identity_{\field^2}}{\identity_{\field^2}}\\
                \end{array}
            \end{gathered}
        \end{equation*}
        
        \medskip

        \noindent which are all rectangle-decomposable except when~$(s_x,t_y)$ and~$(t_x,s_y)$ are both in~$\partconv_3$ where a top hook summand appears. In that case, $s\not\in S$. In fact, $s\in\partconv_2$ with~$s_y \neq y_1$ since~$(t_x,s_y)\in\partconv_3$. Meanwhile, the points~$(s_x,t_y)$, $(t_x,s_y)$ and~$t$ are in~$S$. Hence, $\hook_2(\squarepos{s}{t}) \in \indecsupports_{|\squarepos{s}{t}}$ and we do have~$M_{\squarepos{s}{t}} \in \decclass{\indecsupports_{|\squarepos{s}{t}}}$.
        \item If~$t\in\partconv_4$, then~$M_{|\squarepos{s}{t}}$ is of one of the following forms:
        \begin{equation*}
            \begin{gathered}
                \begin{array}{ccc}
                    \squarediag{\field}{\field}{\field}{\field}{\identity_\field}{\identity_\field}{\identity_\field}{\identity_\field}
                    &\quad&
                    \squarediag{\field}{\field}{\field^2}{\field}{\identity_\field}{\mymatrix{1\\0}}{\mymatrix{1 & 0}}{\identity_\field}
                \end{array}
            \end{gathered}
        \end{equation*}
        which happen respectively when~$(s_x,t_y)\in \partconv_2$ and~$(s_x,t_y)\in \partconv_3$ and are both rectangle-decomposable.
    \end{itemize}

    \medskip

    \paragraph{Case~$s\in\partconv_3$.}
    Then~$M_{|\squarepos{s}{t}}$ is of one of the following forms:
    \begin{equation*}
        \begin{gathered}
            \begin{array}{ccc}
                \squarediag{\field^2}{\field^2}{\field^2}{\field^2}{\identity_{\field^2}}{\identity_{\field^2}}{\identity_{\field^2}}{\identity_{\field^2}}
                &\quad&
                \squarediag{\field^2}{\field}{\field^2}{\field}{\mymatrix{1 & 0}}{\identity_{\field^2}}{\mymatrix{1 & 0}}{\identity_\field}
            \end{array}
        \end{gathered}
    \end{equation*}
    which happen respectively when~$t\in\partconv_3$ and~$t\in\partconv_4$ and are both rectangle-decomposable.

    Thus, we have shown that~$M$ is indecomposable, while~$M_{|Q} \in\decclass{\indecsupports_{|Q}}$ for every square~$Q$ of~$\poset$. This concludes the proof.

\section{Homology pyramids and the strip}
\label{sec:example}
In this section we shall extend Theorem~\ref{thm:blc-decomposition} to persistence modules that are strongly exact on certain ``strip'' subsets of the plane. As we shall see, such persistence modules arise naturally in the context of TDA. 

First we recall the algebraic formulation of strong exactness. 

\begin{defi}[Strong exactness]\label{def:strong-exactness}
  A persistence module $\basemod$ over a subset $Q\subseteq \R^2$ is \emph{strongly exact} if for any $s\leq t$ in $Q$, such that $(s_x,t_y)$ and $(t_x,s_y)$ are also in $Q$, the following sequence is exact:
\begin{equation*}
    M_s \xrightarrow{\mor_s^{(t_x,s_y)}\oplus \mor_s^{(s_x,t_y)}} M_{(s_x,t_y)}\oplus M_{(t_x,s_y)} \xrightarrow{\mor_{(s_x,t_y)}^t - \mor_{(t_x,s_y)}^t} M_t.
\end{equation*}
\end{defi}
Define the \emph{strip} to be the subposet $\mathsf{St}\subset \R^2$ given by
\begin{equation*}\mathsf{St} = \{(x,y) \mid y\leq x+1 \text{ and } y \geq x-1\},\end{equation*}
or equivalently, all points situated on and between the two lines $y = x+1$ and $y=x-1$. Of importance is the \emph{pyramid} subposet $\pyramid$ defined by
 \begin{equation}
    \pyramid = \left\{ (x,y)\in(-1,1)^2,\ |x|+|y| \leq 1 \right\} \setminus \left\{(x,y) \in (-1,1)^2,\ x + y = -1\right\}. 
\end{equation}
See Figure \ref{fig:pyramid} for an illustration of $\pyramid$. 
Concatenating translations of pyramids we obtain our main object of study. Specifically, for a non-negative integer $k$, let $\pyramid(k) = \{ p\in \R^2 \mid p+(k,k) \in \pyramid\}$, and for $0\leq m \leq \infty$, 
\[
\mathsf{St}_m = \begin{cases} \bigcup_{0\leq k\leq m} \pyramid(k) & \text{if } m < \infty, \\
\bigcup_{0\leq k< \infty} \pyramid(k) & \text{if } m = \infty. \end{cases}
\]
Figure \ref{fig:mv-strip} illustrates the process of constructing $\mathsf{St}_\infty$. 
%

%
In Section~\ref{sec:proof-strip} we prove the following. 
\begin{thm}
\label{thm:strip}
Let $M\in\cat{Per}(\mathsf{St}_m)$ be pfd, strongly exact and trivial when restricted to indices on the boundary components $y=x+1$ and $y=x-1$. Then $M$ is interval-decomposable where each interval $I$ is of the form $I=R\cap \mathsf{St}_m$, where $R$ is a maximal rectangle supported on the interior of $\mathsf{St}$. 
\end{thm}
Here maximality is meant in the following sense: if $R\subseteq S\subset \R^2$ where $S$ is a rectangle supported in the interior of $\mathsf{St}$, then $R=S$. 


\begin{rk}
A proof of Theorem \ref{thm:strip} for $m=0$, under the assumption that the persistence modules are determined by a finite subset of $\pyramid$, first appeared in work by Bendich et al.~\cite{bendich2013homology}, following ideas from Carlsson et al.~\cite{carlsson2009zigzag}. An alternative proof in that setting was given by the authors of this paper in \cite{botnan2020rectangledecomposable} using Theorem~\ref{thm:finite-case}. 

A proof of Theorem \ref{thm:strip} for $m=\infty$, under the assumption that  $M$ is \emph{sequentially continuous}, can be found in recent work by Bauer et al.~\cite{bauer2021structure}. 
\end{rk}

\subsection{Motivation from TDA}
\label{sec:motivation-TDA}
The interest in block-decomposable modules originated in the study of levelset persistent homology. Specifically, for a continuous function $f\colon X\to \R$ one constructs the persistence module $M(f)\colon \{(x,y)\in\R^{\text{op}}\times \R, x<y\}\to \cat{Vec}$,
\begin{equation*}M(f)_{(x,y)} = H_i\left(\{ p \in X,  x < f(p) < y\}\right).\end{equation*}
It follows immediately from a simple Mayer--Vietoris calculation that $M(f)$ is strongly exact; see e.g., \cite{Cochoy2016}. Furthermore, assuming that $M(f)$ is pfd, it is interval-decomposable where each interval $I$ is of the form $I=B\cap \{(x,y)\in\R^{\text{op}}\times \R, x<y\}$, and where $B$ is a block in $\R^{\text{op}}\times \R$ \cite{botnan2018decomposition,Cochoy2016}; see Figure~\ref{fig:lzz-barcode} for an example of a real-valued function and the associated intervals. 

The domain of $M(f)$ can be extended by considering relative homology. To see this, consider the poset consisting of all pairs of open sets $(u,v)$ of the form,

\begin{align}
\label{eq:pairs-of-intervals}
 ((x,y), \emptyset) \qquad ((x, \infty), (y, \infty)) \qquad ((-\infty, y), (-\infty, x)) \qquad (\R, (-\infty, x)\cup (y, \infty)),
\end{align}
where $-\infty \leq x \leq  y \leq \infty$ are chosen such that $u\neq \emptyset$. The persistence module $M(f)$ is extended by defining $M(f)_{(u,v)} = H_i(f^{-1}(u), f^{-1}(v))$. Moreover, and as suggested in Figure~\ref{fig:pyramid}, $M(f)$ can be seen as a persistence module over $\pyramid$ after re-scaling, and reversing the horizontal arrow (see also Remark~\ref{rk:Kan-extensions}). We shall refer to the resulting persistence module $M(f)\in \catper{\pyramid}$  as the \emph{(continuous) homology pyramid (in dimension $i$)}. Note that the homology pyramid is trivial on the boundaries $y=x+1$ and $y=x-1$, as the corresponding vector spaces are given by the relative homology of preimages of pairs of the form $((x, \infty), (x, \infty))$ and $((-\infty, x), (-\infty, x))$, respectively. In particular, Theorem~\ref{thm:strip} for $m=0$ gives a structure theorem for pfd homology pyramids. 
\begin{figure}[h]
    \centering
    \includegraphics[scale=0.5]{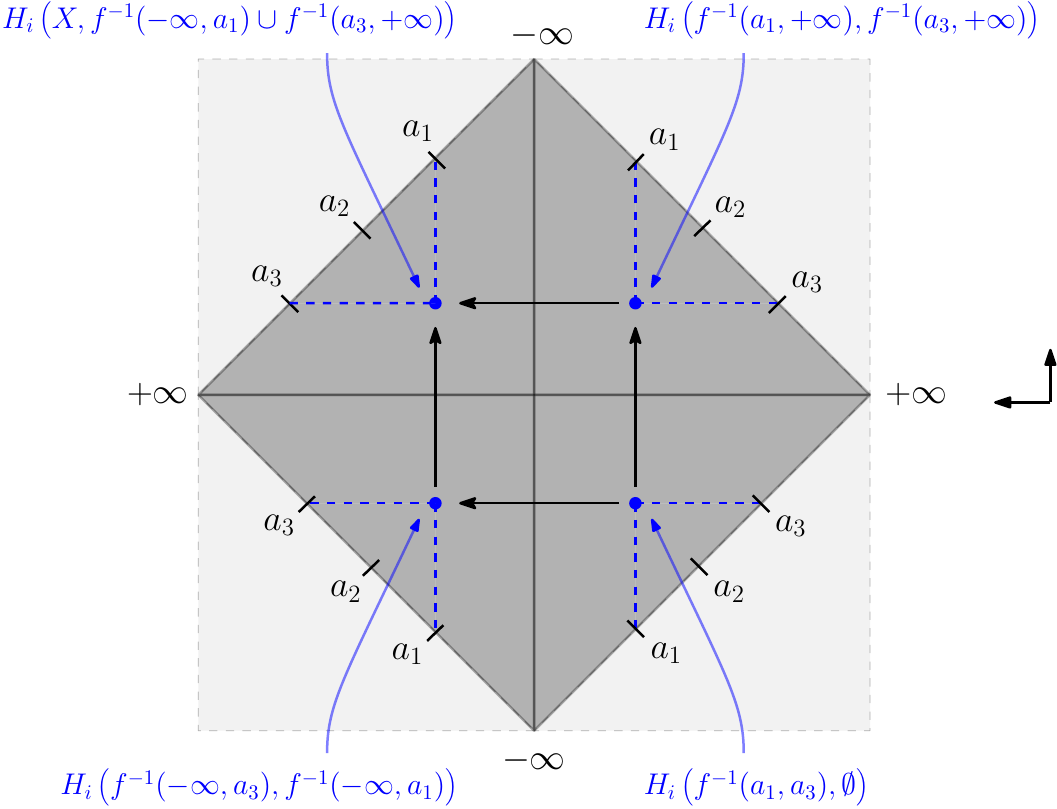}
    \caption{In dark grey, the homology pyramid of \cite[p.~6,7]{bendich2013homology} (up to rotation). Up to rescaling and reversing the horizontal axis, the underlying poset corresponds to the pyramid poset $\pyramid$. }
    \label{fig:pyramid}
\end{figure}

As was first noted by Carlsson et al. \cite{carlsson2009zigzag},  the boundary maps in the relative Mayer--Vietoris sequence can be used to connect homology pyramids of consecutive dimensions if one flips every other pyramid. In particular, $m+1$ consecutive homology pyramids assemble into a persistence module over $\mathsf{St}_m$. Figure~\ref{fig:mv-strip} illustrates this procedure and shows the interval-decomposition of $M(f)\in\catper{\mathsf{St}_\infty}$ associated to the function $f$ from Figure~\ref{fig:lzz-barcode}. 
\begin{remark}
We remark that care must be taken in verifying that the resulting persistence module $M(f)$ over $\mathsf{St}_\infty$ is well-defined. The reader should consult \cite{bauer2021structure} for a careful construction of $M(f)$, and a direct proof of an associated structure theorem. 
\end{remark}

\begin{figure}
\center
\centering
 \scalebox{.8}{
\begin{tikzpicture}[scale=.3]
\begin{scope}

\begin{scope}[rotate=90, xshift=-7cm, scale=1.4]
\fill[red!60,draw=black,even odd rule]
(0,0) to [out=90,in=180] (3,1.5) to [out=0,in=90] (9,-.5) to [out=270,in=270] (1,-1) to [out=90,in=270] (3,0) to [out=90,in=270] (0,0)
(4,-.2) to [out=90, in=90] (7,-.2) to [out=270, in=270] (4,-.2);
\draw[very thick, ->] (-1,4) -- (10,4);
\def\x{-5}
\def\y{5}
\draw[dotted] (0,\y) -- (0,\x);
\draw[dotted] (1,\y) -- (1,\x);
\draw[dotted] (3,\y) -- (3,\x);
\draw[dotted] (4,\y) -- (4,\x);

\draw[dotted] (7,\y) -- (7,\x);
\draw[dotted] (9,\y) -- (9,\x);

\node[left] at (0,\y) {$a_1$};
\node[left] at (1,\y) {$a_2$};
\node[left] at (3,\y) {$a_3$};
\node[left] at (4,\y) {$a_4$};
\node[left] at (7,\y) {$a_5$};
\node[left] at (9,\y) {$a_6$};

\end{scope}

\begin{scope}[xshift=20cm, scale=0.5, yshift=3cm]
\node at (10,5) {$H_0$};

\draw[dashed] (-3,-3) -- (13,13); 

\node[right=3pt] (x1) at (0,0) {\small $a_1$};
\node[right=3pt] (x2) at (1,1) {\small $a_2$};
\node[right=3pt] (x3) at (3,3) {\small $a_3$};
\node[right=3pt] (x4) at (4,4) {\small $a_4$};
\node[right=3pt] (x5) at (7,7) {\small $a_5$};
\node[right=3pt] (x6) at (9,9) {\small $a_6$};

\draw[draw=none, fill=black, fill opacity=0.15] (-5,0) -- (0,0) -- (9,9) -- (9,14) -- (-5,14) -- cycle;
\draw[dashed] (-5,0) -- (0,0);
\draw[dashed] (9,9) -- (9,14);

\draw[draw=none, fill=black, fill opacity=0.15]  (-5,1) -- (1,1) -- (3,3) -- (-5,3) -- cycle;
\draw[dashed] (-5,1) -- (1,1);
\draw (3,3) -- (-5,3);
\draw[->] (-6,8) -- (-6,10);
\draw[->] (-6,8) -- (-8,8);
\draw[dashed, fill=black, fill opacity=0.15]  (4,4) -- (7,7) -- (4,7) -- cycle;
\draw (4,4) -- (4,7) -- (7,7);

\coordinate (y1) at (0,0);
\coordinate (y2) at (1,1);
\coordinate (y3) at (3,3);
\coordinate (y4) at (4,4);
\coordinate (y5) at (7,7);
\coordinate (y6) at (9,9);

\foreach \a in {y1, y2, y3, y4, y5, y6}
	\draw[fill=black] (\a) circle (1ex); 
\end{scope}

\begin{scope}[xshift=20cm, scale=0.55, yshift=-17cm]

\draw[dashed] (-3,-3) -- (13,13); 

\node[right=3pt] (x1) at (0,0) {\small $a_1$};
\node[right=3pt] (x2) at (1,1) {\small $a_2$};
\node[right=3pt] (x3) at (3,3) {\small $a_3$};
\node[right=3pt] (x4) at (4,4) {\small $a_4$};
\node[right=3pt] (x5) at (7,7) {\small $a_5$};
\node[right=3pt] (x6) at (9,9) {\small $a_6$};
\draw[->] (-6,8) -- (-6,10);
\draw[->] (-6,8) -- (-8,8);

\draw[draw=none, fill=black, fill opacity=0.15] (4,7) -- (4,14) -- (-5,14) -- (-5, 7) -- cycle;
\draw[dashed] (-5,7) -- (4,7) -- (4,14);
\node at (10,5) {$H_1$};
\coordinate (y1) at (0,0);
\coordinate (y2) at (1,1);
\coordinate (y3) at (3,3);
\coordinate (y4) at (4,4);
\coordinate (y5) at (7,7);
\coordinate (y6) at (9,9);

\foreach \a in {y1, y2, y3, y4, y5, y6}
	\draw[fill=black] (\a) circle (1ex); 
\end{scope}

\end{scope}
\end{tikzpicture}}
\caption{A function $f$ on a topological space (left) and the associated interval-decompositions in levelset persistent homology in dimensions 0 and 1 (right).  }
\label{fig:lzz-barcode}
\end{figure}

\begin{figure}
\centering\centering
 \scalebox{.8}{
\begin{tikzpicture}[scale=0.2, rotate=0, yscale=1, xscale=-1]
\begin{scope}

\begin{scope}[scale=0.8]
\def\x{25}
\def\y{5}
%

\coordinate (y1) at (0,0);
\coordinate (y2) at (1,1);
\coordinate (y3) at (3,3);
\coordinate (y4) at (4,4);
\coordinate (y5) at (7,7);
\coordinate (y6) at (9,9);

\coordinate (w1) at ($ (y1) + (\x,-\x)$);
\coordinate (w2) at ($ (y2) + (\x,-\x)$);
\coordinate (w3) at ($ (y3) + (\x,-\x)$);
\coordinate (w4) at ($ (y4) + (\x,-\x)$);
\coordinate (w5) at ($ (y5) + (\x,-\x)$);
\coordinate (w6) at ($ (y6) + (\x,-\x)$);

\coordinate (x1) at (-15+\x,17-\x);
\coordinate (x2) at (-13+\x,19-\x);
\coordinate (x3) at (-10+\x,22-\x);
\coordinate (x4) at (-9+\x,23-\x);
\coordinate (x5) at (-7+\x,25-\x);
\coordinate (x6) at (-6+\x,26-\x);

\coordinate (ww1) at ($ (x1) + (\x,-\x)$);
\coordinate (ww6) at ($ (x6) + (\x,-\x)$);

\coordinate (z1) at (-15,17);
\coordinate (z2) at (-13,19);
\coordinate (z3) at (-10,22);
\coordinate (z4) at (-9,23);
\coordinate (z5) at (-7,25);
\coordinate (z6) at (-6,26);

\foreach \a in {z1, z2,z3,z4,z5,z6, y1,y2,y3,y4,y5,y6, w1,w2,w3,w4,w5,w6,x1,x2,x3,x4,x5,x6}
	\draw[fill=black] (\a) circle (1ex);

\node[right=1pt] at (x1) {\tiny $a_6$};
\node[right=1pt] at (x2) {\tiny $a_5$};
\node[right=1pt] at (x3)  {\tiny $a_4$};
\node[right=1pt] at (x4) {\tiny $a_3$};
\node[right=1pt] at (x5) {\tiny $a_2$};
\node[right=1pt] at (x6) {\tiny $a_1$};

\node[left=2pt] at (z1) {\tiny $a_6$};
\node[left=2pt] at (z2) {\tiny $a_5$};
\node[left=2pt] at (z3)  {\tiny $a_4$};
\node[left=2pt] at (z4) {\tiny $a_3$};
\node[left=2pt] at (z5) {\tiny $a_2$};
\node[left=2pt] at (z6) {\tiny $a_1$};

\node[left=2pt] at (y1) {\tiny $a_1$};
\node[left=2pt] at (y2) {\tiny $a_2$};
\node[left=2pt] at (y3)  {\tiny $a_3$};
\node[left=2pt] at (y4) {\tiny $a_4$};
\node[left=2pt] at (y5) {\tiny $a_5$};
\node[left=2pt] at (y6) {\tiny $a_6$};

\node[right=1pt] at (w1) {\tiny $a_1$};
\node[right=1pt] at (w2) {\tiny $a_2$};
\node[right=1pt] at (w3)  {\tiny $a_3$};
\node[right=1pt] at (w4) {\tiny $a_4$};
\node[right=1pt] at (w5) {\tiny $a_5$};
\node[right=1pt] at (w6) {\tiny $a_6$};

\draw (13,13) -- (-3,29) -- (-19,13) -- (-3,-3);
\draw (13+\x,13-\x) -- (-3+\x,29-\x) -- (-19+\x,13-\x) -- (-3+\x,-3-\x);
\draw[dashed] (-3,-3) -- (13,13);
\draw[dashed] (-3+\x,-3-\x) -- (13+\x,13-\x);

\draw[dashed, draw=none, fill=black, fill opacity=0.15] (-6,0) -- (0,0) -- (9,9) -- (9,17) -- (-6,17) -- cycle;
\draw[dashed] (-6,0) -- (0,0) -- (9,9) -- (9,17);
\draw (9,17) -- (-6,17) -- (-6,0);

\draw[draw=none, fill=black, fill opacity=0.15]  (-7,1) -- (1,1) -- (3,3) -- (-7,3) -- cycle;
\draw[dashed] (-7,1) -- (1,1);
\draw (3,3) -- (-7,3) -- (-7,1);

\draw[fill=black, fill opacity=0.15]  (4,4) -- (4,7) -- (7,7);

\draw[draw=none, fill=black, fill opacity=0.2] (4+\x,7-\x) -- (4+\x,22-\x) -- (-10+\x,22-\x) -- (-13+\x,19-\x) -- (-13+\x,7-\x)--  cycle;
\draw[dashed] (-13+\x,7-\x) -- (4+\x,7-\x) -- (4+\x,22-\x);
\draw (4+\x,22-\x) -- (-10+\x,22-\x);
\draw (-13+\x,7-\x) -- (-13+\x,19-\x);

\draw[dashed] (-7+\x,25-\x) -- (-7+\x,3-\x)--(-9+\x,3-\x);
\draw (-9+\x,23-\x) -- (-9+\x,3-\x);
\draw[draw=none, fill=black, fill opacity=0.2] (-7+\x,25-\x) -- (-7+\x,3-\x) -- (-9+\x,3-\x) -- (-9+\x,23-\x) -- cycle;

\draw[dashed] (-6+\x,26-\x) -- (-6+\x,17-\x) -- (-15+\x, 17-\x);
\draw[dashed, draw=none, fill=black, fill opacity=0.2] (-6+\x,26-\x) -- (-6+\x,17-\x) -- (-15+\x, 17-\x); -- cycle;

\draw (-3,-3) -- (-3,29);
\draw (-19,13) -- (13,13);

\draw (-3+\x,-3-\x) -- (-3+\x,29-\x);
\draw (-19+\x,13-\x) -- (13+\x,13-\x);

\node at (-10,0) {$H_0$};
\node at (-10+\x,-\x) {$H_1$};

\draw [thick,->] (x6) to [out=150,in=-120] (y1);
\draw [thick, ->] (x1) to [out=150,in=30] (y6);

\draw [dashed,thick,->] (ww6) to [out=150,in=-120] (w1);
\draw [dashed, thick, ->] (ww1) to [out=150,in=30] (w6);

\end{scope}

\begin{scope}[xshift=25cm,yshift=15cm, scale=0.8]

\node[rotate=-0] at (12,18) {$M(f)$};

\def\x{16}
\def\y{5}
\draw (-3,-3) -- (13,13) -- (-3,29) -- (-19,13) -- cycle;
\draw (-3+\x,-3-\x) -- (13+\x,13-\x) -- (-3+\x,29-\x) -- (-19+\x,13-\x) -- cycle;
\draw (-3+\x, -3-\x) -- (-3+1.5*\x, -3-1.5*\x);
\draw (13+\x, 13-\x) -- (13+1.5*\x, 13-1.5*\x);

\coordinate (y1) at (0,0);
\coordinate (y2) at (1,1);
\coordinate (y3) at (3,3);
\coordinate (y4) at (4,4);
\coordinate (y5) at (7,7);
\coordinate (y6) at (9,9);

\coordinate (x1) at (-15+2*\x,17-2*\x);
\coordinate (x2) at (-13+2*\x,19-2*\x);
\coordinate (x3) at (-10+2*\x,22-2*\x);
\coordinate (x4) at (-9+2*\x,23-2*\x);
\coordinate (x5) at (-7+2*\x,25-2*\x);
\coordinate (x6) at (-6+2*\x,26-2*\x);

\coordinate (z1) at (-15,17);
\coordinate (z2) at (-13,19);
\coordinate (z3) at (-10,22);
\coordinate (z4) at (-9,23);
\coordinate (z5) at (-7,25);
\coordinate (z6) at (-6,26);

\foreach \a in {z1, z2,z3,z4,z5,z6, y1,y2,y3,y4,y5,y6, x1,x2,x3,x4,x5,x6}
	\draw[fill=black] (\a) circle (1ex);

\node[left=1pt] at (x1) {\tiny $a_6$};
\node[left=1pt] at (x2) {\tiny  $a_5$};
\node[left=1pt] at (x3)  {\tiny $a_4$};
\node[left=1pt] at (x4) {\tiny $a_3$};
\node[left=1pt] at (x5) {\tiny $a_2$};
\node[left=1pt] at (x6) {\tiny $a_1$};

\node[left=1pt] at (z1) {\tiny $a_6$};
\node[left=1pt] at (z2) {\tiny $a_5$};
\node[left=1pt] at (z3)  {\tiny $a_4$};
\node[left=1pt] at (z4) {\tiny $a_3$};
\node[left=1pt] at (z5) {\tiny $a_2$};
\node[left= 1pt] at (z6) {\tiny $a_1$};

\draw[dashed, draw=none, fill=black, fill opacity=0.15] (-6,0) -- (9,0) -- (9,17) -- (-6,17) -- cycle;
\draw[dashed] (-6,0) -- (9,0) -- (9,17);
\draw (9,17) -- (-6,17) -- (-6,0);

\draw[dashed, draw=none, fill=black, fill opacity=0.15]  (-7,1) -- (23,1) -- (23,3) -- (-7,3) -- cycle;
\draw[dashed] (-7,1) -- (23,1) -- (23,3);
\draw (-7,1) -- (-7,3) -- (23,3);

\draw[draw=none, fill=black, fill opacity=0.15]  (19,7) -- (4,7) -- (4,-10) -- (19,-10);
\draw[dashed] (4,-10) -- (19,-10) -- (19,7);
\draw (19,7) -- (4,7) -- (4,-10);

\end{scope}

\end{scope}
\end{tikzpicture}}
\caption{The homology pyramids and the corresponding $M(f)\in\catper{\mathsf{St}_\infty}$ of the function in Figure~\ref{fig:lzz-barcode}. }
\label{fig:mv-strip}

\end{figure}

\subsection{Proof of Theorem~\ref{thm:strip}}
\label{sec:proof-strip}
We now return to the proof of Theorem~\ref{thm:strip}. This lemma is fundamental. 
\begin{lem}
\label{lem:strip}
Let $\basemod$ be a strongly exact and pfd persistence module over $\mathsf{St}$, such that $\basemod$ is trivial when restricted to the boundary components $y=x+1$ and $y=x-1$. Then $\basemod$ is interval-decomposable and each interval is a maximal rectangle on the interior of $\mathsf{St}$.
\end{lem}
\begin{proof}
Since $\basemod$ is trivial on the boundary components, we may extend $\basemod$ to a persistence module over $\R^2$ by defining $\basemod_p = 0$ for all $p$ not contained in $\mathsf{St}$. By virtue of Theorem~\ref{thm:positive-rec-squares}, it suffices to show that the extension to $\R^2$ is weakly exact. This is not hard to see:
consider $s$ and $t$ as in Definition~\ref{def:weak-exactness}. Then, by assumption, $\basemod$ can only fail to be weakly exact if $(s_x, t_y)$ or $(t_x,s_y)$ is not contained in $\mathsf{St}$. However, in that case,
  \begin{equation*}
    \begin{split}
        0 = \Ima\mor_s^t &= \Ima\mor_{(t_x,s_y)}^t \cap \Ima\mor_{(s_x,t_y)}^t, \\
        \basemod_s = \Ker\mor_s^t &= \Ker\mor_s^{(t_x,s_y)} + \Ker\mor_s^{(s_x,t_y)}.
    \end{split}
    \end{equation*}
Hence, $\basemod$ is weakly exact. Furthermore, since the extension of $M$ is trivial outside of $\mathsf{St}$, any rectangle $R$ in the decomposition of $\basemod$ must be necessarily be supported on a subset of $\mathsf{St}$. 

We now show that $R$ is maximal. Write $R=\langle a,b\rangle\times \langle c,d\rangle$ where $\langle a,b\rangle$ denotes an interval in $\R$ with left and right endpoints given by $a$ and $b$, respectively, e.g., $[a,b)$ or $(a,b)$.  Assume that the point $(a,c)\in \R^2$ is not on the boundary of $\mathsf{St}$. Then, we can choose $s$ and $t$ such that the square $A=\{s_x, t_x\}\times \{s_y, t_y\}$ is contained in the interior of $\mathsf{St}$ and $R\cap A=\{(s_x, t_y)\}$. In particular, the sequence in Definition~\ref{def:strong-exactness} associated to $A$ has a summand of the form $0\to \field\to 0$, contradicting that the sequence is exact. Symmetrically, the point $(b,d)\in \R^2$ must be contained in the other boundary component. 

\end{proof}

We now describe a procedure to extend $M\in \catper{\mathsf{St}_m}$ to a persistence module over $\mathsf{St}$ by means of left and right Kan extensions. In order to ensure that the extension is trivial on the boundary components we shall consider the poset $\mathsf{St}_m^b$ given by the union of $\mathsf{St}_m$ with the two boundary components of $\mathsf{St}$. The extension works as follows:
\begin{enumerate}
	\item Extend $\basemod\in \catper{\mathsf{St}_m}$ to $ \basemod\in \catper{\mathsf{St}^b_m}$ by defining the module to be trivial on the boundary components.
	\item Extend to $\Ran_\phi(\basemod)\in \catper{\mathsf{St}^b_\infty}$ by a right Kan Extension along the inclusion ~$\phi:\mathsf{St}^b_m \hookrightarrow \mathsf{St}^b_\infty$.
    \item Extend to $\boxper := \Lan_\psi(\Ran_\phi(\basemod))\in\catper{\mathsf{St}}$ by means of a left Kan extension along~$\psi:\mathsf{St}^b_\infty \hookrightarrow \mathsf{St}$.
\end{enumerate}
The fact that $\boxper|_{\mathsf{St}_m} \cong M$ follows from \cite[Corollary 3.6.9]{riehl2017category} as both of the Kan extensions are computed pointwise along a full functor. 

The outline for proving Theorem \ref{thm:strip} is as follows. First, we show in Lemmas \ref{lem:strip-pfd} and \ref{lem:strip-split} that $M$ decomposes as $M\cong N\oplus N_1\oplus N_2$, where $\widetilde{N}$ is pfd and strongly-exact, and $N_1$ and $N_2$ decompose as stated in Theorem \ref{thm:strip}. Then, we apply Lemma~\ref{lem:strip} to $\widetilde{N}$. The result then follows from restricting to $\mathsf{St}_m$ and using the observation that $N\cong \widetilde{N}|_{\mathsf{St}_m}$.

In the proofs, we shall make use of the following notation, 
 \begin{equation*}
    \begin{split}
        \pyramidhigh &:= \left\{ (x,y)\in(0,1)^2,\ x + y \leq  1 \right\}, \\ 
        \pyramidlow &:= \left\{(x,y)\in (-m-1,-m)^2, x+y> -1-2m\right\},
    \end{split}
 \end{equation*}

and we shall denote the union of each of these subsets with the boundary components of $\mathsf{St}$ by $\pyramidhigh^b$ and $\pyramidlow^b$, respectively.

\begin{rk}
    \label{rk:duality}
    Dualizing each internal space and each internal morphism of a persistence module $M$ over $\pyramid$ yields a persistence module $\dual M$ over $\pyramid^\text{op}$, the opposite category of $\pyramid$. Dualization defines a contravariant functor $\dual : \cat{Per}(\pyramid) \to \cat{Per}(\pyramid^\text{op})$ which sends strongly exact modules to strongly exact modules, and which satisfies $\dual \dual M \simeq M$ for any pfd persistence module $M$.
\end{rk}

\begin{rk}
    \label{rk:Kan-extensions}
    In the following we shall make use of results from \cite{botnan2018decomposition} and \cite{Cochoy2016}. The results were originally formulated for persistence modules over~$\R^2$, over~$T := \{(x,y)\in\R^2,\, x+y>0\}$ and over $\overline{T} := \{(x,y)\in\R^2,\, x+y\geq 0\}$. These results do however apply verbatim in the settings of $\pyramidhigh$ and $\pyramidlow$, as can be seen from the following two poset isomorphisms and Remark~\ref{rk:duality},
        \begin{equation*}
            \begin{split}
                (x,y)\in\pyramidlow \mapsto &\left(\tan\left(\frac{\pi}{2}(x+1+m)\right),\tan\left(\frac{\pi}{2}(y+1+m)\right)\right)\in T\\
                (x,y)\in\pyramidhigh^\text{op} \mapsto &\left(\tan\left(\frac{\pi}{2}(-x+1)\right),\tan\left(\frac{\pi}{2}(-y+1)\right)\right)\in\overline{T}
            \end{split}
    \end{equation*}
\end{rk}

The following is a slight reformulation of two results from \cite[Sec.~5.2]{botnan2018decomposition}.
\begin{theorem}
\label{thm:cite-blocks}
Let $M$ be a strongly exact and pfd persistence module over $\pyramidhigh$ or $\pyramidlow$. Then $M$ is interval-decomposable, and each interval $B$ is of the form $B=R\cap \pyramidhigh$ or $I=R\cap \pyramidlow$, respectively, where $R$ is a maximal rectangle in the interior of $\mathsf{St}$.
\end{theorem}
We shall refer to such intervals as \emph{blocks}. If $B=R\cap \pyramidhigh$, and there exists a $p\in \pyramidhigh$ such that $p<q$ for all $q\in B$, then $B$ is a \emph{birth quadrant}. Dually, if $B=R\cap \pyramidlow$ and there exists $q\in B$ such that $p<q$ for all $q\in B$, then $B$ is a \emph{death quadrant}. For instance, the triangular shaped blocks in the homology pyramids in dimension 0 and 1 shown in Figure~\ref{fig:mv-strip} are death and birth quadrants, respectively. The following is an immediate corollary of Theorem~\ref{thm:cite-blocks}.

\begin{cor}
Let $\basemod\in\catper{\mathsf{St}_m}$ be as in the statement of Theorem~\ref{thm:strip}. Then the restriction of $M$ to $\pyramidlow$, denoted by $M|_{\pyramidlow}$, decomposes as 
    \begin{equation}
        \label{eq:blc-dec-M-general}
        \basemod_{|\pyramidlow} \simeq \bigoplus_{B\in\blockcode_1}\field_B,
    \end{equation}
where $\blockcode_1$ is a multiset of blocks in $\pyramidlow$. Dually, the restriction of $M$ to $ \pyramidhigh$ decomposes as 
    \begin{equation}
        \label{eq:blc-dec-M-general2}
        \basemod_{|\pyramidhigh} \simeq \bigoplus_{B\in\blockcode_2}\field_B,
    \end{equation}
where $\blockcode_2$ is a multiset of blocks in $\pyramidhigh$.
\label{cor:decomp}
\end{cor}

\begin{rk}
    \label{rk:description-Kan-extensions}
    Let $\phi:\mathsf{St}^b_m \hookrightarrow \mathsf{St}^b_\infty$, and $\psi:\mathsf{St}^b_\infty \hookrightarrow \mathsf{St}$ denote the inclusions introduced above. The inclusion of posets $\pyramidlow^b \hookrightarrow \mathsf{St}^b_m$ is initial in the sense of \cite[Sec.~IX.3]{MacLane1998}, and therefore \cite[Sec.~IX.3, Thm.~1]{MacLane1998}, we have the following natural isomorphism for all $p\in \mathsf{St}^b_\infty - (\mathsf{St}^b_m - \pyramidlow^b)$,
    \begin{equation*}
        \left(\Ran_{\phi}\basemod\right)_p \simeq \left(\Ran_{\phi|_{\pyramidlow^b}}\basemod|_{\pyramidlow^b}\right)_p
    \end{equation*}
    Similarly, one has:
    \begin{equation*}
        \boxper_p\simeq \left(\Lan_{\psi}(\Ran_{\phi}M)\right)_p \simeq \left(\Lan_{\psi|_{\pyramidhigh^b}}(\Ran_{\phi}M)|_{\pyramidhigh^b}\right)_p,
    \end{equation*}
    for all $p\in \mathsf{St}-(\mathsf{St}^b_\infty-\pyramidhigh^b)$. %
\end{rk}

\begin{lem}
\label{lem:strip-pfd}
Let $\basemod\in\catper{\mathsf{St}_m}$ be as in the statement of Theorem ~\ref{thm:strip}, and let $\blockcode_1$ and $\blockcode_2$ be as in Corollary~\ref{cor:decomp}. 
\begin{enumerate}
\item $\boxper$ is strongly exact.
\item If $\blockcode_1$ contains no death quadrant, and $\blockcode_2$ contains no birth quadrant, then $\boxper$ is pfd. 
\end{enumerate}
\end{lem}

\begin{proof}
First we prove (1). Following the first step in the extension procedure outlined above, we obtain $M\in\catper{\mathsf{St}_m^b}$ by adding 0 vector spaces. To simplify notation, we shall let $D$ and $U$ denote the sets $D=\mathsf{St}_\infty - (\mathsf{St}_m - \pyramidlow)$, $U=(\mathsf{St}-(\mathsf{St}_\infty-\pyramidhigh))$ and $N=\basemod|_{\pyramidlow^b}$. Then 
    \begin{equation*}
        N \simeq \bigoplus_{B\in\blockcode_1}\field_{B}.
    \end{equation*}
    where each block $B\in \blockcode_1$ is of the form $B=R\cap\pyramidlow$ where $R$ is a maximal rectangle in the interior of $\mathsf{St}$. It is not hard to see that extending $\field_{B}$ ``to the left'' using a right Kan extension will recover $\field_{R}$ on $\mathsf{St}_\infty - \mathsf{St}_m$. That is, we have the following natural isomorphism 
   \begin{equation*}\Ran_{\phi|_{\pyramidlow^b}} (\field_B)_p \cong (\field_R)_p,\end{equation*} 
   for all $p\in D$. In particular, since $\field_R$ is strongly exact on $\mathsf{St}$ (and therefore on $D$), it follows that $\Ran_{\phi|_{\pyramidlow^b}} (\field_B)$ is strongly exact on $D$. Hence, 
   \begin{equation*} (\Ran_\phi M)|_D \simeq \left(\Ran_{\phi|_{\pyramidlow^b}} N\right)|_D\simeq \left(\bigoplus_{B\in\blockcode_1}\Ran_{\phi|_{\pyramidlow^b}}\field_B\right)|_D \end{equation*}   
    is strongly exact on $D$. Here the first isomorphism follows from   Remark~\ref{rk:description-Kan-extensions}, and the second isomorphism follows from the fact that right Kan extensions commute with direct products, and that direct products and direct sums coincide when working in the pfd setting; see e.g. \cite[Rk.~2.16]{Botnan2016} for more details. 
    
Let $N'$ denote the restriction of $\Ran_\phi M$ to $\pyramidhigh^b$. From Corollary~\ref{cor:decomp} we now have that
    \begin{equation*}
        N' \simeq \bigoplus_{B'\in\blockcode_2}\field_{B'}.
    \end{equation*}
    where $\blockcode_2$ is a multiset of blocks of $\pyramidhigh$. It follows from the second part of Remark~\ref{rk:description-Kan-extensions}, that
    $\boxper|_U \cong (\Lan_{\psi|_{\pyramidhigh^b}} N')|_U$. Showing that the latter is strongly exact is dual to the first part of this proof. In conclusion, we have that $\boxper$ is strongly exact when restricted to $D$, $\mathsf{St}_m$ and $U$. To see that this implies strong exactness on the whole of $\mathsf{St}$, let $s$ and $t$ be as in Definition~\ref{def:strong-exactness}, and assume that $s\in D-D\cap\mathsf{St}_m$ and $t\in \mathsf{St}_m - D\cap\mathsf{St}_m$. Then we can find $p\in D\cap\mathsf{St}_m=\pyramidlow$, such that $s\leq p\leq t$, and such that each of the smaller squares in the following diagram is strongly exact,
        \begin{equation}
        \label{eq:diag-elem-strongly-exact}
        \begin{tikzcd}
        {\boxper_{(s_x,t_y)}} \arrow[r] & {\boxper_{(p_x,t_y)}} \arrow[r]           & \boxper_{t}\\
        {\boxper_{(s_x,p_y)}} \arrow[r] \arrow[u]  & {\boxper_{p}} \arrow[r]   \arrow[u]          & \boxper_{(t_x, p_y)} \arrow[u] \\
        \boxper_{s} \arrow[r] \arrow[u] & {\boxper_{(p_x,s_y)}} \arrow[u] \arrow[r] & {\boxper_{(t_x,s_y)}} \arrow[u]
        \end{tikzcd}.
    \end{equation}
A simple diagram chase shows that the outer square is strongly exact. Hence, $\boxper$ is strongly exact on $\mathsf{St}_\infty = D\cup \mathsf{St}_m$. A similar argument shows that $\boxper$ is strongly exact on $\mathsf{St}_\infty\cup U = \mathsf{St}$.

Now we prove (2). Since $\boxper|_{\mathsf{St}_m} \cong \basemod|_{\mathsf{St}_m}$, it suffices to prove that $\boxper$ is pfd over $D-D\cap \mathsf{St}_m$ and $U-U\cap \mathsf{St}_m$. We prove the former, the second is dual. Assume that $\Ran_{\phi_{\pyramidlow^b}}N$ is not pfd at $p\in D-D\cap \mathsf{St}_m$. Then, there must exist an infinite family of blocks $\{B_\lambda = R_\lambda\cap\pyramidlow\}\subset \blockcode_1$ such that $p\in R_\lambda$ for all $\lambda$. However, by assumption, $R_\lambda$ is not bounded from above in $\pyramidlow$. In other words, $R_\lambda$ and thus $B_\lambda$ must contain at least one of the two line segments
\begin{equation*}\left([p_x,-m)\times\{p_y\}\right) \cap\pyramidlow\qquad\text{and}\qquad \left(\{p_x\}\times[p_y,-m)\right)\cap\pyramidlow.\end{equation*}
Hence, either there exists an infinite number of blocks $B_\lambda\in\blockcode_1$ such that \begin{equation*}\left([p_x,-m)\times\{p_y\}\right) \cap\pyramidlow \subseteq B_\lambda,\end{equation*} or there exists an infinite number of blocks $B_\lambda\in\blockcode_1$ such that \begin{equation*}\left(\{p_x\}\times[p_y,0)\right)\cap\pyramidlow\subseteq B_\lambda.\end{equation*} 
We conclude that $\dim N_q = \infty$ for all $q\in \left([p_x, -m)\times \{p_y\}\right) \cap\pyramidlow$ or $\dim N_q = \infty$ for all $q\in\left(\{p_x\}\times[p_y,-m)\right)\cap\pyramidlow$. This contradicts the assumption that $N$ is pfd. 
\end{proof}

\begin{lem}
\label{lem:strip-split}
Let $M\in\catper{\mathsf{St}_m}$ be as in the statement of Theorem~\ref{thm:strip}, and let $\blockcode_1$ and $\blockcode_2$ be as in Corollary~\ref{cor:decomp}. Then, 
\begin{equation*}        M \simeq N \oplus \bigoplus_{B\in\sbirth} \field_B \oplus \bigoplus_{B\in\sdeath} \field_B,\end{equation*}
where $\sbirth\subseteq\blockcode_1$ is the collection of all birth quadrants, and $\sdeath\subseteq\blockcode_2$ is the collection of all death quadrants, and $N$ contains no summand of the form $\field_B$ where $B$ is a birth quadrant in $\pyramidhigh$ or a death quadrant in  $\pyramidlow$.
\end{lem}
\begin{proof}
Recall that $B$ is a death quadrant in $\pyramidlow$  if and only if its support is bounded from above by some $x\in \pyramidlow$. Dually, a birth quadrant $B$ in $\pyramidhigh$ is a  birth quadrant if and only if it is bounded from below by some $x\in \pyramidhigh$. In either case, there is enough space around their supports to extend them to summands of $M$. In particular, if $B\in \sdeath$, then the inclusions and projection maps
\begin{equation*}\field_B \xrightarrow{i_1}  \basemod_{|\pyramidlow}  \xrightarrow{j_1} \field_B\end{equation*}
such that $j_1 \circ i_1 = {\rm id}_{\field_B}$, extend to
\begin{equation*}\field_B \xrightarrow{i}  M|_{\mathsf{St}_m}  \xrightarrow{j} \field_B\end{equation*}
such that $j\circ i = {\rm id}_{\field_B}$ (here the interval $B$ is considered as a subset of $\mathsf{St}_m$). Doing this for every  death quadrant yields,
    \begin{equation*}
        M \simeq N' \oplus \bigoplus_{B\in\sdeath} \field_B.
    \end{equation*}
Iterating the above argument in the dual setting of birth quadrants in $\pyramidhigh$ gives,
    \begin{equation*}
        M \simeq N \oplus \bigoplus_{B\in\sbirth} \field_B \oplus \bigoplus_{B\in\sdeath} \field_B.
    \end{equation*}
\end{proof}

\begin{proof}[Proof of Theorem~\ref{thm:strip}]
By Lemma~\ref{lem:strip-split} we may assume that $\blockcode_1$ and $\blockcode_2$ contain no death and birth quadrants, respectively. Hence $\boxper$ is strongly exact and pfd by Lemma~\ref{lem:strip-pfd}. The result now follows from Lemma~\ref{lem:strip}. 
\end{proof}

\section{Conclusion}
\label{sec:conclusion}

In this paper we have provided a local characterization of pfd rectangle-decomposa\-ble modules in the plane. We have also shown that it is not possible to move beyond intervals of rectangular shape when considering square ``test subsets''. The following questions are natural to consider:
\begin{enumerate}
\item Allowing for test subsets of other shapes than squares, is it possible to locally characterize interval-decomposability beyond rectangles? 
\item Do the results generalize to persistence modules over $\R^n$? The injective and projective persistence modules can be characterized locally by considering the multi-graded betti numbers. Is there a more general class of interval-decomposable modules which can be locally characaterized? 
\item Is it possible to move beyond indecomposables which are pointwise 0- or 1-dimensional? Can one determine ``locally'' if $M$ decomposes into a direct sum of indecomposables belonging to a certain predefined class of indecomposables?
\item Does \Cref{thm:rec-dec-weak-exact} hold without the assumption that $X$ and $Y$ both admit a countable coinitial subset? That is indeed the case in the block-decomposable case\cite[Thm.~1.3]{botnan2018decomposition}. We state this as a conjecture
\begin{conj}
    \Cref{thm:rec-dec-weak-exact} holds verbatim for all totally ordered sets $X$ and~$Y$.
\end{conj}
\end{enumerate}

\bibliographystyle{spmpsci}
\bibliography{./mybib}

\end{document}